\documentclass[10pt,leqno]{siamltex}
\usepackage{amsfonts,amsmath,amssymb}
\usepackage{graphics,url,color,epsfig}
\usepackage[hidelinks]{hyperref}
\usepackage{stmaryrd}
\usepackage{enumitem}

\usepackage{enumerate}
\usepackage{algorithm}
\usepackage{algorithmicx}

\newcommand{\px}[1][x]{\partial_{#1}}
\newcommand{\pt}[1][t]{\partial_{#1}}
\newcommand{\dx}[1][x]{\,{\rm d}#1}

\newcommand{\mfrac}[1][2]{\frac{1}{2}}

\newtheorem{example}{Example}[section]
\newtheorem{remark}{Remark}[section]
 \allowdisplaybreaks
 \graphicspath{{figures/}}

\title{Efficient multistep methods for tempered fractional calculus: Algorithms and simulations
\thanks{This work was supported by the National Natural Science Foundation of China
(No. 11671265), the science challenge project (No. TZ2018001),
 ARC Discovery Project DP150103675, and
 the MURI/ARO on ``Fractional PDEs for Conservation Laws and Beyond: Theory, Numerics and  Applications  (W911NF-15-1-0562)''.}
}

\author{Ling Guo\thanks{Department of Mathematics, Shanghai Normal University, Shanghai, China
(lguo@shnu.edu.cn)}
\and Fanhai Zeng\thanks{School  of  Mathematical   Sciences,
Queensland   University  of  Technology,   Brisbane,   QLD 4001, Australia
(fanhaiz@foxmail.com).}
\and Ian Turner$^{\ddag,}$\thanks{Australian Research Council Centre of Excellence for Mathematical and Statistical Frontiers, Queensland University of
Technology, Brisbane, QLD 4001, Australia (i.turner@qut.edu.au).}
\and Kevin Burrage$^{\ddag,}$\thanks{Visiting Professor, Department of  Computer  Science,     University  of
Oxford, OXI 3QD, UK (kevin.burrage@qut.edu.au).}
\and George Em Karniadakis\thanks{Division of Applied Mathematics, Brown University, Providence RI, 02912 (george\_karniadakis@brown.edu).}
}

\begin{document}

\maketitle

\begin{abstract}
In this work, we extend the fractional linear multistep methods in [C. Lubich, SIAM J. Math. Anal., 17 (1986), pp.704--719] to the tempered fractional integral and derivative operators in the sense that the tempered fractional derivative operator is interpreted in terms of the Hadamard finite-part integral. We develop two fast methods, Fast Method I and Fast Method II, with linear complexity to calculate the discrete convolution for the approximation  of the (tempered) fractional operator. Fast Method I is based on a local approximation for the contour integral that represents the convolution weight. Fast Method II is based on a globally uniform approximation of the trapezoidal rule for the integral on the real line. Both methods are efficient, but numerical experimentation reveals that Fast Method II outperforms Fast Method I in terms of accuracy, efficiency, and coding simplicity. The memory requirement and computational cost of Fast Method II are $O(Q)$ and $O(Qn_T)$, respectively, where $n_T$ is the number of the final time steps and $Q$ is the number of  quadrature points used in the trapezoidal rule.  The effectiveness of the fast methods is verified through a series of numerical examples for long-time integration, including a numerical study of a fractional reaction-diffusion model.
\end{abstract}

\begin{keywords}
fractional linear multistep method, fast convolution, (tempered) fractional integral and derivative, fractional activator-inhibitor system, fractional Brusselator model.
\end{keywords}

\begin{AMS}
26A33, 65M06, 65M12, 65M15, 35R11
\end{AMS}



\section{Introduction}\label{sec1}
Fractional calculus is emerging as a powerful tool to model various physical processes involving  anomalous diffusion. Under the framework of the
continuous time random walks (CTRWs) model, the fractional Fokker-Planck and Klein-Kramers equations \cite{MetKla00} are derived
with power law waiting time distribution, assuming the particles may exhibit long waiting time. However, for practical physical processes, it is necessary to make the waiting time finite,
for example, the biological particles moving in viscous cytoplasm and displaying trapped dynamical behavior must have finite lifetime. This leads to the tempered Fokker-Planck equation corresponding to the CTRWs model with a tempered power law waiting time distribution \cite{MGW07,Gajda10}. For more applications of tempered fractional calculus and differential equations in poroelasticity, ground water hydrology and geophysical flows,  see \cite{Hanyga01,Cartea07,Meerschaert08,Meerschaert12,Meerschaert14}.

The aim of this paper is to develop   fast and memory-saving methods for discretizing the (tempered) fractional integral of the following form
\begin{equation}\label{eq:fracint}
\frac{1}{\Gamma(\alpha)}\int_{0}^t(t-s)^{\alpha-1}e^{-\sigma(t-s)}{u(s)}\dx[s],
{\quad} \alpha<1,\sigma\geq 0.
\end{equation}
If $\alpha<0$, the above integral is interpreted in terms of the Hadamard finite-part integral,
which is
equivalent to the  (tempered) fractional derivative of order $-\alpha$, see Lemma \ref{lem3-1}.

When $\sigma=0$, \eqref{eq:fracint} reduces to the Riemann--Liouville (RL) fractional integral of
order $\alpha$ ($\alpha>0$) or the RL fractional derivative of order $-\alpha$ ($\alpha<0$).
Thus, the method developed in the present paper is a general framework for (tempered) fractional calculus.
Therefore, we will mainly focus on the fast computation of the tempered fractional integral \eqref{eq:fracint} for $\alpha<0$.
Recently, some numerical methods have been developed to solve the tempered fractional differential equations via finite difference methods, see \cite{Cartea07,Dengzhaowu,Lidengzhao15,Marom09}.
However, fast and memory-saving methods for tempered fractional differential equations
are limited.

In this paper, we extend Lubich's fractional linear multistep methods (FLMMs) (see \cite{Lub86}) to discretize the tempered fractional integral and derivative operators, which yields the discrete convolution as
\begin{equation}\label{eq:frac-convolution}
\tau^{-\alpha}\sum_{j=0}^n\omega^{(\alpha,\sigma)}_{n-k}u_k, \quad 0\leq n\leq n_T,
\end{equation}
where $\tau$ is the time step size, $n_T$ is a positive integer, $\alpha$ is real,
$\sigma\geq 0$, $\omega^{(\alpha,\sigma)}_{k}$ are the
convolution quadrature weights, and $u_k$ can be any number; see Section
\ref{sec-3} for details.

The  discrete convolution \eqref{eq:frac-convolution} requires $O(n_T)$ active memory and $O(n_T^2)$ arithmetic  operations by direct  computation. Thus, the direct calculation of \eqref{eq:frac-convolution} becomes computationally expensive when it is applied to discretize  time-fractional partial differential equations (PDEs). Recently, some progress has been made to reduce the memory requirement and computational cost of the discrete convolution for approximating
the RL fractional operators \cite{BanjaiLopez18,JiangZZZ16,JingLi10,LopLubSch08,McLean12,YuPK16,ZengTBK2018}.
For the fast methods based on piecewise polynomial interpolation, the kernel function in the
fractional operators is approximated by the sum-of-exponentials, that is to say,
the quadrature weights $\omega^{(\alpha,\sigma)}_{n}$ in \eqref{eq:frac-convolution} originate from interpolation; see \cite{JiangZZZ16,JingLi10,YuPK16,ZengTB2018}.

In this work, we develop two fast methods for calculating \eqref{eq:frac-convolution}
with the quadrature weights $\omega^{(\alpha,\sigma)}_{n}$ derived from generating functions,
where the methods in  \cite{JiangZZZ16,JingLi10,YuPK16,ZengTB2018} cannot apply here.
The basic idea is to
re-express the weight $\omega^{(\alpha,\sigma)}_{n}$ as an integral form.
In the \emph{first method}, we express $\omega^{(\alpha,\sigma)}_{n}$ as a contour integral
of the form
\begin{equation}\label{fast-contour-int}
\omega_n^{(\alpha,\sigma)} = \frac{\tau^{1+\alpha} e^{-n\sigma\tau}}{2\pi i}\int_{\mathcal{C}}\lambda^{\alpha}(1-\lambda\tau)^{-1-n}
F_{\omega}(\lambda)\dx[\lambda],
\end{equation}
then a suitable contour quadrature (such as  Talbot, hyperbolic, or parabolic contour quadrature)
is used to discretize \eqref{fast-contour-int}. The case of $\sigma=0$ has been
investigated in \cite{BanjaiLopez18,LopLubSch08,ZengTBK2018}.
The detailed derivation of \eqref{fast-contour-int} is illustrated in  \cite{LopLubSch08,ZengTBK2018}.
In this work, we extend the method in \cite{ZengTBK2018} to the tempered fractional calculus
($\sigma>0$) to obtain Fast Method I, in which the Talbot contour quadrature
used in  \cite{ZengTBK2018} is also applied here.

The \emph{second method} is inspired by   \cite{BanjaiLopez18}, where
a Hankel contour beginning and ending in the left half of the complex plane
is applied to transform the contour integral into an integral on the half line,
which is discretized by a multi-domain Gauss quadrature, yielding a uniform approximation.
We can also choose the same  Hankel contour as in  \cite{BanjaiLopez18} to
express the quadrature weight $\omega^{(\alpha,\sigma)}_{n}$  defined by \eqref{fast-contour-int}
as an integral on the half line
\begin{equation}\label{fast-real-int}
 \omega_n^{(\alpha,\sigma)}
=  {\tau^{1+\alpha} e^{-n\sigma\tau}} \frac{\sin(\alpha \pi)}{\pi}
\int_0^{\infty}\lambda^{\alpha}(1+\lambda\tau)^{-1-n}
F_{\omega}(-\lambda)\dx[\lambda].
\end{equation}
The above integral   is further transformed
into  an integral on the real line by letting
$\lambda=\exp(x)$; see \eqref{real-int-2}.
Finally, the exponentially convergent trapezoidal rule \cite{TrefethenWeidman14}
is applied to obtain a uniform  approximation for \eqref{fast-real-int},
which leads to Fast Method II.

%


We list  the main contributions of this work as follows.
\begin{itemize}
  \item We extend the fractional linear multistep methods (FLMMs) proposed in \cite{Lub86} to both the tempered fractional integral and derivative operators, where the tempered fractional operators are interpreted in terms of the Hadamard finite-part integral,
      which significantly simplifies the results in \cite{ChenDeng15}.

  \item  We develop two new fast methods, Fast Methods I and II, to calculate the discrete convolutions to the approximation of the (tempered) fractional integral and derivative operators. Fast Method II
      outperforms Fast Method I in terms of accuracy, efficiency, and coding simplicity,
      and has the following  advantages.
      \begin{itemize}
        \item[{(a)}] The time interval is not divided into  exponentially increasing subintervals, which makes the implementation of Fast Method II much easier than Fast Method I and
            the existing fast methods in \cite{BanLopSch17,SchLopLub06,ZengTBK2018}.
        \item[{(b)}] Only real operations are performed and the recurrence  relation \eqref{real-int-7} used in Fast Method II is stable.
        \item[{(c)}] Fast  Method II also works very well for fractional orders greater than one.
        \item[{(d)}] Using the same number of quadrature points, Fast Method II achieves higher accuracy than Fast Method I.
      \end{itemize}
\end{itemize}

We emphasize that Fast Method I still works well.
The obvious disadvantage of Fast Method I is that its implementation is
more complicated than Fast Method II. We compare the two fast methods to show
the superiority of Fast Method II over Fast Method I. We focus
on the use of Fast Method II to solve fractional models
through numerical simulations.

This paper is organized as follows.
In Section \ref{sec-2}, we prove that the
tempered fractional derivative can be interpreted in terms of the
Hadamard finite-part integral. This interpretation helps us to
extend Lubich's FLMMs
to both the tempered fractional integral and derivative operators
directly, see Section \ref{sec-3}. In Section \ref{sec-4}, we propose two
fast methods  for approximating the discrete convolution in the considered  FLMM for the tempered
fractional operator, and we also make a comparison between these two methods.
Fast Method II is applied to solve tempered fractional ordinary
differential equations and a coupled system of nonlinear time-fractional
activator-inhibitor equations in Section \ref{sec:numerical} before
the conclusion is given in the last section.


\section{Preliminaries} \label{sec-2}
In this section, we  introduce  definitions of fractional integrals
and derivatives, and the properties that will be used in this paper.

\begin{definition}[RL fractional integral]
The RL  fractional integral operator $I^{\alpha}_{0,t}u(t)$
of order $\alpha \,(\alpha\geq 0)$  is defined by
\begin{equation}\label{eq:fint}
I^{\alpha}_{0,t}u(t)
=\frac{1}{\Gamma(\alpha)}\int_{0}^t(t-s)^{\alpha-1}{u(s)}\dx[s].
\end{equation}
\end{definition}

\begin{definition}[RL fractional derivative]
The  RL  fractional derivative operator ${}_{RL}D^{\alpha}_{0,t}$
of order $\alpha$  is defined by
\begin{equation} \label{RL-L}
{}_{RL}D_{0,t}^{\alpha}u(t)=\frac{1}{\Gamma(m-\alpha)}
\left[\frac{{\dx[]}^m}{\dx[t]^m}\int_{0}^{t}(t-s)^{m-\alpha-1}u(s)\dx[s]\right],
\end{equation}
where $m-1<\alpha\leq m$, $m$ is a positive integer.
\end{definition}


\begin{definition}[Tempered factional   integral]
The tempered fractional integral operator $I^{\sigma,\alpha}_{0,t}$
of order $\alpha \,(\alpha,\sigma\geq 0)$  is defined by
\begin{equation}\label{eq:sfint}\begin{aligned}
I^{\sigma,\alpha}_{0,t}u(t)
=\frac{1}{\Gamma(\alpha)}\int_{0}^t(t-s)^{\alpha-1}e^{-\sigma(t-s)}{u(s)}\dx[s].
\end{aligned}\end{equation}
\end{definition}

\begin{definition}[Tempered fractional derivative]
The tempered fractional  \\ derivative operator $D^{\sigma,\alpha}_{0,t}$
of order $\alpha>0$  is defined by
\begin{equation} \label{sRL-L}\begin{aligned}
D^{\sigma,\alpha}_{0,t}u(t)=&(\pt + \sigma)^mI^{\sigma,m-\alpha}_{0,t}u(t) \\
=&\left(\pt+\sigma\right)^m\left[\frac{1}{\Gamma(m-\alpha)}\int_{0}^t(t-s)^{m-\alpha-1}
e^{-\sigma(t-s)}{u(s)}\dx[s]\right],
\end{aligned}\end{equation}
where $\sigma\geq 0$, $\left(\pt+\sigma\right)^m
=\sum_{k=0}^m  \binom{m}{k}\pt^k\sigma^{m-k}$, $m-1<\alpha\leq m$, $m$ is a positive integer.
\end{definition}

Next, we introduce the Hadamard finite-part integral, which plays a crucial role
in the numerical approximation of the (tempered) fractional derivative operator.

\subsection{Fractional derivatives in the Hadamard sense}
In \cite{IP1999,SamKM-B93},  the RL fractional derivative operator
is proved to be equivalent to a Hadamard finite-part integral.
In this section, we extend this proof to the tempered fractional calculus.
\begin{definition}[Hadamard finite-part integral, see \cite{IP1999,SamKM-B93}]\label{DefFPI}
Let a function $f(x)$ be integrated on an interval $(\epsilon, A)$ for any $A>0$ and $0<\epsilon<A$.
The  function $f(x)$ is said to possess the Hadamard property at the point $x=0$ if there exist constants $a_k,b_0$ and $\lambda_k>0$ such that
\begin{equation} \label{fpi}
 \begin{aligned}
\int_{\epsilon}^{A}f(x)\dx[x]=\sum_{k=1}^Na_k\epsilon^{-\lambda_k}
+ b_0\ln{\frac{1}{\epsilon}} + J_0(\epsilon),
 \end{aligned}
\end{equation}
where $\lim_{\epsilon\to0}J_0(\epsilon)$ exists and is finite, which is also denoted by
\begin{equation} \label{fpi-2}
\begin{aligned}
P.V.\int_{0}^{A}f(x)\dx=\lim_{\epsilon\to0}J_0(\epsilon).
\end{aligned}
\end{equation}
\end{definition}


\begin{lemma}[see {\cite[p. 112]{SamKM-B93}}]\label{lem:1-5}
The RL fractional derivative ${}_{RL}D_{0,t}^{\alpha}u(t)$,
$\alpha>0,\alpha\neq 1,2,...$, is equivalent to the following integral in the Hadamard sense,
that is
\begin{equation} \label{h1-2}
 \begin{aligned}
{}_{RL}D_{0,t}^{\alpha}u(t)=\frac{1}{\Gamma(-\alpha)}
P.V.\int_{0}^{t}(t-s)^{-\alpha-1}u(s)\dx[s].
 \end{aligned}
\end{equation}
\end{lemma}

\begin{lemma}\label{lem3-1}
The tempered fractional derivative of order $\alpha>0$ is equivalent to the following
Hadamard finite-part integral
\begin{equation} \label{sec3:eq-1}
 \begin{aligned}
D^{\sigma,\alpha}_{0,t}u(t) =\frac{1}{\Gamma(-\alpha)}
P.V.\int_{0}^{t}(t-s)^{-\alpha-1}e^{-\sigma(t-s)}u(s)\dx[s].
 \end{aligned}
\end{equation}
\end{lemma}
\begin{proof}
Let $f(t)=e^{-\sigma t}$ and
$g(t)=\frac{1}{\Gamma(m-\alpha)}\int_{0}^t(t-s)^{m-\alpha-1} e^{\sigma{s}}{u(s)}\dx[s]$.
Then
\begin{eqnarray}
D^{\sigma,\alpha}_{0,t}u(t)&=&(\pt+\sigma)^m(fg)
=\sum_{k=0}^m\binom{m}{k}\sigma^{m-k}\pt^k(fg)\nonumber\\
&=&\sum_{k=0}^m\binom{m}{k}\sigma^{m-k}  \sum_{j=0}^k\binom{k}{j}f^{(k-j)}(t)g^{(j)}(t)\nonumber\\
&=&f(t)\sum_{k=0}^m\sum_{j=0}^k\binom{m}{k}\binom{k}{j}\sigma^{m-j} (-1)^{k-j}g^{(j)}(t)\nonumber\\
&=&f(t)\sum_{j=0}^m\sigma^{m-j}g^{(j)}(t)\sum_{k=j}^m\binom{m}{k}\binom{k}{j} (-1)^{k-j},\label{sec3:eq-2}
\end{eqnarray}
where we have used $f^{(k)}(t)=(-\sigma)^{k}e^{-\sigma t}=(-\sigma)^{k} f(t)$.

In the following, we will prove that
\begin{equation}\label{sec3:eq-3}
\sum_{k=j}^m\binom{m}{k}\binom{k}{j} (-1)^{k-j}
= \left\{\begin{aligned}
&0,{\qquad} 0\leq j \leq m-1,\\
&1,{\qquad} j = m.
\end{aligned}\right.
\end{equation}
Obviously, one has $\sum_{k=j}^m\binom{m}{k}\binom{k}{j} (-1)^{k-j}=1$  for $j=m$.
For $0\leq j \leq m-1$, we have
\begin{equation}
\begin{aligned}
\sum_{k=j}^m\binom{m}{k}\binom{k}{j}(-1)^{k-j}
=& \sum_{k=j}^m(-1)^{k-j}\frac{m!}{k!(m-k)!}\frac{k!}{(k-j)!j!}\\
=&\frac{m(m-1)\cdots (m-j+1)}{j!}\sum_{k=j}^m\frac{(m-j)!(-1)^{k-j}}{(m-k)!(k-j)!}  \\
=&\frac{m(m-1)\cdots (m-j+1)}{j!}(1-1)^{m-j}=0,{\quad}j<m.
\end{aligned}
\end{equation}

Combining \eqref{sec3:eq-2}  and \eqref{sec3:eq-3} yields
\begin{equation}
 \begin{aligned}
D^{\sigma,\alpha}_{0,t}u(t)
=&f(t)g^{(m)}(t)
=e^{-\sigma t}
\frac{\dx[]^m}{\dx[t]^m}\left[\frac{1}{\Gamma(m-\alpha)}\int_{0}^t(t-s)^{m-\alpha-1} e^{\sigma{s}}{u(s)}\dx[s]\right]\\
=&\frac{e^{-\sigma t}}{\Gamma(-\alpha)}P.V.\int_{0}^t(t-s)^{-\alpha-1} e^{\sigma{s}}{u(s)}\dx[s],
 \end{aligned}
\end{equation}
where Lemma \ref{lem:1-5} is applied. The proof is complete.
\end{proof}


\section{Fractional linear multistep methods (FLMMs)}\label{sec-3}
In this section, we extend Lubich's FLMMs (see \cite{Lub86}) to discretize the tempered fractional integral and derivative
operators. For convenience, we introduce the following notation:
\begin{equation}\label{FLMM}
D_{\tau}^{\alpha,\sigma,\gamma,m,n} u = \tau^{-\alpha}\sum_{k=0}^n\omega^{(\alpha,\sigma)}_{n-k}(u(t_k)-u_0) +
 \tau^{-\alpha}\sum_{k=1}^mw^{(\alpha,\sigma)}_{n,k}(u(t_k)-u_0) ,
\end{equation}
where $\tau$ is the step size, $t_k=k\tau$ is the grid point, $\gamma=(\gamma_1,\gamma_2,\cdots)$,
$\gamma_{j+1}>\gamma_j>0$, and the quadrature
weights $\omega^{(\alpha,\sigma)}_k$  are chosen such that
$D_{\tau}^{\alpha,\sigma,\gamma,m,n} u$ is a stable approximation of
$\left[D^{\sigma,\alpha}_{0,t}(u(t)-u_0)\right]_{t=t_n}$.
When the quadrature weights $\omega^{(\alpha,\sigma)}_k$ are determined, the  starting weights
$w^{(\alpha,\sigma)}_{n,k}(1\leq k \leq m)$ will be chosen such that
$D_{\tau}^{\alpha,\sigma,\gamma,m,n} u=\left[D^{\sigma,\alpha}_{0,t}u(t)\right]_{t=t_n}$ for some
$u(t)=t^{\gamma_j}$, $1\leq j\leq m$.

The convolution quadrature weights $\omega^{(\alpha,\sigma)}_k$ in \eqref{FLMM} can be given by the following generating functions; see \cite{Lub86}.
\begin{itemize}
  \item The  fractional backward difference formula of order $p$ (FBDF-$p$):
  \begin{equation}\label{FBDF-S}
\omega^{(\alpha,\sigma)}(z) =\left(\sum_{k=1}^p\frac{1}{k}(1-ze^{-\sigma\tau})^k\right)^{\alpha}
=\sum_{k=0}^{\infty}\omega^{(\alpha,\sigma)}_{k}z^k.
\end{equation}

  \item The generalized Newton--Gregory formula of order $p$ (GNGF-$p$)
\begin{equation}\label{GNGF-S}
\omega^{(\alpha,\sigma)}(z) =\left({1-ze^{-\sigma\tau}}\right)^{\alpha}
\sum_{k=1}^p g_{k-1}(1-ze^{-\sigma\tau})^{k-1}
=\sum_{k=0}^{\infty}\omega^{(\alpha,\sigma)}_{k}z^k,
\end{equation}
where $g_0=1$ and $g_1= {\alpha}/{2}$; see  \cite{GalGar08} for $g_k(k=2,3,4,5)$.
  \item The fractional trapezoidal rule
\begin{equation}\label{FTRAP-S}
\omega^{(\alpha,\sigma)}(z)
=\left(\frac{(1-ze^{-\sigma\tau})}{2(1+ze^{-\sigma\tau})}\right)^{\alpha}
=\sum_{k=0}^{\infty}\omega^{(\alpha,\sigma)}_{k}z^k.
\end{equation}
\item  See \eqref{FLMM-S} for other choices of the coefficients $\omega^{(\alpha,\sigma)}_k$.
\end{itemize}

Under suitable conditions, \eqref{FLMM} is a $p$-th order approximation of
$D^{\sigma,\alpha}_{0,t}(u(t)-u_0)$
if the generating function \eqref{FBDF-S} or
\eqref{GNGF-S} is used, and a second-order approximation is derived if
\eqref{FTRAP-S} is applied.

From \cite{Lub86}, we immediately derive the following two theorems.
\begin{theorem}
Let $\alpha\in \mathbb{R}$, $\delta > 0$. Then
for $u(t)=t^{\delta}$, one has
\begin{equation}\label{FLMM-2}
 \left[D^{\sigma,\alpha}_{0,t}u(t)\right]_{t=t_n} =D_{\tau}^{\alpha,\sigma,\gamma,0,n} u
+ O(t_n^{\alpha+\delta-p}\tau^{p}) +  O(t_n^{\alpha-1}\tau^{\delta+1}).
\end{equation}
\end{theorem}

\begin{theorem} \label{thm:3-2}
Let $(\hat{\rho}, \hat{\sigma})$ denote an implicit
linear multistep method (LMM) which is stable and consistent of order $p$,
i.e., $\hat{\rho}(z)$ and $\hat{\sigma}(z)$ are the characteristic polynomials
of the LMM of order $p$ for the first-order ordinary differential equation.
Assume that the zeros of $\hat{\sigma}(z)$ have absolute value less than 1.
Let   $\omega(z)=\frac{\hat{\sigma}(1/z)}{\hat{\rho}(1/z)}$ and
\begin{equation}\label{FLMM-S}
\omega^{(\alpha,\sigma)}(z)=(\omega(z e^{-\sigma\tau}))^{\alpha}
=\sum_{k=0}^{\infty}\omega^{(\alpha,\sigma)}_kz^k.
\end{equation}
Then, we have
\begin{equation}
\left[D^{\sigma,\alpha}_{0,t}(u(t)-u(0))\right]_{t=t_n} =D_{\tau}^{\alpha,\sigma,\gamma,m,n} u
+O(\tau^{p}).
\end{equation}

\end{theorem}

Next, we   discuss   how to implement the fast computation of the convolution quadrature coefficients
$\omega^{(\alpha,\sigma)}_{k}$ defined by \eqref{FBDF-S}, \eqref{GNGF-S}, and \eqref{FTRAP-S}.
In fact, we need only to consider how to derive $\omega^{(\alpha,\sigma)}_{k}$ defined by
\eqref{FBDF-S}, since $\omega^{(\alpha,\sigma)}_{k}$  given in
\eqref{GNGF-S}  and \eqref{FTRAP-S} can be derived from
the coefficients given in \eqref{FBDF-S} for $p=1$.
For the FBDF-$p$ given in  \eqref{FBDF-S}, the coefficients satisfy  $\omega^{(\alpha,\sigma)}_{k}=e^{-k\tau\sigma}\omega^{(\alpha,0)}_{k}$,
where $\omega^{(\alpha,0)}_{k}$ can be efficiently calculated by the
recurrence formula; see, e.g., \cite{DieFFW06,LiZeng13}.



\section{Fast calculation}\label{sec-4}
In this section, we present fast calculations for the convolution
$\tau^{-\alpha}\sum_{k=0}^n\omega^{(\alpha,\sigma)}_{n-k}u(t_k) $ defined in \eqref{FLMM-2}, where the coefficients $\omega^{(\alpha,\sigma)}_{k}$
can be derived from  \eqref{FBDF-S}, \eqref{GNGF-S},   \eqref{FTRAP-S}, or \eqref{FLMM-S}.
The key idea is to represent the
coefficients $\omega^{(\alpha,\sigma)}_{k}$ using the integral formula and then approximate it using numerical quadrature.
We first extend the fast method in \cite{ZengTBK2018} to calculate
the discrete convolution $\tau^{-\alpha}\sum_{k=0}^n\omega^{(\alpha,\sigma)}_{n-k}u(t_k)$ in \eqref{FLMM},
which is called Fast Method I in the following context.
Then we propose the second fast method based on the approaches in \cite{BanjaiLopez18,LubSch02,ZengTBK2018}, which is called  Fast Method II.

\subsection{Fast Method I}
Following the approach developed in \cite{ZengTBK2018}, the convolution quadrature weights $\omega_n^{(\alpha,\sigma)}$ in \eqref{FLMM-2} can be expressed as
\begin{equation}\label{contour-int}
\omega_n^{(\alpha,\sigma)} =e^{-n\sigma\tau}\omega_n^{(\alpha,0)}
= \frac{\tau^{1+\alpha} e^{-n\sigma\tau}}{2\pi i}\int_{\mathcal{C}_{\ell}}\lambda^{\alpha}(1-\lambda\tau)^{-1-n}
F_{\omega}(\lambda)\dx[\lambda],
\end{equation}
where $\mathcal{C}_{\ell}$ is
a contour that surrounds the poles  of $(1-\lambda\tau)^{-1-n}$
and $F_{\omega}(\lambda)$ (see also   (38) in \cite{ZengTBK2018})
is related to the FLMM   \eqref{FLMM} defined by
the generating functions, which is given by
\begin{equation}\label{contour-int-F}
F_{\omega}(\lambda)=(\tau\lambda)^{-\alpha}\omega^{(\alpha,0)}(1-\tau\lambda).
\end{equation}
where  $\omega^{(\alpha,0)}(z)$ is defined by  \eqref{FBDF-S}, \eqref{GNGF-S},  \eqref{FTRAP-S}, or \eqref{FLMM-S}

To approximate the contour integral \eqref{contour-int} with high accuracy, we apply the
trapezoidal rule based on the   Talbot contour
(see, e.g., \cite{LopLubSch08,ZengTBK2018}) to obtain
\begin{equation}\label{contour-int-2}
\omega_n^{(\alpha,\sigma)}
\approx\widetilde{\omega}_n^{(\alpha,\sigma)}
= 2\tau^{1+\alpha} e^{-n\sigma\tau} \mathbf{Im}
\left(\sum_{j=0}^{N-1}w_j^{(\ell)}(\lambda_j^{(\ell)})^{\alpha}(1-\lambda^{(\ell)}_j\tau)^{-1-n}
F_{\omega}(\lambda^{(\ell)}_j) \right),
\end{equation}
where the  quadrature points $\lambda^{(\ell)}_j$ and weights $w_j^{(\ell)}$ are given by
(see, e.g., \cite{Weideman06,ZengTBK2018})
\begin{equation}\label{Talbot}
{\lambda}_j^{(\ell)} = z(\theta_j,N/T_{\ell}),{\quad}
{w}_j^{(\ell)}=\px[\theta]z(\theta_j,N/T_{\ell}), {\quad}\theta_j={(2j+1)\pi}/({2N}),
\end{equation}
with $z(\theta,N)=N\left(-0.4814 +0.6443(\theta\cot(\theta)+i0.5653\theta)\right)$,
$T_{\ell}=(2B^{\ell}-2+n_0)\tau$,  $B>1$ is a positive integer.

According to the procedure in \cite{LubSch02,SchLopLub06,ZengTBK2018}, we need to first find the smallest integer $L$ satisfying $n-n_0+1\le 2B^L$ for each $n\ge n_0$. Then for $\ell=1,2,...L$,
we obtain a unique integer $q_{\ell}$ satisfying
\begin{equation}\label{fast1-b}
b_\ell^{(n)}=q_{\ell}B^l \quad \text{with} \quad n-n_0+1-b_\ell^{(n)}\in [B^{\ell-1},2B^\ell-1].
\end{equation}
Set $b_0^{(n)}=n-n_0$ and $b_L^{n}=0$. Readers can refer to \cite{SchLopLub06}
for the  pseudocode for determining $q_{\ell}$ and  $b_\ell^{(n)}$.

To develop the fast method, the convolution $u^{(\alpha,\sigma)}_n=\tau^{-\alpha}\sum_{k=0}^{n}\omega^{(\alpha,\sigma)}_{n-k}u_{k}$
is decomposed as
\begin{equation}\label{fast1-decomposition}
u^{(\alpha,\sigma)}_n=\tau^{-\alpha}\sum_{k=0}^{n}\omega^{(\alpha,\sigma)}_{n-k}u_{k}
=\sum_{\ell=0}^{n}u_n^{(\ell)},
\end{equation}
with $u_n^{(0)}=\sum_{k=n-n_0}^{n}\omega^{(\alpha,\sigma)}_{n-k}u_{k}$ and $u_n^{(\ell)}=\sum_{k=b_{\ell}^{(n)}}^{b_{\ell-1}^{(n)}-1}\omega^{(\alpha,\sigma)}_{n-k}u_{k}$. For each part $u_n^{(\ell)}$, we can use \eqref{contour-int-2} to approximate the corresponding quadrature weights. The summary of Fast Method I is given in Algorithm \ref{alg:fc1}.

\begin{algorithm}
\caption{Fast Method I for approximating $u^{(\alpha,\sigma)}_n=\tau^{-\alpha}\sum_{k=0}^{n}\omega^{(\alpha,\sigma)}_{n-k}u_{k}$, where $\omega^{(\alpha,\sigma)}_{k}$ satisfies \eqref{contour-int}.}\label{alg:fc1}
\begin{algorithmic}[1]
\State Input: the fractional order $\alpha$ and $\sigma\geq 0$, a time stepsize $\tau>0$,
  the suitable positive integers $n_0$, $N$, and $B\geq 2$,
  the quadrature points $\lambda_j^{(\ell)}$ and weights $w_j^{(\ell)}$  (see \eqref{real-int-3}), the coefficients $\omega^{(\alpha,\sigma)}_{n}(0\leq n \leq n_0)$ defined by \eqref{FBDF-S}, \eqref{GNGF-S},   \eqref{FTRAP-S}, or \eqref{FLMM-S},
  and the function $F_{\omega}(\lambda)$ defined by \eqref{contour-int-F}.
\State Output: the fast approximation ${}_Fu^{(\alpha,\sigma)}_n$ of $u^{(\alpha,\sigma)}_n$
(see \eqref{fast1-decomposition}).

\begin{itemize}
  \item Step 1. Find the smallest integer $L$ satisfying $n-n_0+1\le 2B^L$ for each $n\ge n_0$.
  \item Step 2. Determine $q_{\ell}$ according to \eqref{fast1-b} for $\ell=1,2,...,L-1$.
  \item Step 3. For every $1\leq \ell\leq L$, approximate $u_n^{(\ell)}$ by
  \begin{equation}\label{fast1-app}
 \widehat{u}_n^{(\ell)}= 2\mathbf{Im}\bigg\{\sum_{j=0}^{N-1}w_j^{(\ell)}F_\omega(\lambda_j^{(\ell)})
 (1-\tau\lambda_j^{(\ell)})^{-[n-b_{\ell-1}^{(n)}-1]}y^{(\ell)}(\tau\lambda_j^{(\ell)})\bigg\},
\end{equation}
where
$y^{(\ell)}(\tau\lambda)=y^{(\ell)}(b^{(n)}_{\ell-1}\tau,b^{(n)}_{\ell}\tau,\tau\lambda)$ is
the backward Euler approximation to the solution at $t=b^{(n)}_{\ell-1}\tau$
of the linear initial-value problem
\begin{equation}\label{ode2}
(y^{(\ell)})'(t)=\lambda y^{(\ell)}(t) + u(t),{\quad}y^{(\ell)}(b^{(n)}_{\ell}\tau)=0.
\end{equation}

  \item Step 4. Calculate
  \begin{equation}\label{fast1}
{}_Fu^{(\alpha,\sigma)}_n=u_n^{(0)}+\widehat{u}_n^{(1)}+\cdot\cdot\cdot+\widehat{u}_n^{(L)}.
\end{equation}

\end{itemize}
\end{algorithmic}
\end{algorithm}

\begin{remark}
Here we use the Talbot contour quadrature to approximate the contour integral
\eqref{contour-int}. However, other contour quadratures can be used to discretize \eqref{contour-int},
such as  the hyperbolic and parabolic contour quadratures.
For more details,   see \cite{BanLopSch17,LopLubPS2005,SchLopLub06,ZengTBK2018} and references therein.
\end{remark}

\begin{remark}
It is shown in \cite{ZengTBK2018} that the memory requirement and computational cost of Fast Method I are about $O(N\log n_T)$ and $O(Nn_T\log n_T)$, respectively,
when $n_T$ is sufficiently large.
\end{remark}

\subsection{Fast Method II}
Instead of using   \eqref{contour-int} for expressing
the convolution weights $\omega_n^{(\alpha,\sigma)}$, we extend Lemma 9 in \cite{BanjaiLopez18} to
re-express the contour integral \eqref{contour-int} into the following form
\begin{equation}\label{real-int}
 \omega_n^{(\alpha,\sigma)}
=  {\tau^{1+\alpha} e^{-n\sigma\tau}} \frac{\sin(\alpha \pi)}{\pi}
\int_0^{\infty}\lambda^{\alpha}(1+\lambda\tau)^{-1-n}
F_{\omega}(-\lambda)\dx[\lambda],
\end{equation}
where $F_{\omega}$ is given by \eqref{contour-int-F}.
The key point is how to approximate \eqref{real-int} efficiently and accurately.
Here we follow the idea in \cite{McLean2016} and let $\lambda= \exp(x)$.
Then the integral \eqref{real-int}  becomes
\begin{equation}\label{real-int-2}
 \omega_n^{(\alpha,\sigma)}
=  \tau^{1+\alpha} e^{-n\sigma\tau}
\int_{-\infty}^{\infty} \phi_n(x)    \dx[x],
\end{equation}
where
\begin{equation}\label{phix}
\phi_n(x)= (1+e^x\tau)^{-1-n}\phi(x),{\qquad}
\phi(x)= -\frac{\sin(\alpha \pi)}{\pi}e^{(1+\alpha)x}F_{\omega}(-e^x).
\end{equation}
We find that $\phi_n(x)$ decays exponentially as  {$|x|\to \infty$} for any $n>n_0$,
where $n_0$ is a suitable positive integer.
Figure \ref{eg51fig0} shows the exponential decay of $\phi_n(x)$ for $\alpha=0.2$ and $0.8$
when the second-order GNGF \eqref{GNGF-S} is applied, $n_0=50$.
For the GNGF-$p$ and FBDF-$p$, and
any fractional order $\alpha>0$, the corresponding $\phi_n(x)$ decays exponentially
for $n>n_0$ as {$|x|\to \infty$}
but these results are not shown here.

\begin{figure}[!h]
\begin{center}
\begin{minipage}{0.47\textwidth}\centering
\epsfig{figure=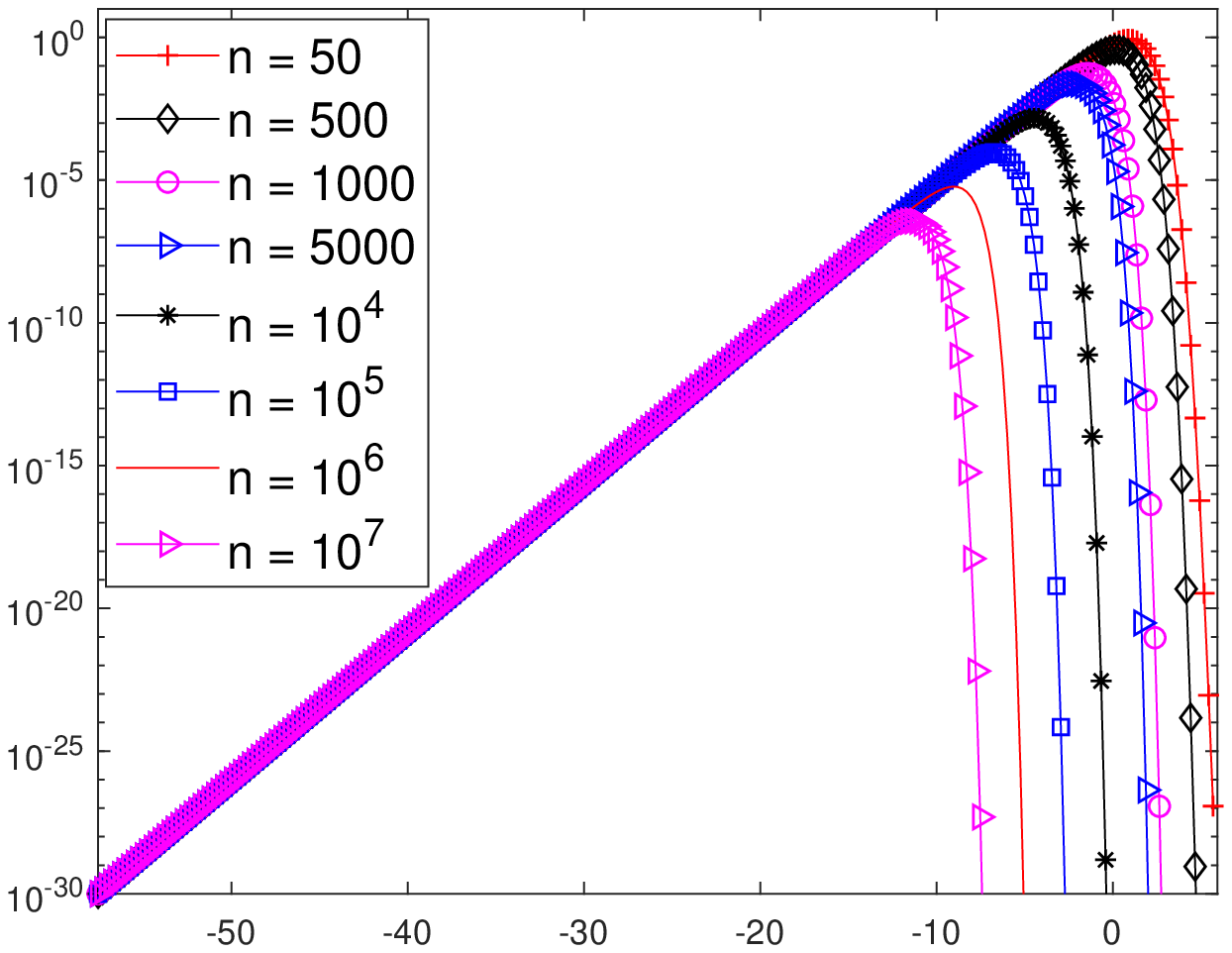,width=5.3cm} \par {(a) $\alpha=0.2$.}
\end{minipage}
\begin{minipage}{0.47\textwidth}\centering
\epsfig{figure=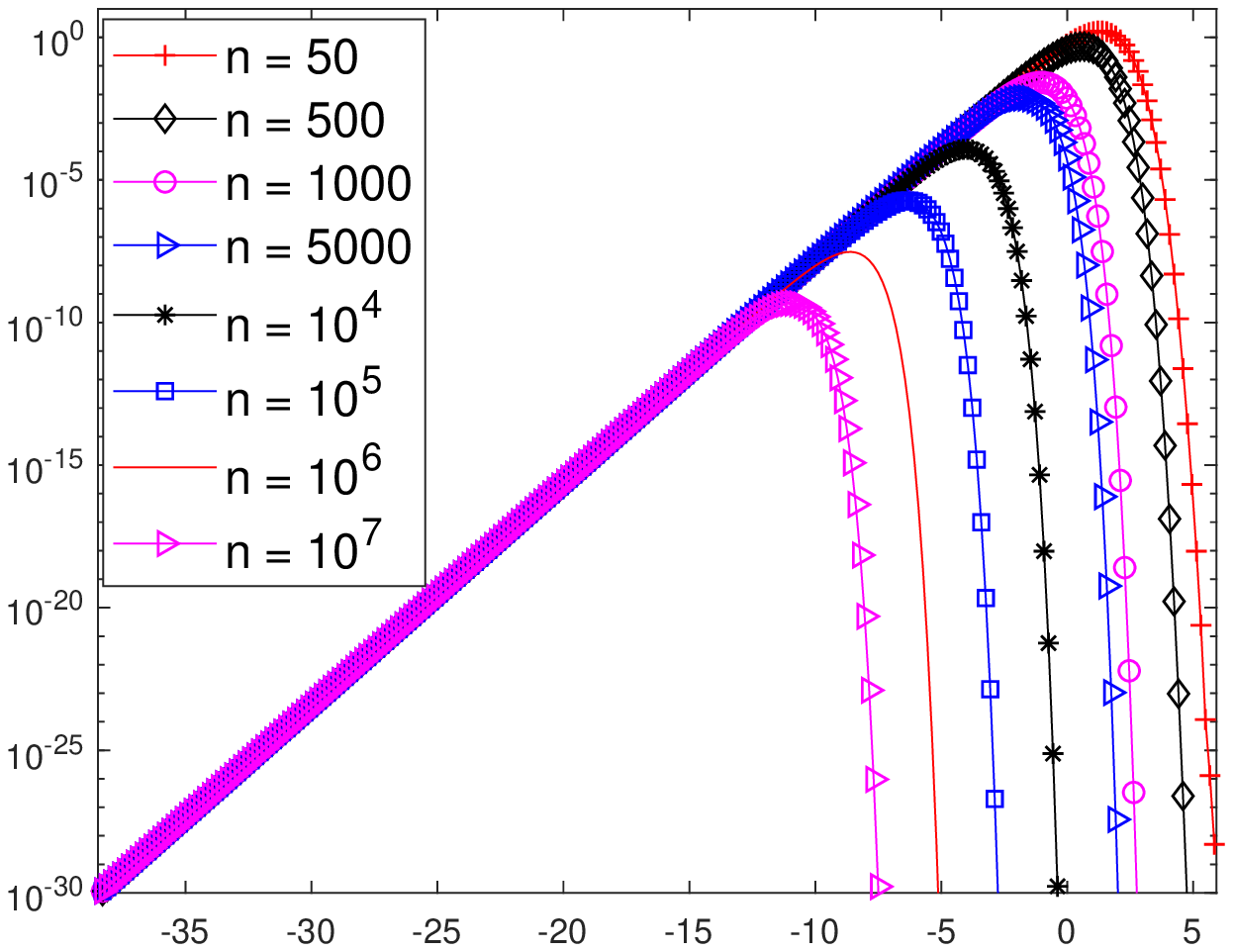,width=5.3cm} \par {(b) $\alpha=0.8$.}
\end{minipage}
\end{center}
\caption{The exponential decay of $\phi_n(x)$ for second-order GNGF, $\tau=0.01$.\label{eg51fig0}}
\end{figure}

The exponential decay of $\phi_n(x)$ inspires us to use the exponentially convergent trapezoidal rule
(see \cite{TrefethenWeidman14}) to approximate the integral $\int_{-\infty}^{\infty}\phi_n(x)\dx$.
Thus, we have
\begin{equation}\label{real-int-3-2}\begin{aligned}
\omega_n^{(\alpha,\sigma)}\approx  \widehat{\omega}_n^{(\alpha,\sigma)}
=&{\tau^{1+\alpha} e^{-n\sigma\tau}}
\Delta{x}\sum_{j=-\infty}^{\infty} (1+e^{x_j}\tau)^{-1-n} \phi(x_j)\\
=&{\tau^{1+\alpha} e^{-n\sigma\tau}} \sum_{j=-\infty}^{\infty}w_j(1+\lambda_j\tau)^{-1-n},
\end{aligned}\end{equation}
where $\lambda_j=e^{x_j}$, $x_j=j\Delta{x}$, $\Delta{x}>0$ is a positive number,
$w_j=\Delta{x}\phi(x_j)$.

We have the following theorem for the error of   \eqref{real-int-3-2}, see \cite{TrefethenWeidman14}.
\begin{theorem}[\cite{TrefethenWeidman14}]
Suppose $\phi_n(x)$ is analytic in the strip $\big |Im(x)\big | < a$ for some $a > 0$.
Suppose further that $\phi_n(x)\to0$ uniformly as $|x|\to\infty$ in the strip, and for some $M$,
$\int_{-\infty}^{\infty}|\phi_n(x+ib)| \dx[x]\leq M$
for all $b\in(-a,a)$. Then, for any $\Delta{x}>0$,
$\widehat{\omega}_n^{(\alpha,\sigma)}$ as defined by
\eqref{real-int-3-2} exists and satisfies
\begin{equation}\label{real-int-4}
|\omega_n^{(\alpha,\sigma)}-\widehat{\omega}_n^{(\alpha,\sigma)}|
\leq \tau^{1+\alpha} \pi^{-1}|\sin(\alpha\pi)| \frac{2M}{e^{{2\pi a}/{\Delta{x}}}-1}.
\end{equation}
\end{theorem}

In real applications, we do not use \eqref{real-int-3-2}. Instead, we truncate  \eqref{real-int-3-2}
and derive the following modified version, which is used in this paper
\begin{equation}\label{real-int-3}\begin{aligned}
\widehat{\omega}_n^{(\alpha,\sigma)}
=&{\tau^{1+\alpha} e^{-n\sigma\tau}} \sum_{j=0}^{Q-1}w_j(1+\lambda_j\tau)^{-1-n},
\end{aligned}\end{equation}
where $w_j$ and $\lambda_j$ are determined according  to the property of $\phi_n(x)$
(see Figure \ref{eg51fig0}), which may be different from  those used in \eqref{real-int-3-2}.

We illustrate how to derive $w_j$ and $\lambda_j$ in \eqref{real-int-3}.
Denote
$\phi_{n,\max}=\max\limits_{|x|<\infty}\phi_n(x)$ and
$S_n(x) =\{x|{\phi_{n}(x)}/{\phi_{n,\max}}\geq\epsilon,x\in\mathbb{R}\}$, where $\epsilon>0$ is given (we set
$\epsilon=10^{-20}$ in this paper).
Find $x_{\min}^n=\min\limits_{x} \{S_n(x)\}$ and $x_{\max}^n=\max\limits_{x} \{S_n(x)\}$.
Denote $x_{\min}=\min\{x_{\min}^{n_0},x_{\min}^{n_T}\}$ and $x_{\max}=\max\{x_{\max}^{n_0},x_{\max}^{n_T}\}$.
Given a positive integer $Q$, let $\Delta{x}=(x_{\max}-x_{\min})/(Q-1)$ and $x_j=x_{\min}+j\Delta{x}$.
Then, $w_j$ and $\lambda_j$ are determined by $w_j=\Delta{x}\phi(x_j)$ and $\lambda_j=\exp(x_j)$.

Based on \eqref{real-int-3},  we   give a detailed implementation of Fast Method II. We first decompose the discrete convolution $u^{(\alpha,\sigma)}_n=\tau^{-\alpha}\sum_{j=0}^{n}\omega^{(\alpha,\sigma)}_{n-j}u_{j}$ into
      \begin{equation}\label{real-int-5}
u^{(\alpha,\sigma)}_n={}_Lu_{n,n_0}^{(\alpha,\sigma)}+{}_Hu_{n,n_0}^{(\alpha,\sigma)}
 \equiv\tau^{-\alpha}\sum_{k=n-n_0}^{n}\omega^{(\alpha,\sigma)}_{n-k}u_{k}
 +\tau^{-\alpha}\sum_{k=0}^{n-n_0-1}\omega^{(\alpha,\sigma)}_{n-k}u_{k}.
\end{equation}
Then, the local part ${}_Lu_{n,n_0}^{(\alpha,\sigma)}$ is calculated directly. In the following, we give a simple illustration on how to obtain ${}_Hu_{n,n_0}^{(\alpha,\sigma)}$.
Inserting $\omega_n^{(\alpha,\sigma)}$ defined by \eqref{real-int-2} into
${}_Hu_{n,n_0}^{(\alpha,\sigma)}$, we obtain
\begin{equation}\label{real-int-8}
\begin{aligned}
{}_Hu_{n,n_0}^{(\alpha,\sigma)}
      =&\tau\sum_{k=0}^{n-n_0-1}{e^{-(n-k)\sigma\tau}}u_{k}
\int_{-\infty}^{\infty} (1+e^x\tau)^{-1-(n-k)} \phi(x)dx.
\end{aligned}\end{equation}
Applying \eqref{real-int-3} to the above integral yields
\begin{equation}\label{real-int-9}
\begin{aligned}
 {}_Hu_{n,n_0}^{(\alpha,\sigma)}\approx
 &{}_H^Fu_{n,n_0}^{(\alpha,\sigma)}
 = \tau\sum_{k=0}^{n-n_0-1}u_{k}{ e^{-(n-k)\sigma}} \sum_{j=0}^{Q-1}w_j(1+\lambda_j\tau)^{-1-(n-k)}\\
=&\sum_{j=0}^{Q-1}w_j\tau\sum_{k=0}^{n-n_0-1}{e^{-(n-k)\sigma\tau}} (1+\lambda_j\tau)^{-1-(n-k)}u_{k}\\
=&e^{-n_0\tau\sigma}(1+\lambda_j\tau)^{-(n_0+1)}\sum_{j=0}^{Q-1}w_jy^{(j)}_{n-n_0},
\end{aligned}\end{equation}
where
$y^{(j)}_{n-n_0}=\tau\sum_{k=0}^{n-n_0-1} \left(e^{\sigma\tau}(1+\lambda_j\tau)\right)^{-(n-n_0-k)}u_{k}$,
which satisfies \eqref{real-int-7}.

A summary of the entire procedure of Fast Method II is given in Algorithm \ref{alg:fc2}.

We now compare the computational performance and accuracy of the proposed
fast methods against the direct method.
\begin{algorithm}
\caption{Fast calculation   of $u^{(\alpha,\sigma)}_{n}=\tau^{-\alpha}\sum_{k=0}^{n}\omega^{(\alpha,\sigma)}_{n-k}u_{k}$, where $\omega^{(\alpha,\sigma)}_{k}$ satisfies   \eqref{FBDF-S}, \eqref{GNGF-S},  \eqref{FTRAP-S}, or \eqref{FLMM-S} (see also \eqref{real-int}).}\label{alg:fc2}
\begin{algorithmic}[1]
\State Input: the fractional order $\alpha$ and $\sigma\geq 0$, a time stepsize $\tau>0$,
  a suitable positive integer $n_0$, the convolution weights
  $\omega^{(\alpha,\sigma)}_{n}(0\leq n \leq n_0)$ defined by \eqref{FBDF-S}, \eqref{GNGF-S},   \eqref{FTRAP-S}, or \eqref{FLMM-S},
  the quadrature points $\lambda_j$ and weights $w_j$  (see \eqref{real-int-3} and its following paragraph).
\State Output: ${}_Fu^{(\alpha,\sigma)}_n$.

\begin{itemize}
  \item Step 1. Approximate  the history part
   ${}_Hu_{n,n_0}^{(\alpha,\sigma)}=\tau^{-\alpha}\sum_{k=0}^{n-n_0-1}\omega^{(\alpha,\sigma)}_{n-k}u_{k}$  by
      \begin{equation}\label{real-int-6}
\begin{aligned}
{}_H^F{u}_{n,n_0}^{(\alpha,\sigma)}
      =e^{-n_0\tau\sigma}(1+\lambda_j\tau)^{-(n_0+1)\tau}\sum_{j=0}^{Q-1}w_jy^{(j)}_{n-n_0},
\end{aligned}\end{equation}
where $y^{(j)}_{n}$ is calculated by the following recurrence formula
\begin{equation}\label{real-int-7}
y^{(j)}_{n}=\frac{e^{-\tau\sigma}}{1+\lambda_j\tau}
\left(y^{(j)}_{n-1}+\tau u_{n-1}\right),
{\quad}y^{(j)}_0=0.
\end{equation}

  \item Step 2. Calculate the local part ${}_Lu_{n,n_0}^{(\alpha,\sigma)}$ directly and let
  \begin{equation}\label{real-int-7-2}
  {}_Fu^{(\alpha,\sigma)}_n={}_Lu_{n,n_0}^{(\alpha,\sigma)}+{}_H^F{u}_{n,n_0}^{(\alpha,\sigma)}.
  \end{equation}

\end{itemize}
\end{algorithmic}
\end{algorithm}

\begin{example}\label{eg21}
Let $u^{(\alpha,\sigma)}_{n}=\tau^{-\alpha}\sum_{k=0}^{n}\omega^{(\alpha,\sigma)}_{n-k}u_{k}$,
where   $\omega^{(\alpha,\sigma)}_{n}$ satisfies \eqref{GNGF-S}. Compute $u^{(\alpha,\sigma)}_n$
by the direct convolution method,  Fast Method I, and Fast Method II.
\end{example}

Define the pointwise error $e^{(r)}_n$ and the maximum pointwise error $\|e^{(r)}\|_{\infty}$ by
$$e_n^{(r)}={\big|u^{(\alpha,\sigma)}_{n}-{}_F^{(r)}u^{(\alpha,\sigma)}_{n}\big|}/
{\big|u^{(\alpha,\sigma)}_{n}\big|},{\qquad}
\|e^{(r)}\|_{\infty}=\max_{0\leq n\leq{T/\tau}} e_n^{(r)},{\quad}r=1,2,$$
where ${}_F^{(1)}u^{(\alpha,\sigma)}_{n}={}_Fu^{(\alpha,\sigma)}_{n}$ is the fast solution from Fast Method I and ${}_F^{(2)}u^{(\alpha,\sigma)}_{n}$ is the fast solution from  Fast Method II.

Figure \ref{eg51-relativeerror} shows the relative errors of
Fast Method I and Fast Method II for Example \ref{eg21}.
We can see that Fast Method II shows better accuracy
than Fast Method I when the same number of the quadrature points is used,
which means Fast Method II saves memory and computational cost
to achieve the same level of accuracy. Furthermore, Fast Method II requires only
real arithmetic operations rather than
complex arithmetic  operations as in Fast Method I, which further reduces the computational cost.

Figure \ref{eg51fig2} depicts a comparison of the computational time of the direct method and
the fast methods.
We can see that both fast methods are more efficient than the direct method for long time computation,
while Fast Method II is much faster than Fast Method I, since Fast Method II uses
real arithmetic operations instead of complex arithmetic operations in Fast Method I.

Figure \ref{eg51fig3} further illustrates why Fast Method II is more accurate than Fast Method I,
since the trapezoidal rule \eqref{real-int-3} used in Fast Method II for approximating the quadrature
weights is more accurate than the Talbot contour quadrature \eqref{contour-int-2}.
Moreover, Figure \ref{eg51fig3}
shows that both the trapezoidal rule and Talbot contour quadrature are also effective for the fractional orders greater than one and the trapezoidal rule shows more accurate approximations.

In summary, Fast Method II is more efficient than Fast Method I for any fractional orders $\alpha\in (0,2)$
and $\sigma\geq0$. For the fractional order $\alpha\in(-1,0)$ and $\alpha>2$, the same effect is still observed, but the results are not displayed here.

\begin{figure}[!h]
\begin{center}
\begin{minipage}{0.47\textwidth}\centering
\epsfig{figure=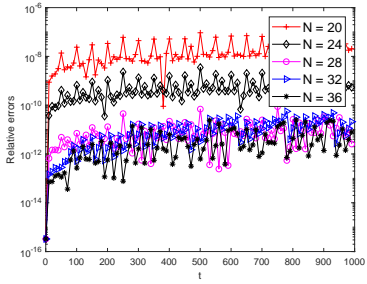,width=5.3cm} \par {(a) Fast Method I, $B=5$.}
\end{minipage}
\begin{minipage}{0.47\textwidth}\centering
\epsfig{figure=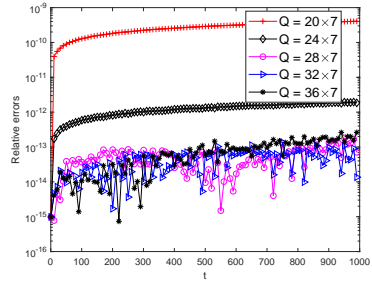,width=5.3cm} \par {(b) Fast Method II.}
\end{minipage}
\end{center}
\caption{The relative errors of Fast Method I and Fast Method II, Example \ref{eg21}:  $u(t)=t+t^2$, $\tau=0.01,\sigma=0$.
The total number of quadrature points used for Fast Method I is $N\times 7$
for $n_0=50,\tau=0.01$ and $T=1000$. The same number $Q=N\times 7$ of quadrature points
are used in Fast Method II for a fair comparison.
\label{eg51-relativeerror}}
\end{figure}

\begin{figure}[!h]
\begin{center}
\begin{minipage}{0.47\textwidth}\centering
\epsfig{figure=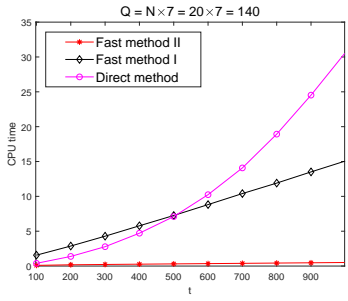,width=5.3cm} \par {(a) Computational time.}
\end{minipage}
\begin{minipage}{0.47\textwidth}\centering
\epsfig{figure=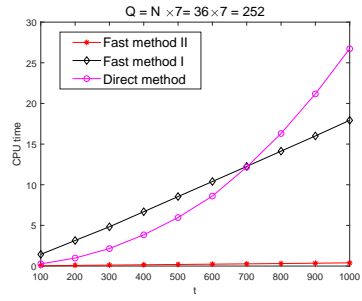,width=5.3cm} \par {(b) Computational time.}
\end{minipage}
\end{center}
\caption{The computational time of direct method, Fast Method I and Fast Method II, Example \ref{eg21}: $u(t)=t+t^2$,  $\tau=0.01,\sigma=0$. The total number of quadrature points:
(a) $Q=N\times 7 = 20\times7=140$; (b) $Q=N\times 7 = 36\times7=252$.
\label{eg51fig2}}
\end{figure}


\begin{figure}[!h]
\begin{center}
\begin{minipage}{0.47\textwidth}\centering
\epsfig{figure=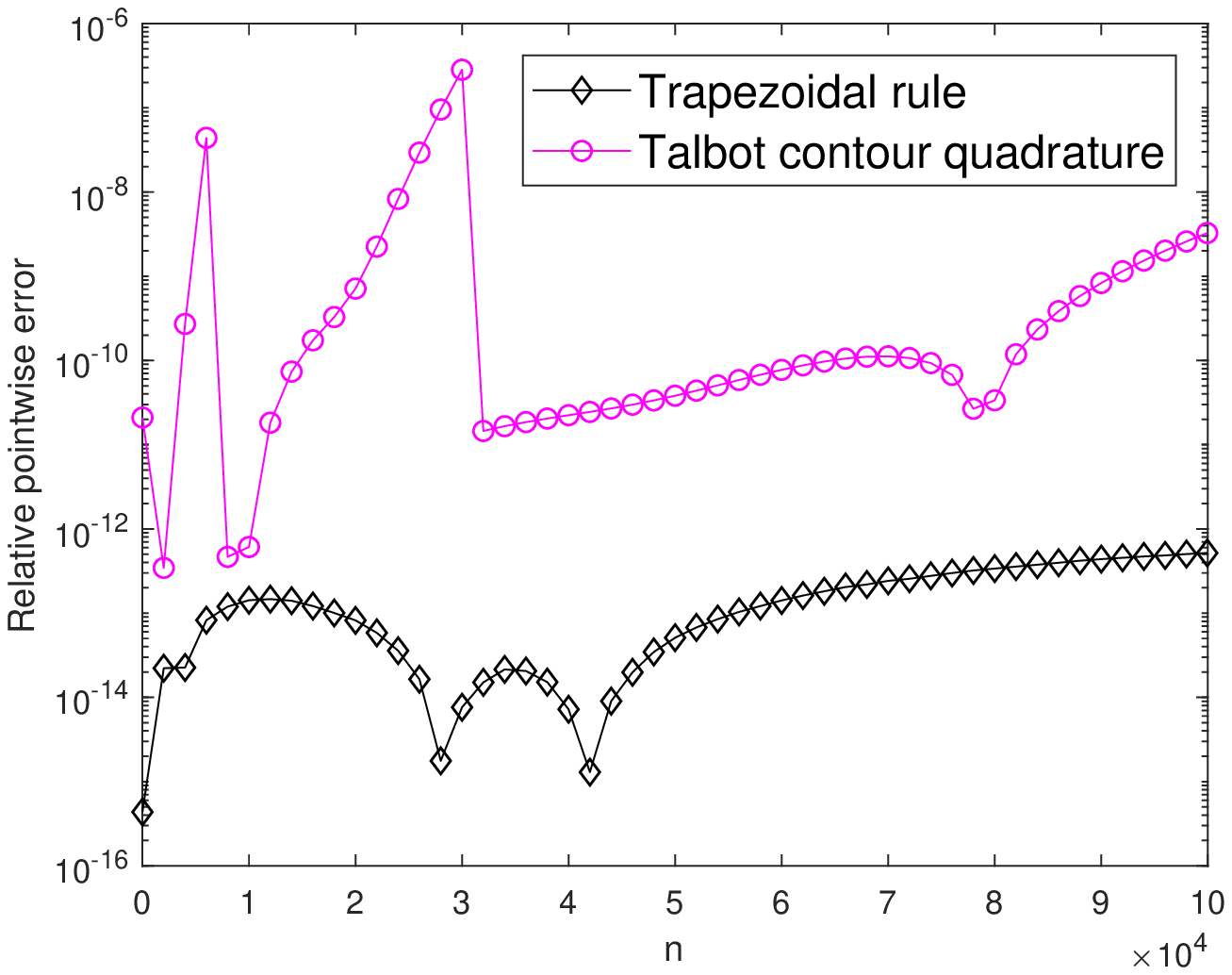,width=5.3cm} \par {(a) $\alpha=0.5,\sigma=0$. }
\end{minipage}
\begin{minipage}{0.47\textwidth}\centering
\epsfig{figure=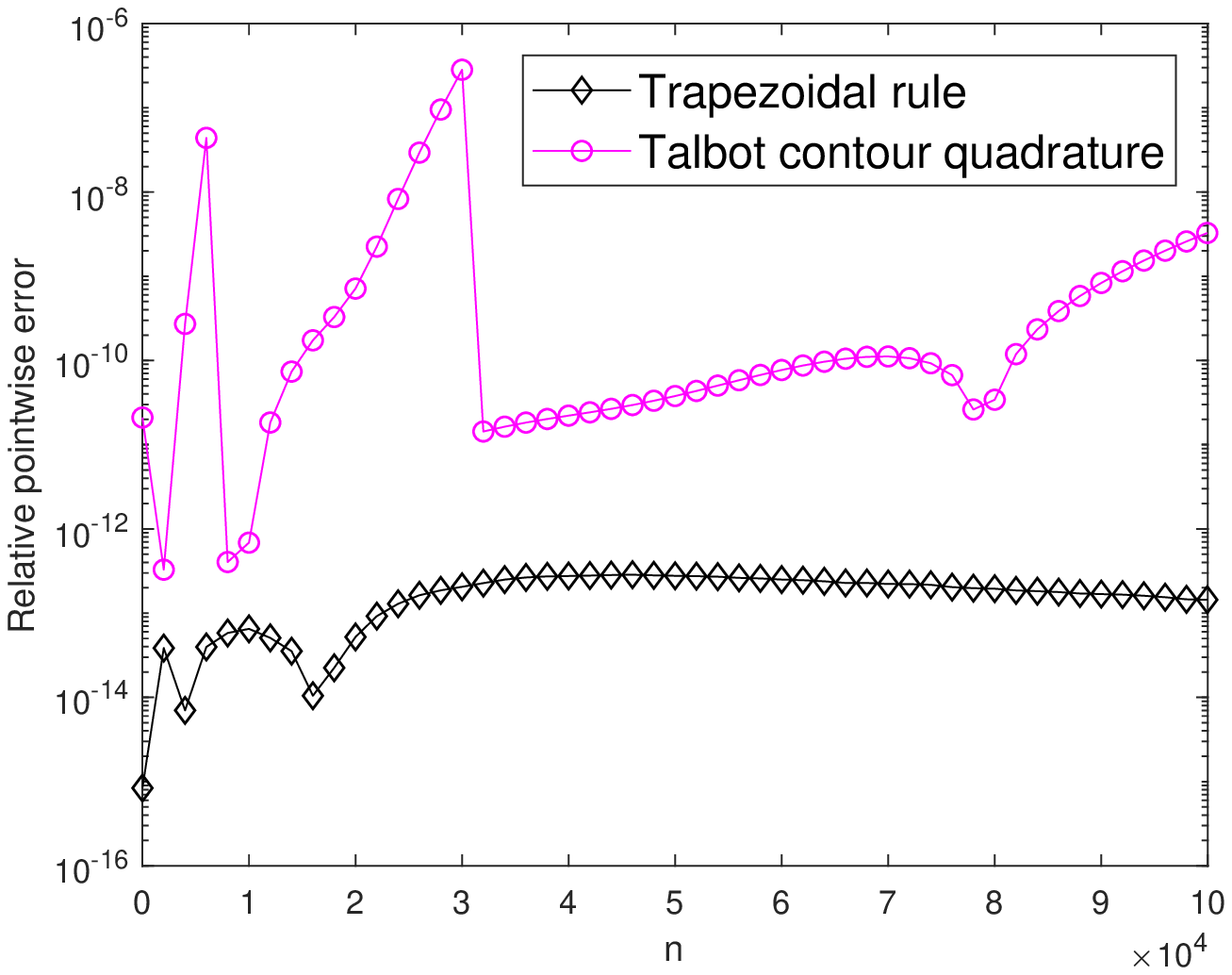,width=5.3cm}
\par {(b) $\alpha=0.5,\sigma=0.5$. }
\end{minipage}
\begin{minipage}{0.47\textwidth}\centering
\epsfig{figure=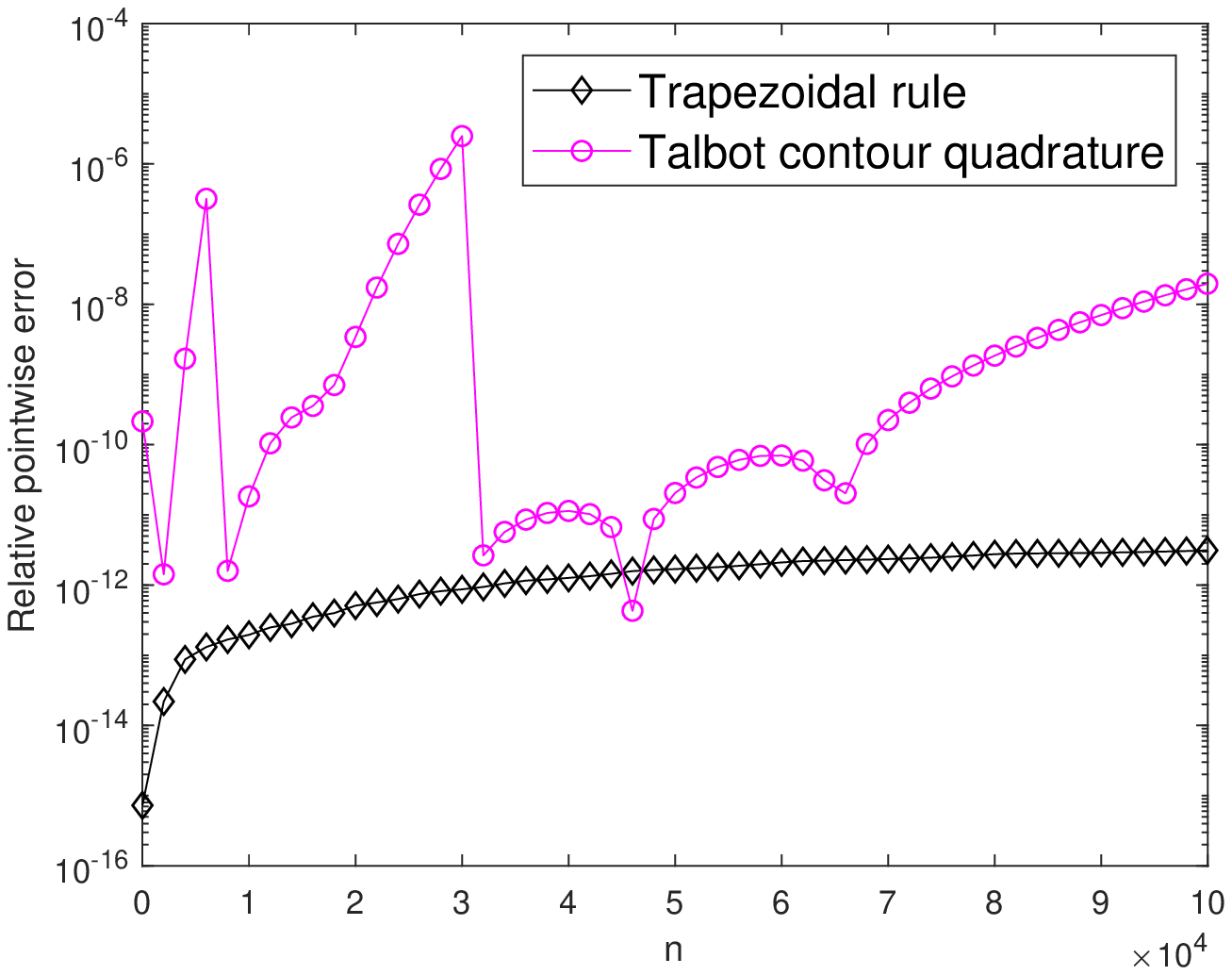,width=5.3cm}
\par {(c) $\alpha=1.5,\sigma=0$. }
\end{minipage}
\begin{minipage}{0.47\textwidth}\centering
\epsfig{figure=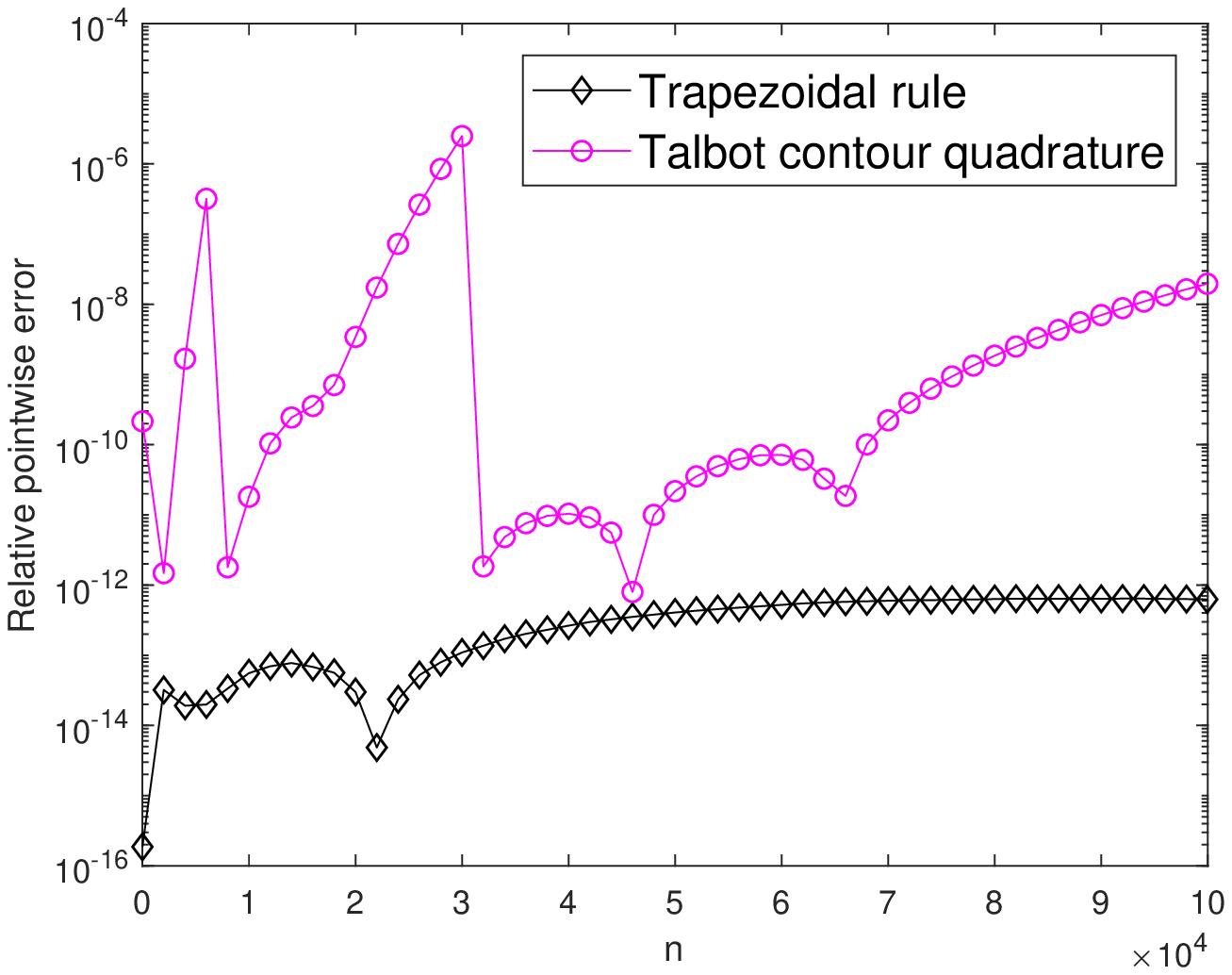,width=5.3cm}
\par {(d) $\alpha=1.5,\sigma=0.5$. }
\end{minipage}
\end{center}
\caption{The relative errors
$|\omega_n^{(\alpha,\sigma)}-\widehat{\omega}_n^{(\alpha,\sigma)}|/|\omega_n^{(\alpha,\sigma)}|$
(circles)
and
$|\omega_n^{(\alpha,\sigma)}-\widetilde{\omega}_n^{(\alpha,\sigma)}|/|\omega_n^{(\alpha,\sigma)}|$
(diamonds), Example \ref{eg21}, $B=5$, $\tau=0.01$.
The number of quadrature points for discretizing each Talbot contour quadrature is
 $N=64$; the number of quadrature points for the trapezoidal rule is
 $Q=N\times L=64\times 7=448$. \label{eg51fig3}}
\end{figure}

In the following section, we apply Fast Method II to solve
a number of time-fractional differential models.

\section{Numerical examples}\label{sec:numerical}
In this section, two examples are presented to verify the effectiveness
of the present fast convolution.
In the direct methods for solving FDEs in this section,
the (tempered) fractional operators in the considered FDEs are always discretized by
$D_{\tau}^{\alpha,\sigma,\gamma,m,n}$ (see \eqref{FLMM})
with the convolution quadrature weights defined by \eqref{GNGF-S} with $p=2$,
i.e.,  GNGF-$2$ is applied.
For convenience, we define
\begin{equation}\label{sec5:eq-0}
{}_FD_{\tau}^{\alpha,\sigma,\gamma,m,n}u
={}_Fu^{(\alpha,\sigma)}_n
+\tau^{-\alpha}\sum_{k=1}^mw^{(\alpha,\sigma)}_{n,k}(u_k-u_0) - b_n^{(\alpha,\sigma)}u_0,
\end{equation}
where $b_n^{(\alpha,\sigma)}=\tau^{-\alpha}\sum_{j=0}^n\omega^{(\alpha,\sigma)}_{j}$
and ${}_Fu^{(\alpha,\sigma)}_n$ is defined by \eqref{real-int-7-2}.

All the algorithms are implemented using MATLAB
2017b, which were run in a 3.40 GHz PC having 16GB RAM and Windows 7 operating system.
\begin{example}\label{s5-eg-1}
Consider the following scalar FODE
\begin{equation}\label{sec5:eq-1}
D^{\sigma,\alpha}_{0,t}(u(t)-u(0)) = -u(t) + f(u,t),{\quad}u(0)=u_0,{\quad}t\in(0, T],
\end{equation}
where $0<\alpha \leq 1$ and $\sigma \geq 0$.
\end{example}

Let $U_n$ be the numerical solution of \eqref{sec5:eq-1}.
The fully implicit fast method  for solving \eqref{sec5:eq-1} is given by
\begin{equation}\label{scheme1}
{}_FD_{\tau}^{\alpha,\sigma,\gamma,m,n}U=-U_n + f(U_n,t_n), \quad U_0=u_0,
\end{equation}
where ${}_FD_{\tau}^{\alpha,\sigma,\gamma,m,n}$ is defined by \eqref{sec5:eq-0}, $m$ is the number
of correction terms.

We need to know $U_k(1\leq k\leq m)$ when \eqref{scheme1} is applied.
In this paper, $U_k(1\leq k\leq m)$ are obtained by solving \eqref{scheme1} with a
small step size $2^{-7}\tau$ and $m=0$ or $m=1$ if there is at least one correction term.
When we say the direct method is applied, we mean that ${}_FD_{\tau}^{\alpha,\sigma,\gamma,m,n}$
in \eqref{scheme1} is replaced by $D_{\tau}^{\alpha,\sigma,\gamma,m,n}$.
The Newton method is applied to solve the nonlinear system \eqref{scheme1} to obtain $U_n$.

The following two cases are considered in this example.
\begin{itemize}
  \item Case I: For the linear case of $f=0$, the exact solution of \eqref{sec5:eq-1} is
  $$u(t)=E_{\alpha}(-t^{\alpha})e^{-\sigma t},$$
  where $E_{\alpha}(t)$ is the  Mittag--Leffler function defined by
$E_{\alpha}(t)=\sum_{k=0}^{\infty}\frac{t^{k}}{\Gamma(k\alpha+1)}.$
  \item Case II: Let $f=u(1-u^2)$ and  $u_0=1$.
\end{itemize}

The maximum error is defined by
$$\|e\|_{\infty}=\max_{0\leq n \leq {T}/\tau}\left|e_n\right|,  {\quad}e_n=u(t_n)-U_n,\,\,T=10.$$

We first show that  the use of the correction terms decreases the global error of
the method significantly for Case I. From Tables \ref{s5:tb1}--\ref{s5:tb3}, we can see
that increasing the number of correction terms improves the accuracy significantly,
and second-order accuracy is observed for some suitable $m$.
Numerical simulations show that the inaccurate numerical solutions near the origin
weakly affect the numerical solutions far from the origin. We show the numerical solutions
at $t=10$ for $\sigma=0.2$ and $0.5$ in Table \ref{s5:tb4}. We can see that
much better numerical solutions are obtained even if no correction term is added
and second-order accuracy is observed using one or two correction terms.

\begin{table}[!h]
\caption{{The maximum error $\|e\|_{\infty}$ for Example \ref{s5-eg-1}, Case I, $\gamma_k=k\alpha$,  $\alpha=0.5$,  ${T=10}$, $N=256$, and $\sigma=0$.}}\label{s5:tb1}
\centering\footnotesize
\begin{tabular}{|c|c|c|c|c|c|c|c|c|c|c|c|c|}
\hline
 $\tau$ & $m=0$ & Order& $m=1$ & Order& $m=2$ & Order & $m=3$   & Order\\
 \hline
$2^{-5}$ &4.8036e-2&      &4.7715e-4&      &2.5974e-5&      &1.9848e-5&      \\
$2^{-6}$ &3.5869e-2&0.4214&2.5331e-4&0.9135&1.2654e-5&1.0374&6.9865e-6&1.5064\\
$2^{-7}$ &2.6373e-2&0.4437&1.3164e-4&0.9443&5.6593e-6&1.1609&2.3047e-6&1.6000\\
$2^{-8}$ &1.9175e-2&0.4598&6.7484e-5&0.9640&2.3617e-6&1.2608&7.2179e-7&1.6749\\
$2^{-9}$ &1.3830e-2&0.4714&3.4294e-5&0.9766&9.3879e-7&1.3309&2.1679e-7&1.7353\\
\hline
\end{tabular}
\end{table}

\begin{table}[!h]
\caption{{The maximum error $\|e\|_{\infty}$  for Example \ref{s5-eg-1}, Case I, $\gamma_k=k\alpha$,  $\alpha=0.5$, $N=256$, and $\sigma=0.5$.}}\label{s5:tb3}
\centering\footnotesize
\begin{tabular}{|c|c|c|c|c|c|c|c|c|c|c|c|c|}
\hline
 $\tau$ & $m=0$ & Order& $m=1$ & Order& $m=3$ & Order & $m=5$   & Order\\
 \hline
$2^{-5}$&4.8581e-2&      &2.1708e-4&      &4.1006e-5&      &3.6445e-5&      \\
$2^{-6}$&3.6270e-2&0.4216&1.2066e-4&0.8473&1.1085e-5&1.8873&1.1950e-5&1.6087\\
$2^{-7}$&2.6622e-2&0.4462&6.4140e-5&0.9116&2.9336e-6&1.9178&3.5868e-6&1.7363\\
$2^{-8}$&1.9318e-2&0.4627&3.3264e-5&0.9473&7.6422e-7&1.9406&1.0161e-6&1.8196\\
$2^{-9}$&1.3908e-2&0.4740&1.7008e-5&0.9677&2.3723e-7&1.6877&2.7697e-7&1.8753\\
\hline
\end{tabular}
\end{table}

\begin{table}[!h]
\caption{The absolute error  $\left|e_n\right|$
at   ${t=10}$, Example \ref{s5-eg-1}, Case I, $\gamma_k=k\alpha$,  $\alpha=0.5$, $N=256$.}\label{s5:tb4}
\centering\footnotesize
$\sigma=0.2$\\
\begin{tabular}{|c|c|c|c|c|c|c|c|c|c|c|c|c|}
\hline
 $\tau$ & $m=0$ & Order& $m=1$ & Order& $m=2$ & Order  \\
 \hline
$2^{-5}$&1.5240e-5&      &6.0747e-6&      &1.4693e-5&        \\
$2^{-6}$&7.9796e-6&0.9335&1.5738e-6&1.9486&4.0865e-6&1.8462\\
$2^{-7}$&4.0758e-6&0.9692&4.0084e-7&1.9731&1.1023e-6&1.8903\\
$2^{-8}$&2.0580e-6&0.9859&1.0057e-7&1.9948&2.9090e-7&1.9220\\
$2^{-9}$&1.0335e-6&0.9937&2.4856e-8&2.0166&7.5573e-8&1.9446\\
\hline
\end{tabular}\\
$\sigma=0.5$\\
\begin{tabular}{|c|c|c|c|c|c|c|c|c|c|c|c|c|}
\hline
 $\tau$ & $m=0$ & Order& $m=1$ & Order& $m=2$ & Order  \\
 \hline
$2^{-5}$&2.0238e-5&      &6.0702e-5&      &4.1007e-5&        \\
$2^{-6}$&5.0097e-6&2.0142&1.5645e-5&1.9560&1.0668e-5&1.9426\\
$2^{-7}$&1.2174e-6&2.0409&3.9910e-6&1.9709&2.7506e-6&1.9554\\
$2^{-8}$&2.8564e-7&2.0916&1.0114e-6&1.9804&7.0398e-7&1.9661\\
$2^{-9}$&6.1912e-8&2.2059&2.5521e-7&1.9866&1.7910e-7&1.9748\\
\hline
\end{tabular}
\end{table}


For Case II,  the explicit form of the analytical solution is unknown, and numerical solutions
are shown in Figure \ref{eg1fig1}.  For a fixed fractional order $\alpha=0.3$,
the solution decays slower and attains a steady state as $\sigma$ increases, see Figure  \ref{eg1fig1}(a).
We observe similar behavior for $\alpha=0.8$, see Figure  \ref{eg1fig1}(b).
For other fractional orders $\alpha\in (0,1)$, we observe similar results, which are not shown here.
Figures \ref{eg1fig2} (a)--(c) show the difference between the fast solution and the direct solution.
We can see that the two solutions are very close, which means the error caused from the trapezoidal \eqref{real-int-3}
rule in Fast Method II is very small.  Figure  \ref{eg1fig2} (d) shows the computational time of
the fast method and the direct method, and we observe that the fast method really outperforms
the direct method in efficiency and saves computational cost.
The advantage of the fast method will be further displayed in the following example,
solving a  time-fractional  activator-inhibitor system.

\begin{figure}[!h]
\begin{center}
\begin{minipage}{0.47\textwidth}\centering
\epsfig{figure=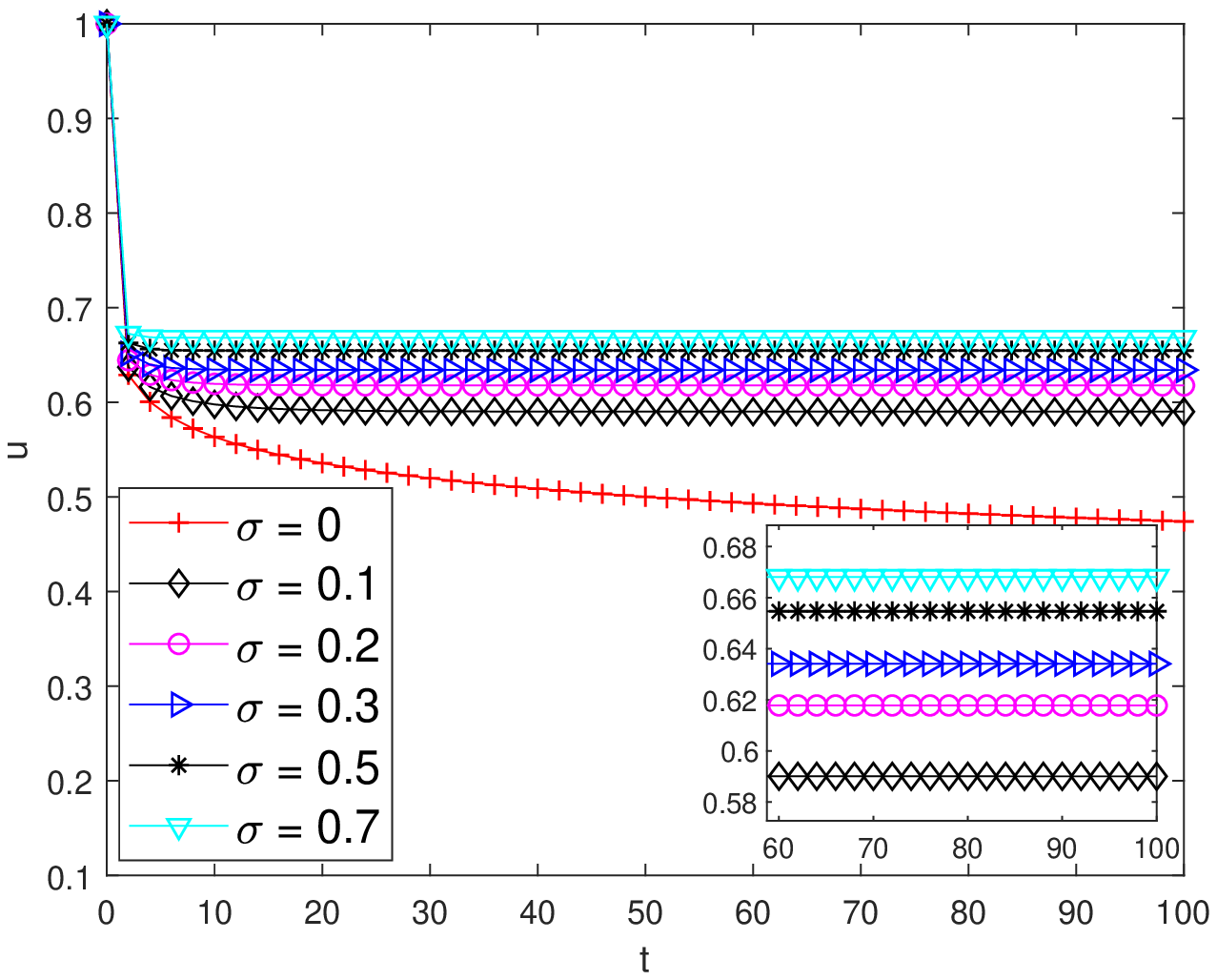,width=5.3cm} \par {(a) $\alpha=0.3$.}
\end{minipage}
\begin{minipage}{0.47\textwidth}\centering
\epsfig{figure=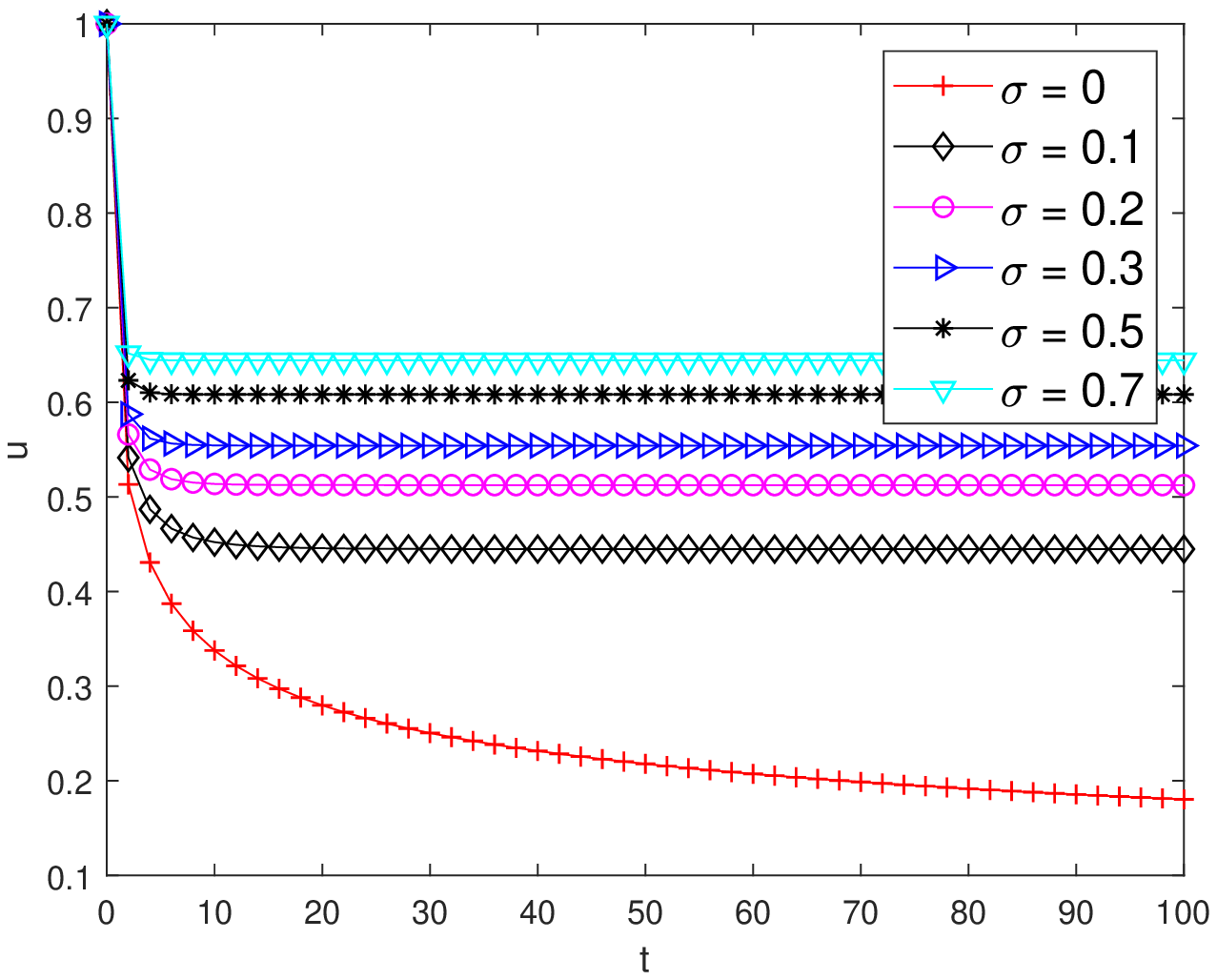,width=5.3cm} \par {(b)   $\alpha=0.8$.}
\end{minipage}
\end{center}
\caption{Numerical solutions for Example \ref{s5-eg-1}, Case II,    $\tau=0.001,N=256$.\label{eg1fig1}}
\end{figure}

\begin{figure}[!h]
\begin{center}
\begin{minipage}{0.47\textwidth}\centering
\epsfig{figure=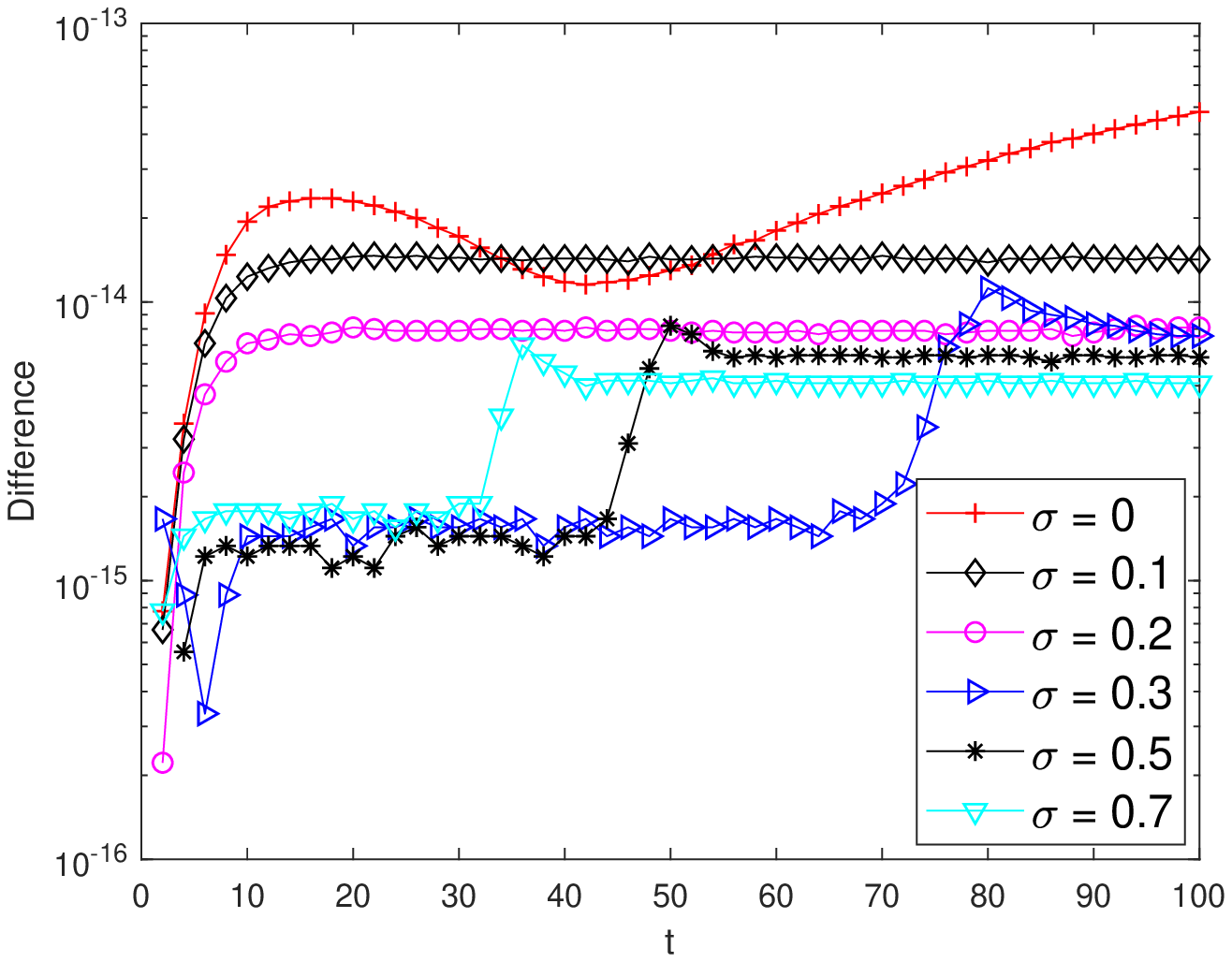,width=5.3cm} \par {(a) $\alpha=0.3$.}
\end{minipage}
\begin{minipage}{0.47\textwidth}\centering
\epsfig{figure=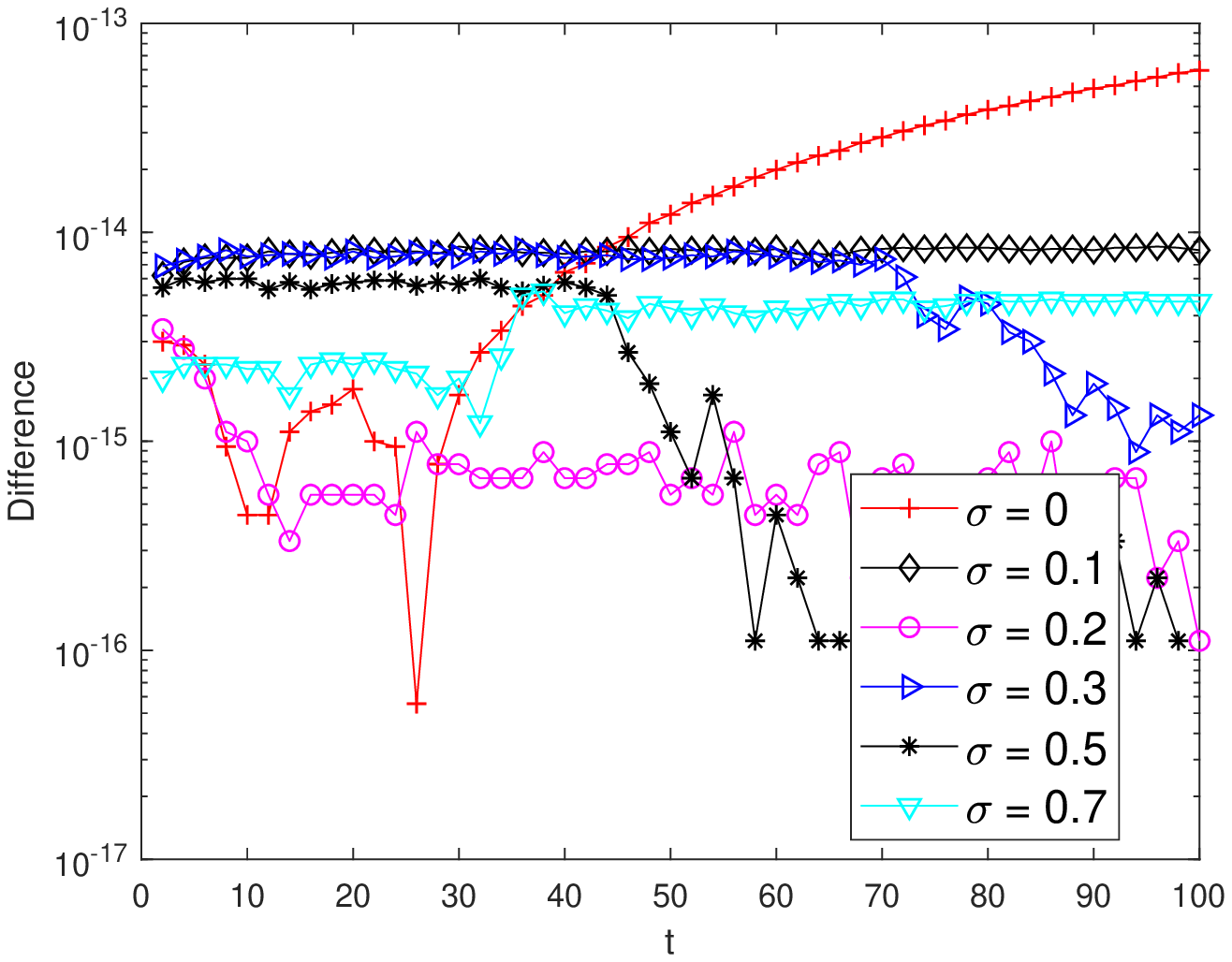,width=5.3cm} \par {(b) $\alpha=0.5$.}
\end{minipage}
\begin{minipage}{0.47\textwidth}\centering
\epsfig{figure=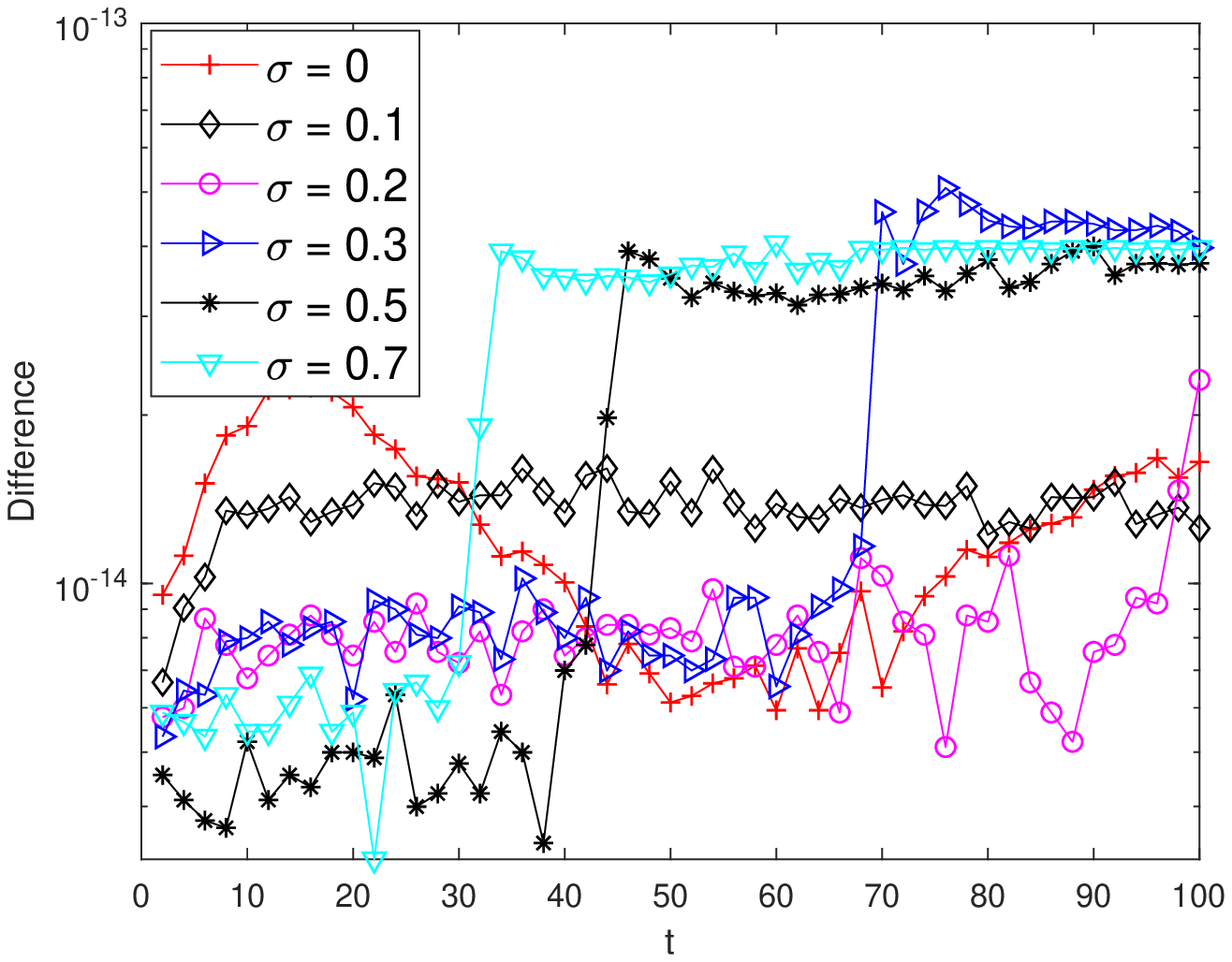,width=5.3cm} \par {(c) $\alpha=0.8$.}
\end{minipage}
\begin{minipage}{0.47\textwidth}\centering
\epsfig{figure=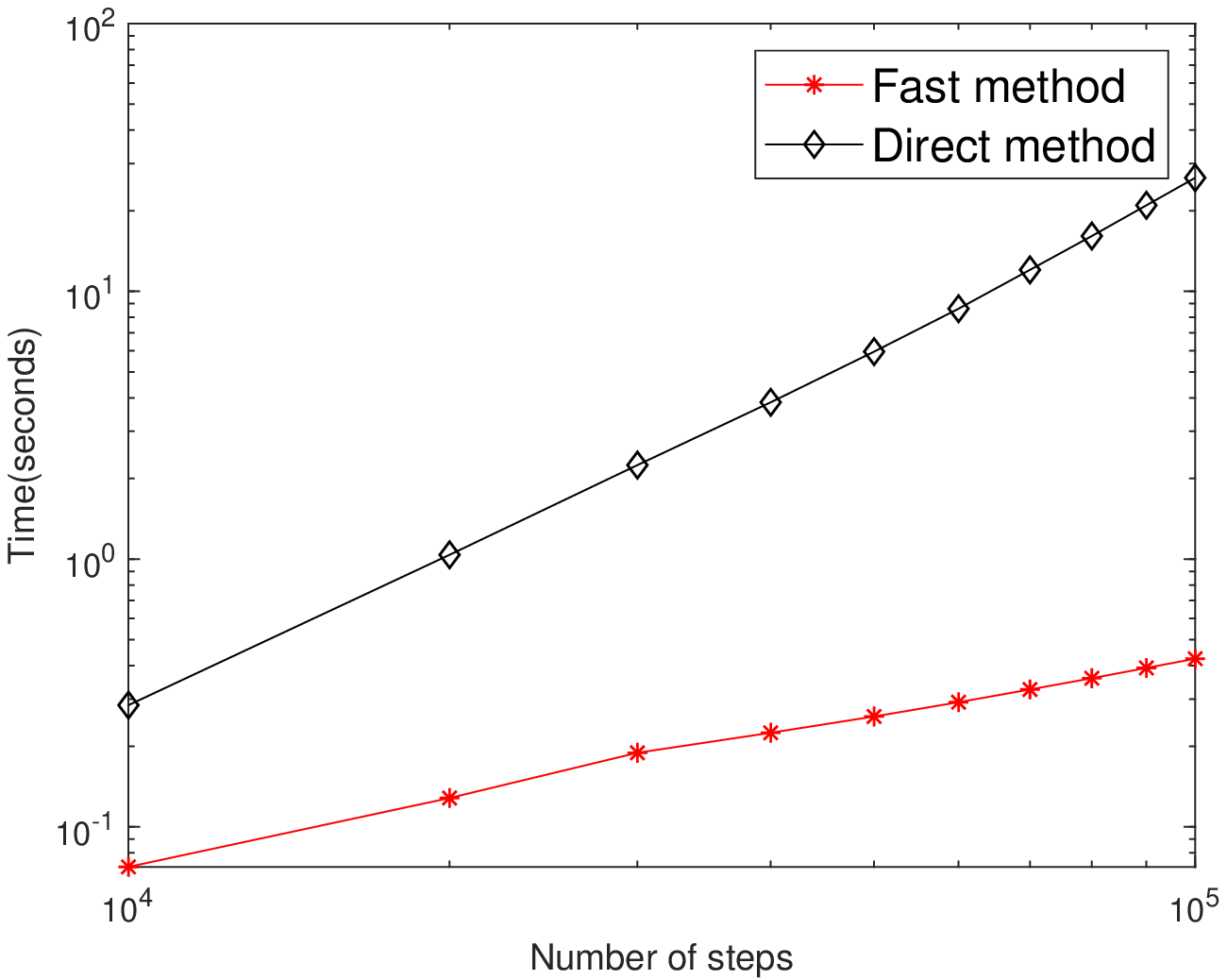,width=5.3cm} \par {(d) Computational time.}
\end{minipage}
\end{center}
\caption{(a)-(c): the  difference between the numerical solutions of the direct method
and the fast method; (d): the computational time of the fast method and direct method;
Example \ref{s5-eg-1}, Case II,  $\tau=0.001,Q=256$.\label{eg1fig2}}
\end{figure}

%

\begin{example}\label{s5-eg-2}
Consider the   fractional activator-inhibitor system   \cite{Henry02}
\begin{eqnarray}
\px[t] u(x,t) &=&\kappa f_1(u,v)+{}_{RL}D_{0,t}^{1-\alpha_1}\px^2u(x,t),
{\quad}0\leq x \leq D,\label{sec5:eq-21}\\
\px[t] v(x,t)&=&\kappa f_2(u,v)+d\,{}_{RL}D_{0,t}^{1-\alpha_2}\px^2v(x,t),
{\quad}0\leq x \leq D,\label{sec5:eq-22}
\end{eqnarray}
where $u(x,t)$ and $v(x,t)$ denote the concentrations of the activator and inhibitor, respectively, $0\le\alpha_1\leq 1$ is the anomalous
diffusion exponent of the activator, and $0\le\alpha_2\leq 1$ is the anomalous diffusion exponent of the inhibitor, $d$ is the ratio of the diffusion
coefficients of inhibitor to activator, and $\kappa> 0$ is a scaling variable that can be interpreted as the characteristic size of the spatial domain
or as the relative strength of the reaction terms.  The reaction
kinetics is defined by the functions $f_1(u,v)$ and $f_2(u,v)$.

In our following numerical test, we will consider the Turing pattern formation in the fractional activator-inhibitor model system described by system \eqref{sec5:eq-21}--\eqref{sec5:eq-22} with zero-flux boundary conditions at both ends of the spatial domain of length $D$, i.e.
\begin{equation}\label{sec5:eq-24}
\px[x]u(0,t)=\px[x]v(0,t)=0,\quad
\px[x]u(D,t)=\px[x]v(D,t)=0.
\end{equation}
\end{example}

We apply cubic finite element to approximate the space of \eqref{sec5:eq-21}--\eqref{sec5:eq-22}.
For the time discretization,  we apply a stabilized  semi-implicit time-stepping method,
i.e., the first-order time derivative  is discretized by the  second-order backward
difference formula, the time-fractional derivative  is discretized by
the second-order generalized Newton--Gregory formula, and the nonlinear term is
approximated using a second-order extrapolation with a stablization factor.

Let $X_h$ be a cubic piecewise finite element space defined on
the uniform grids $\{x_i\}$, where $x_i=ih$, $h$ is space stepsize, and $D/h$ is a positive integer.
The numerical scheme for  \eqref{sec5:eq-21}--\eqref{sec5:eq-24}
is given by: For $2\leq n \leq n_T$, find $u_h^n,v_h^n\in X_h$, such that
\begin{eqnarray}
&&(D^{n}_{\tau}u_h,w) +({}_FD_{\tau}^{1-\alpha_1,0,0,0,n}\px u_h,\px w)
+b_n^{(1-\alpha_1,0)}(\px u_h^0,\px w)\nonumber\\
=&&\kappa (2F_1^{n-1}-F_1^{n-2},w)-\kappa_1(u_h^n-2u_h^{n-1}+u_h^{n-2},w),
\quad\forall w\in X_h, \label{sec5:eq-21-1}\\
&&(D^{n}_{\tau}v_h,w)+d({}_FD_{\tau}^{1-\alpha_2,0,0,0,n}\px v_h,\px w)
+db_n^{(1-\alpha_2,0)}(\px v_h^0,\px w)\nonumber\\
=&&\kappa (2F_2^{n-1}-F_2^{n-2},w)-\kappa_2(v_h^n-2v_h^{n-1}+v_h^{n-2},w),
\quad\forall w\in X_h,\\
&&(u_h^0,w) = (u(0),w), \quad (u_h^1,w) = (u(0)+\tau\pt u(0),w),\quad\forall w\in X_h,\\
&&(v_h^0,w) = (v(0),w),\quad (v_h^1,w) = (v(0)+\tau\pt v(0),w),\quad\forall w\in X_h,\label{sec5:eq-22-1}
\end{eqnarray}
where
$F_1^n=f_1(u^n_h,v_h^n)$, $F_2^n=f_2(u^n_h,v_h^n)$,
$u(t)=u(x,t)$, $v(t)=v(x,t)$,
$\kappa_1$ and $\kappa_2$ are positive numbers that stabilize the time-stepping method, $D^{n}_{\tau}u_h= (3u_h^n-4u_h^{n-1}+u_h^{n-2})/(2\tau)$,
${}_FD_{\tau}^{\alpha,0,0,0,n}$ and $b_n^{(\alpha,\sigma)}$ are defined in \eqref{sec5:eq-0}.

Two kinds of reaction kinetics, Gierer--Meinhardt and Brusselator, will be considered for the fractional activator-inhibitor model system.
We consider the same initial conditions as those in \cite{Henry05},
which take the forms $u(x,0)=u^*+\epsilon r_1(x)$ and $v(x,0)=v^*+\epsilon r_2(x)$.
Three different types of perturbation are considered here:
(i) random, where $r_j(x)$ is a uniform random function on the interval $[-1,1]$;
(ii) long-wavelength sinusoidal, $r_1(x)=r_2(x)=\epsilon\sin(qx)$,
with $q=0.4$ (Gierer--Meinhardt) or $q=0.5$ (Brusselator);
(iii) short-wavelength sinusoidal, $r_1(x)=r_2(x)=\epsilon\sin(qx)$,
with $q=5$ (both Gierer--Meinhardt and Brusselator).
We set $\epsilon=0.01$ in each case.

\begin{itemize}[leftmargin=*]
  \item  \textbf{Gierer--Meinhardt reaction kinetics}. For
  the Gierer-Meinhardt reaction kinetics, $f_1$ and $f_2$ are given by
\begin{eqnarray}
f_1(u,v)&=&1-u+3{u^2}/{v},\label{sec5:eq-25}\\
f_2(u,v)&=&u^2-v.\label{sec5:eq-26}
\end{eqnarray}

The fractional activator-inhibitor model system defined by \eqref{sec5:eq-21}--\eqref{sec5:eq-24} and \eqref{sec5:eq-25}--\eqref{sec5:eq-26}
has a homogeneous steady state of $u^*=4$ and $v^*=16$. Standard linear stability analysis \cite{Murray03,Henry02,Henry05} reveals that in the case of standard diffusion
$\alpha_1=\alpha_2=1$ nonhomogeneous steady states can occur if the value of $d$ exceeds the critical value $d^*\approx 19.79$, while for $d\le d^*$ initial perturbations
about the steady state decay to zero and no pattern results.
The critical value of $d^*$ for the fractional Gierer--Meinhardt reaction kinetics and the corresponding maximally excited modes over a range of $\alpha$ are listed in \cite{Henry05}.
\item  \textbf{Brusselator reaction kinetics}.
For the Brusselator reaction kinetics reaction kinetics, $f_1$ and $f_2$ are given by
\begin{eqnarray}
f_1(u,v)&=&2-3u+u^2v,\label{sec5:eq-27}\\
f_2(u,v)&=&2u-u^2v.\label{sec5:eq-28}
\end{eqnarray}
In this case, the homogeneous steady-state solution is given by $u^*=2$ and $v^*=1$. The critical value of $d$ for a turing instability is given by $d^*\approx 23.31$. It has been shown that the overall pattern of behavior is similar to that found for the fractional Gierer-Meinhardt model
\cite{Henry05}.
\end{itemize}

The parameters are taken as $h=D/256$, $D=100$, $\tau=0.01$, $\kappa_1=\kappa_2=10$ (Gierer--Meinhardt)
or $\kappa_1=\kappa_2=2$ (Brusselator)
when the numerical method \eqref{sec5:eq-21-1}--\eqref{sec5:eq-22-1} is applied.
For both Gierer--Meinhardt and Brusselator reaction kinetics, we take the same values
of $\alpha_1,\alpha_2$ and $d$ as in \cite{Henry05} in our simulations.

Figures  \ref{turing_pattern1} and \ref{turing_pattern2}  show the full surface profiles for
the concentrations of the
activator $u$ (left column) and inhibitor $v$ (right column) with randomly perturbed initial conditions
and $\alpha_1=\alpha_2=\alpha$, where the activator shows similar behavior as the inhibitor.
We obtain similar results as those in \cite{Henry05}:
i) The concentrations of the activator and inhibitor both fluctuate about the homogenous steady-state values.
ii) A spatiotemporal pattern develops on or before $t=500$. iii) The surface profiles
become more spatially rough and/or less stationary as the fractional order $\alpha$ decreases.

\begin{figure}[!h]
\begin{center}
\begin{minipage}{0.47\textwidth}\centering
\epsfig{figure=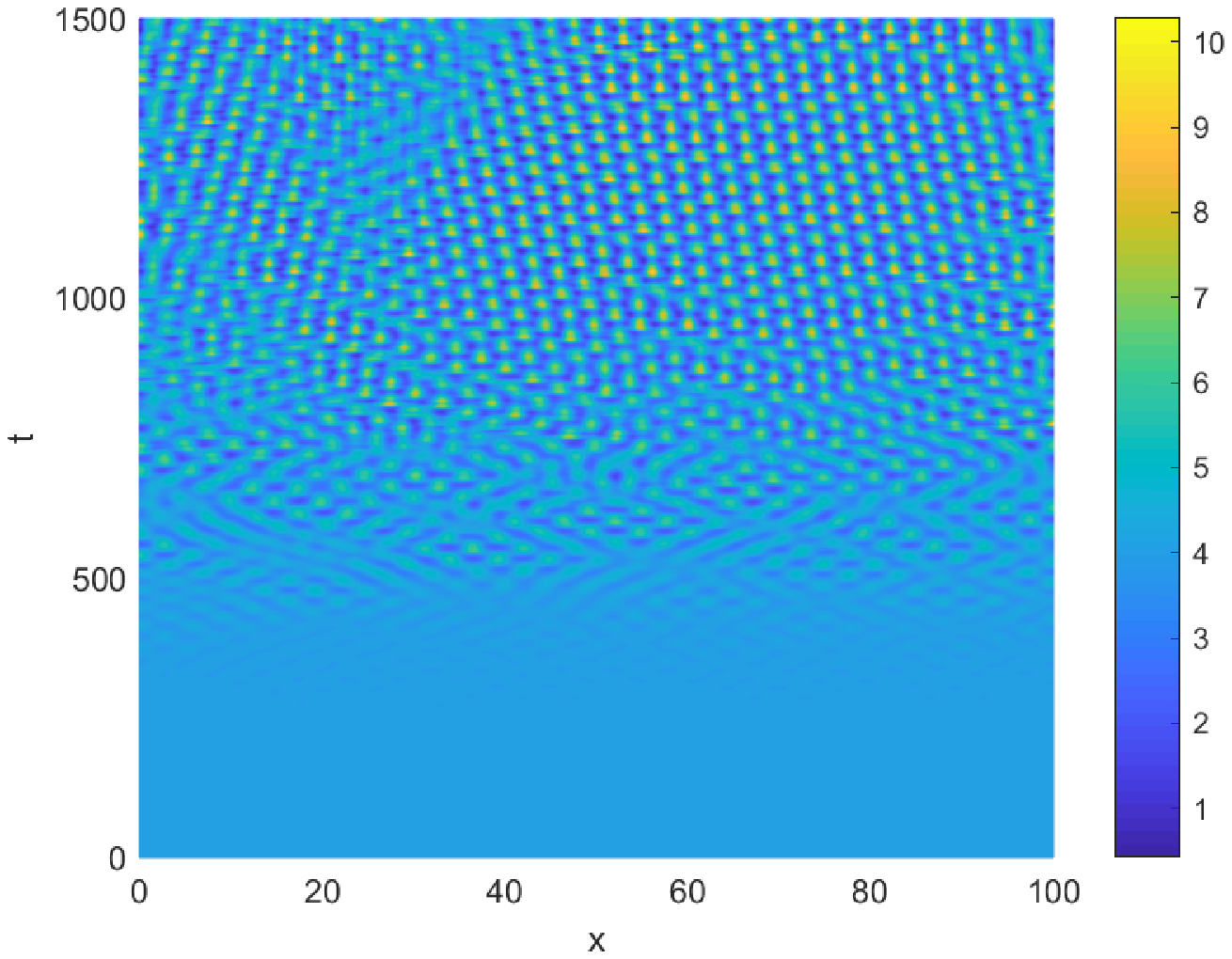,width=5cm} \par {(a) Surface profile of $u(x,t)$.}
\end{minipage}
\begin{minipage}{0.47\textwidth}\centering
\epsfig{figure=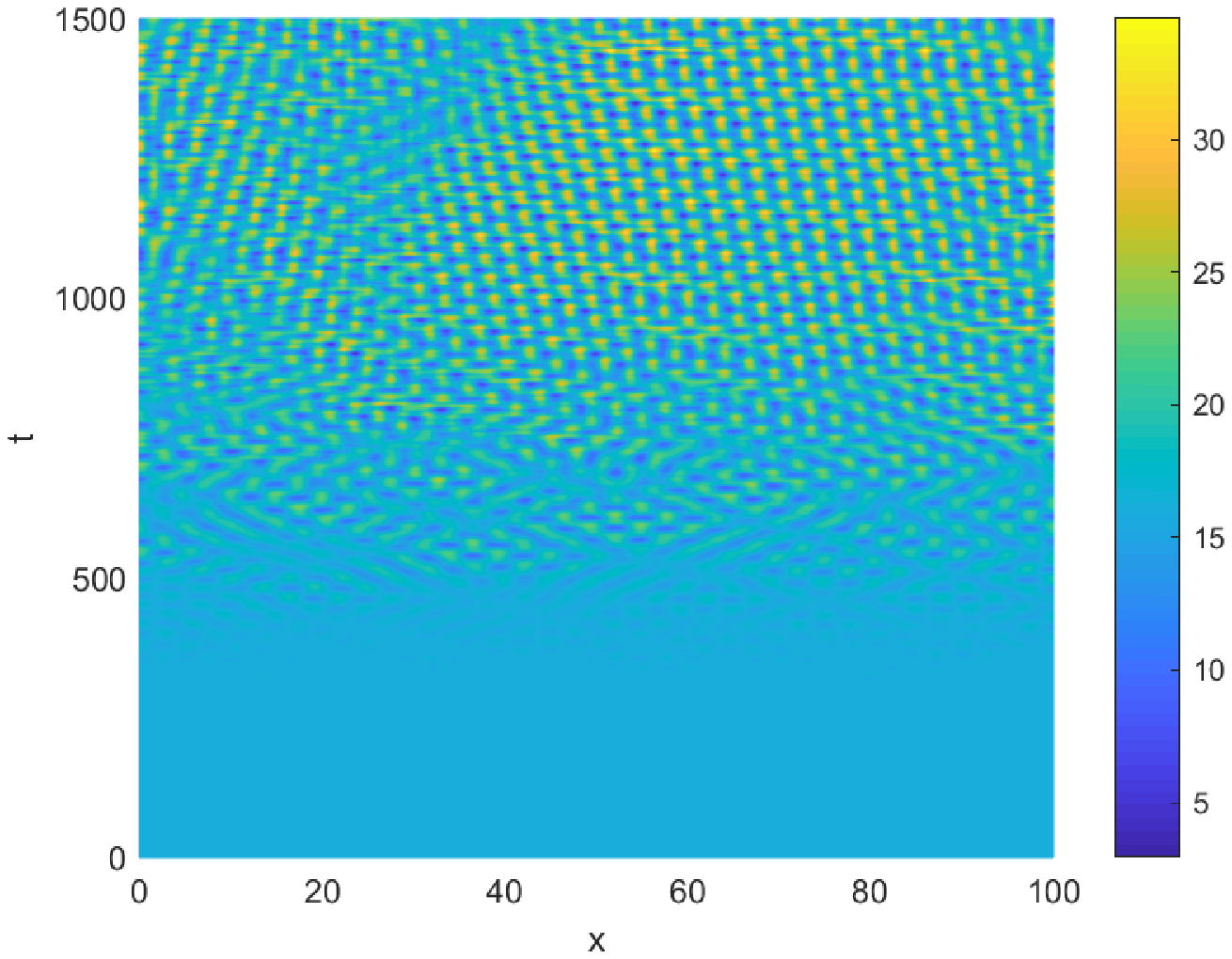,width=5cm} \par {(a) Surface profile of  $v(x,t)$.}
\end{minipage}
\begin{minipage}{0.47\textwidth}\centering
\epsfig{figure=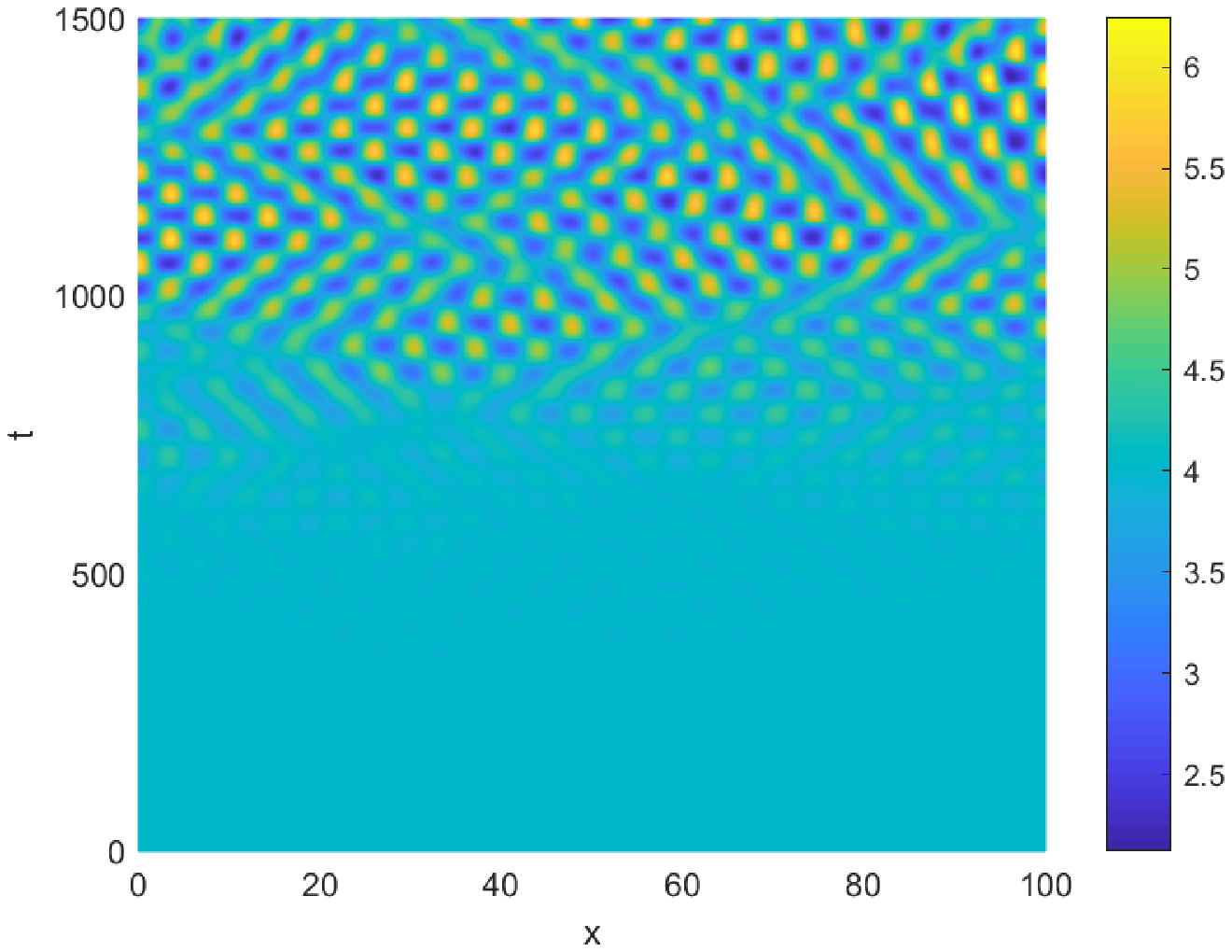,width=5cm} \par {(b) Surface profile of $u(x,t)$.}
\end{minipage}
\begin{minipage}{0.47\textwidth}\centering
\epsfig{figure=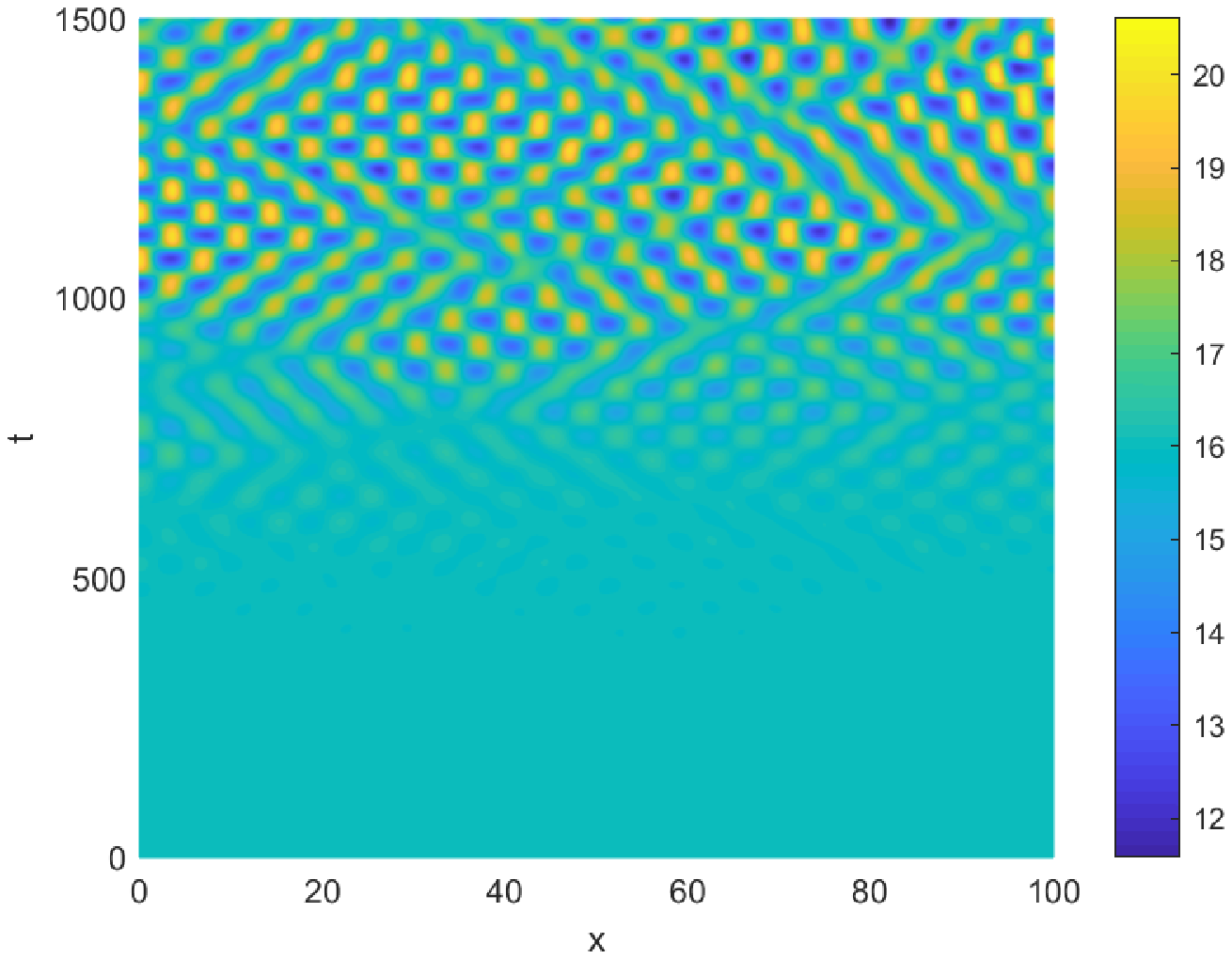,width=5cm} \par {(b) Surface profile of  $v(x,t)$.}
\end{minipage}
\begin{minipage}{0.47\textwidth}\centering
\epsfig{figure=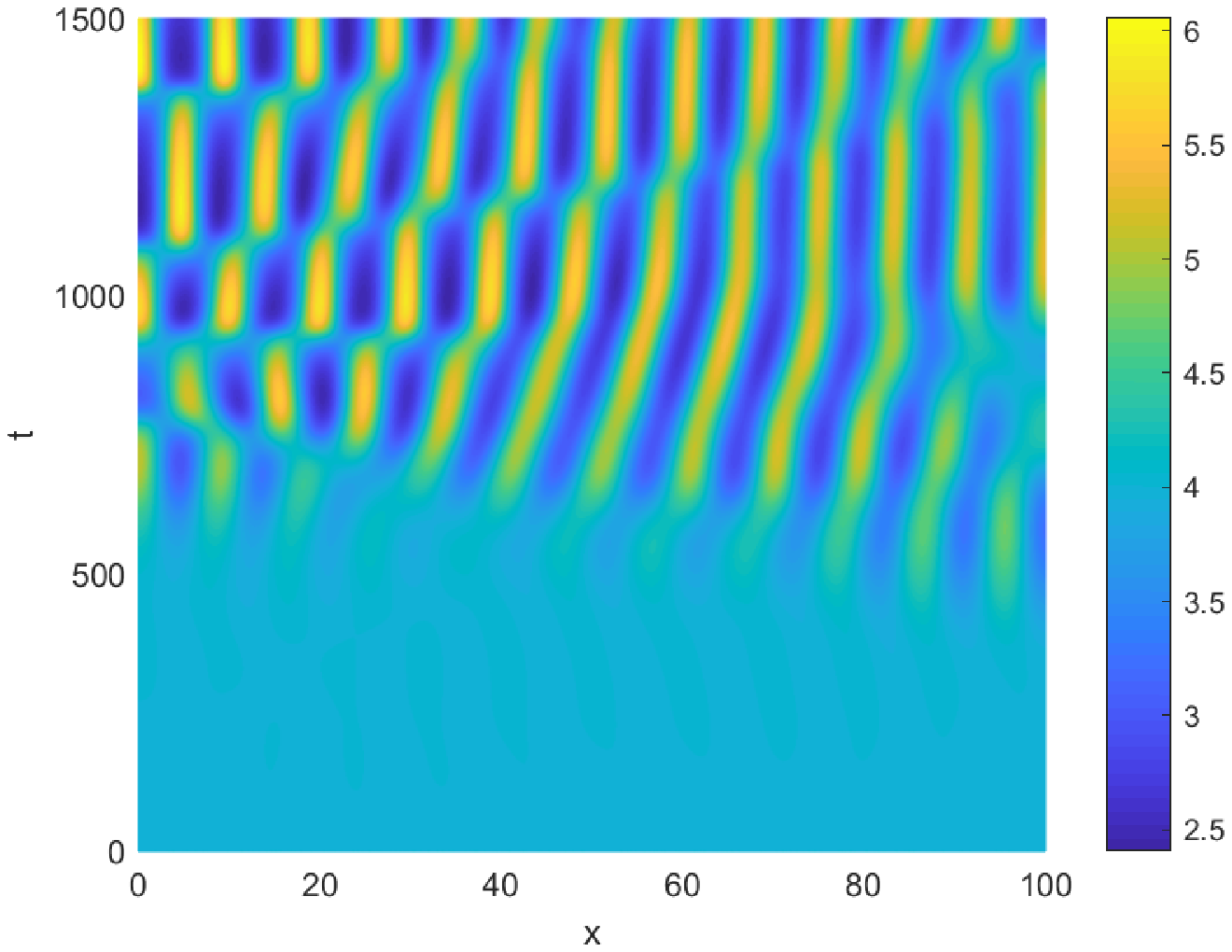,width=5cm} \par {(c) Surface profile of  $u(x,t)$.}
\end{minipage}
\begin{minipage}{0.47\textwidth}\centering
\epsfig{figure=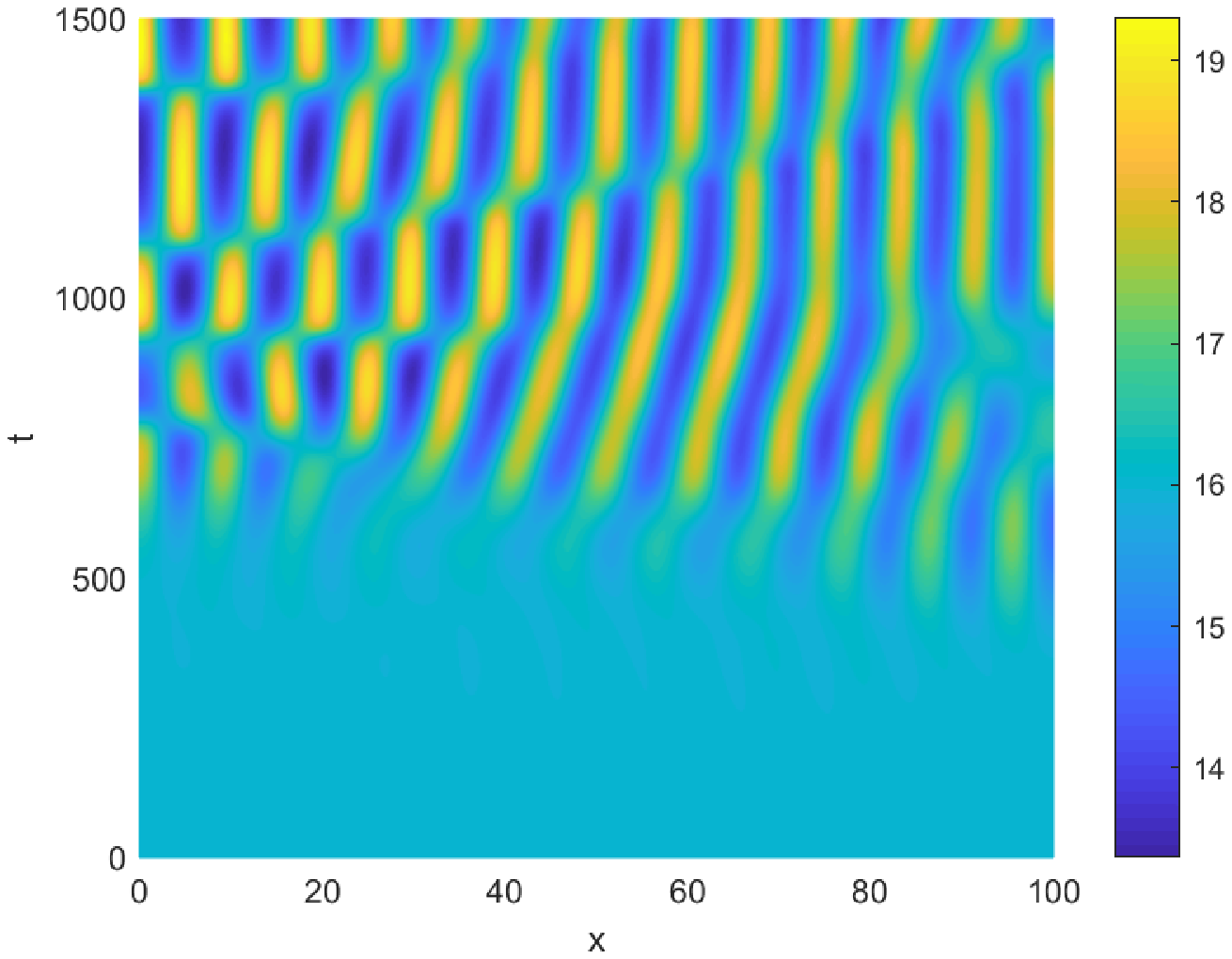,width=5cm} \par {(c) Surface profile of  $v(x,t)$.}
\end{minipage}
\end{center}
\caption{Fractional Gierer--Meinhardt model with randomly perturbed initial conditions (i),
Example \ref{s5-eg-2}: (a) $\alpha_1=\alpha_2=0.2$, $d=7$;
(b) $\alpha_1=\alpha_2=0.5$, $d=14$; (c) $\alpha_1=\alpha_2=0.8$, $d=21$.\label{turing_pattern1}}
\end{figure}

\begin{figure}[!h]
\begin{center}
\begin{minipage}{0.47\textwidth}\centering
\epsfig{figure=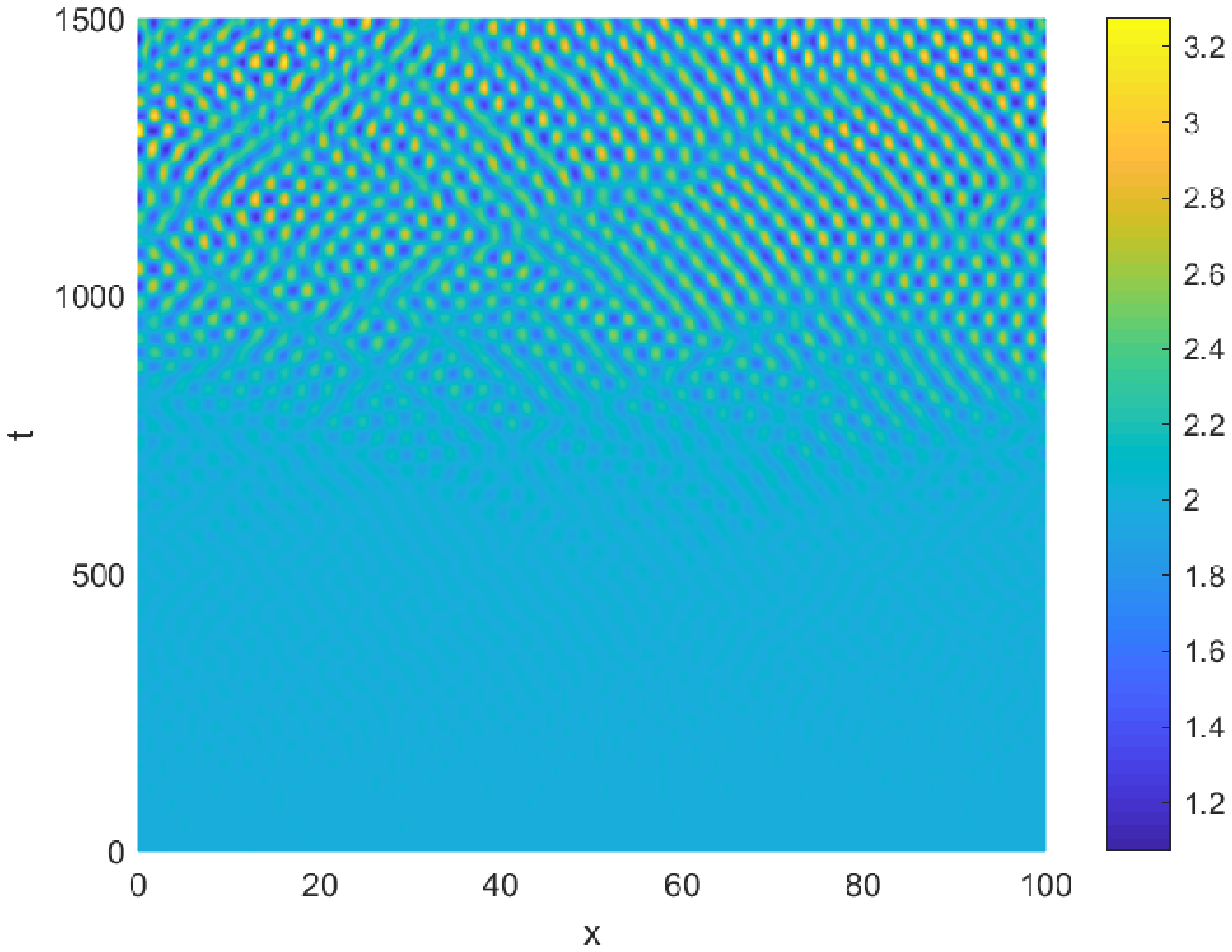,width=5cm} \par {(a) Surface profile of  $u(x,t)$.}
\end{minipage}
\begin{minipage}{0.47\textwidth}\centering
\epsfig{figure=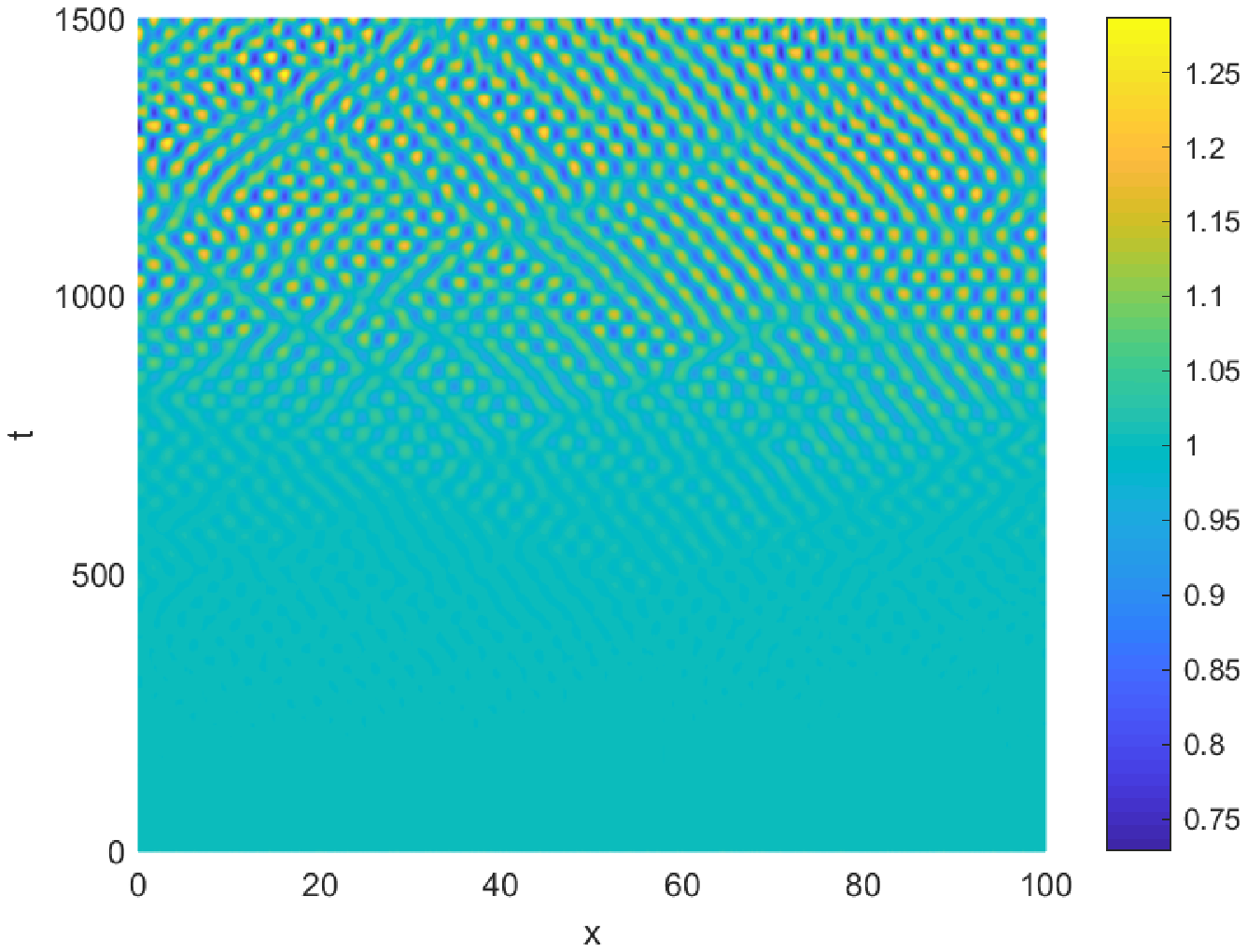,width=5cm} \par {(a) Surface profile of  $v(x,t)$.}
\end{minipage}
\begin{minipage}{0.47\textwidth}\centering
\epsfig{figure=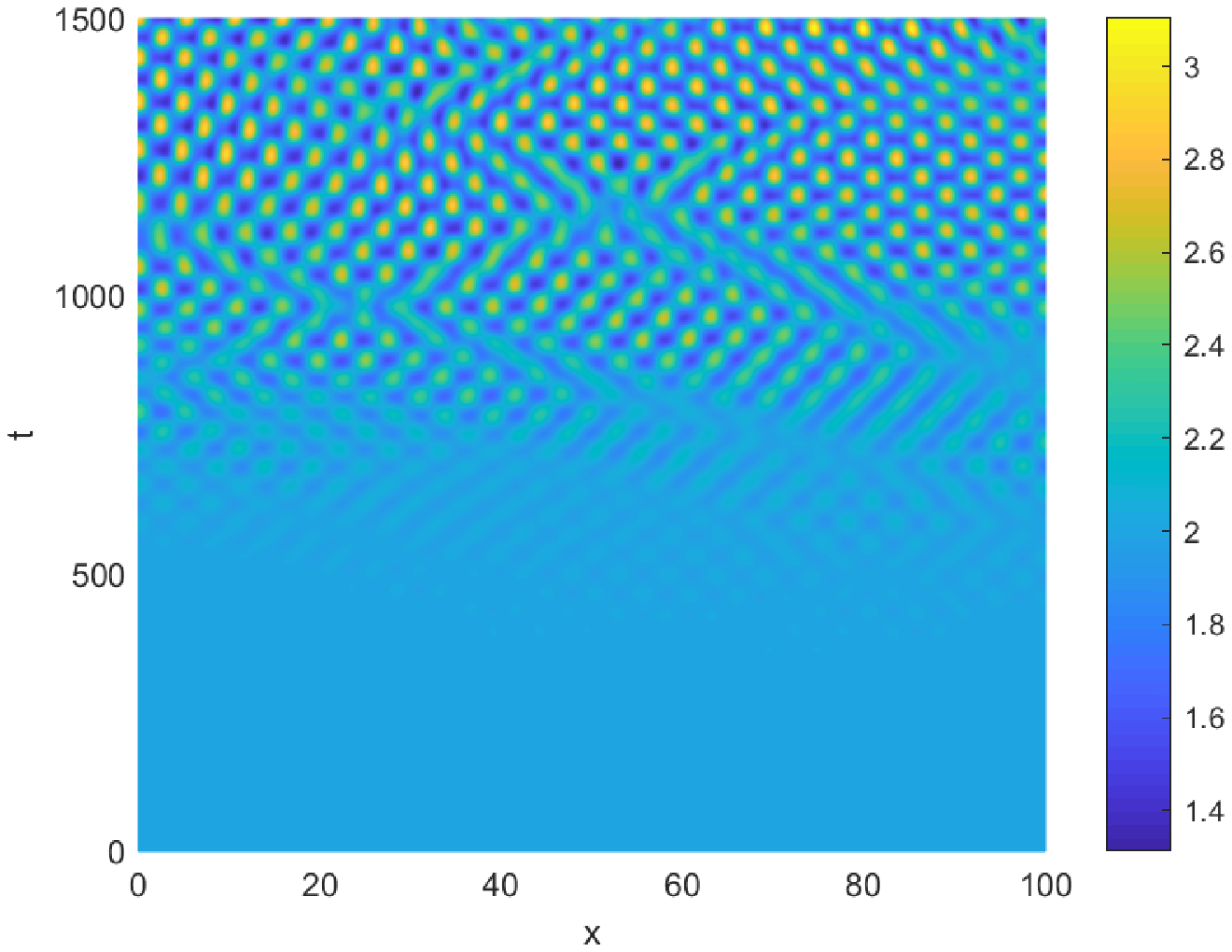,width=5cm} \par {(b) Surface profile of  $u(x,t)$.}
\end{minipage}
\begin{minipage}{0.47\textwidth}\centering
\epsfig{figure=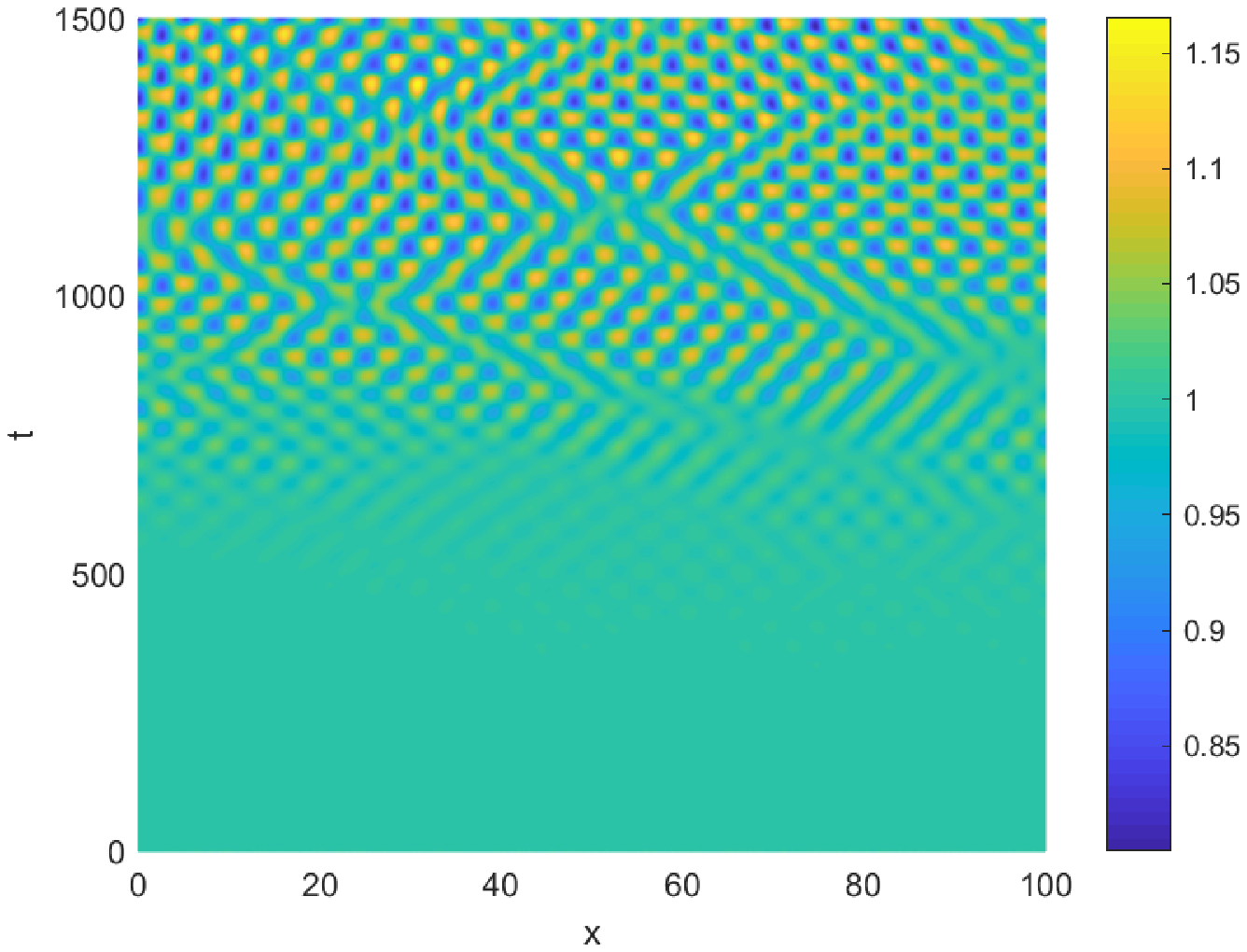,width=5cm} \par {(b) Surface profile of  $v(x,t)$.}
\end{minipage}
\begin{minipage}{0.47\textwidth}\centering
\epsfig{figure=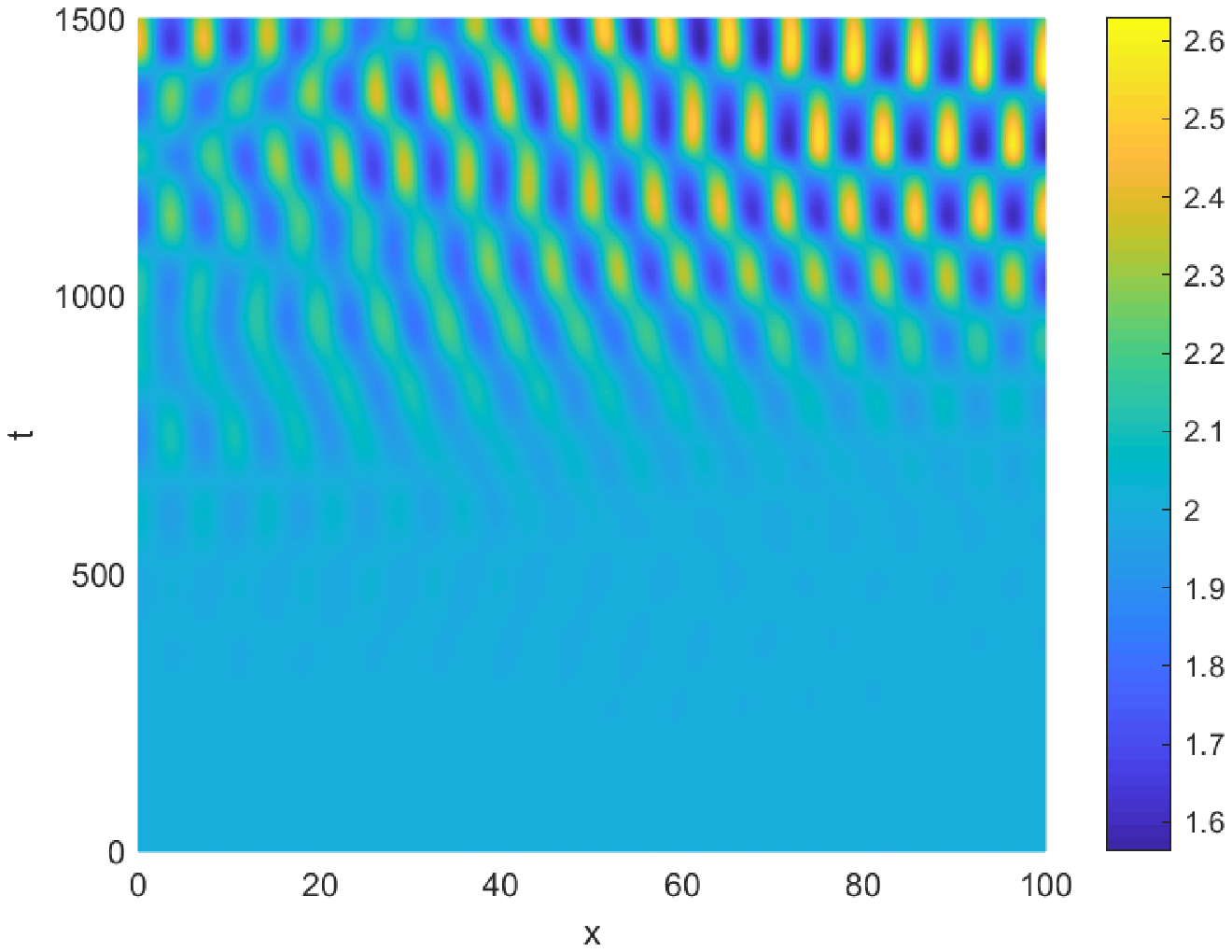,width=5cm} \par {(c) Surface profile of  $u(x,t)$.}
\end{minipage}
\begin{minipage}{0.47\textwidth}\centering
\epsfig{figure=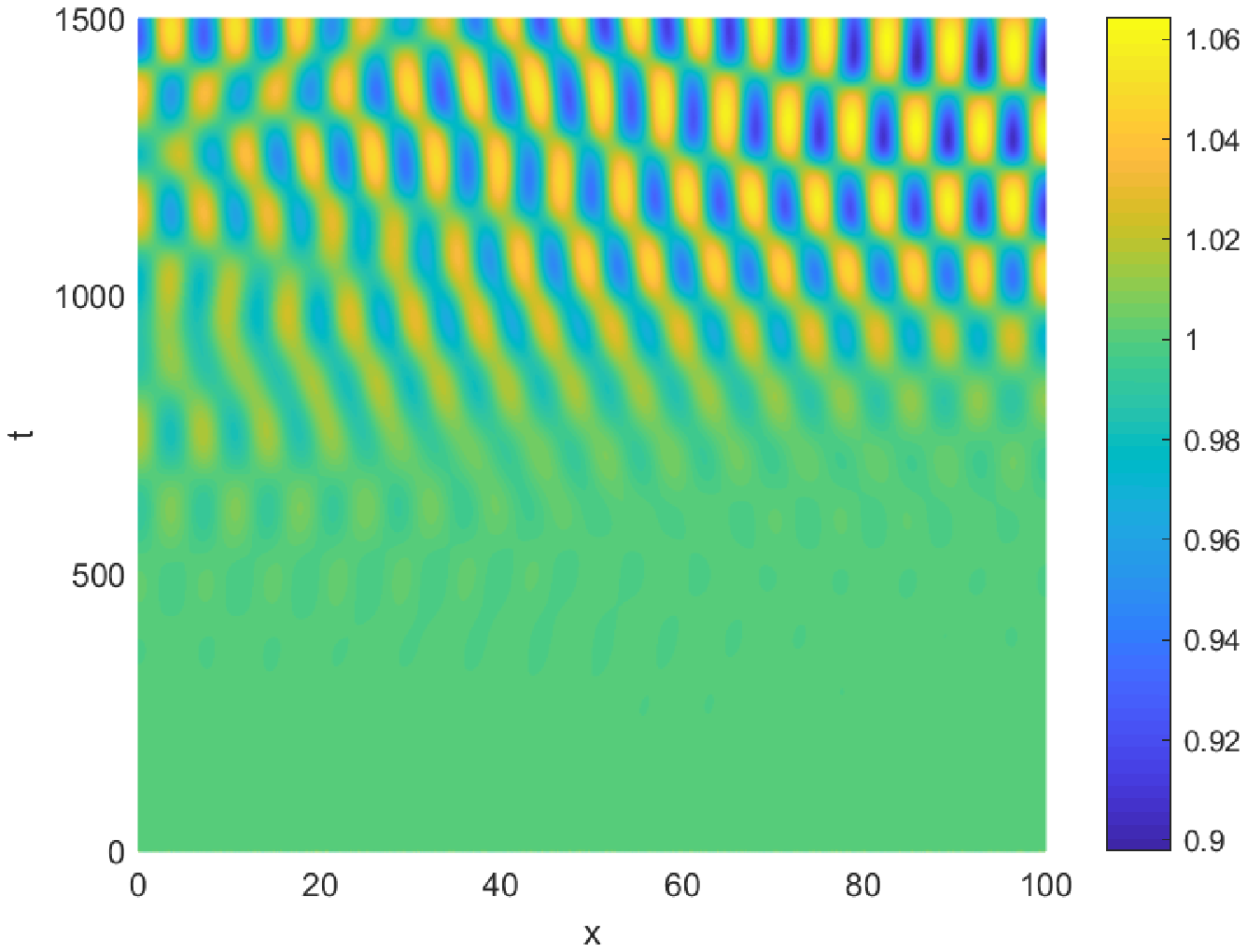,width=5cm} \par {(c) Surface profile of  $v(x,t)$.}
\end{minipage}
\end{center}
\caption{Fractional Brusselator model with randomly perturbed initial conditions (i), Example \ref{s5-eg-2}:
(a) $\alpha_1=\alpha_2=0.2$, $d=9$;
(b) $\alpha_1=\alpha_2=0.5$, $d=17$; (c) $\alpha_1=\alpha_2=0.8$, $d=23$.\label{turing_pattern2}}
\end{figure}

Figure \ref{turing_pattern3} shows the surface density plots of $u(x,t)\geq u^*$ (black)
and $u(x,t)<u^*$ (white) for the Brusselator model with  sinusoidally perturbed initial conditions (ii):
long-wavelength sinusoidally perturbations (left column) and
short-wavelength sinusoidally perturbations  (right column).
We observe the same results as shown in \cite{Henry05}, but
we use finer spatial resolution to obtain more accurate solutions.
For both long-wavelength sinusoidally perturbations and short-wavelength sinusoidally perturbations,
similar patterns are observed after $t=500$ for the same parameters $d$ and fractional
orders $\alpha_1=\alpha_2$.

\begin{figure}[!h]
\begin{center}
\begin{minipage}{0.47\textwidth}\centering
\epsfig{figure=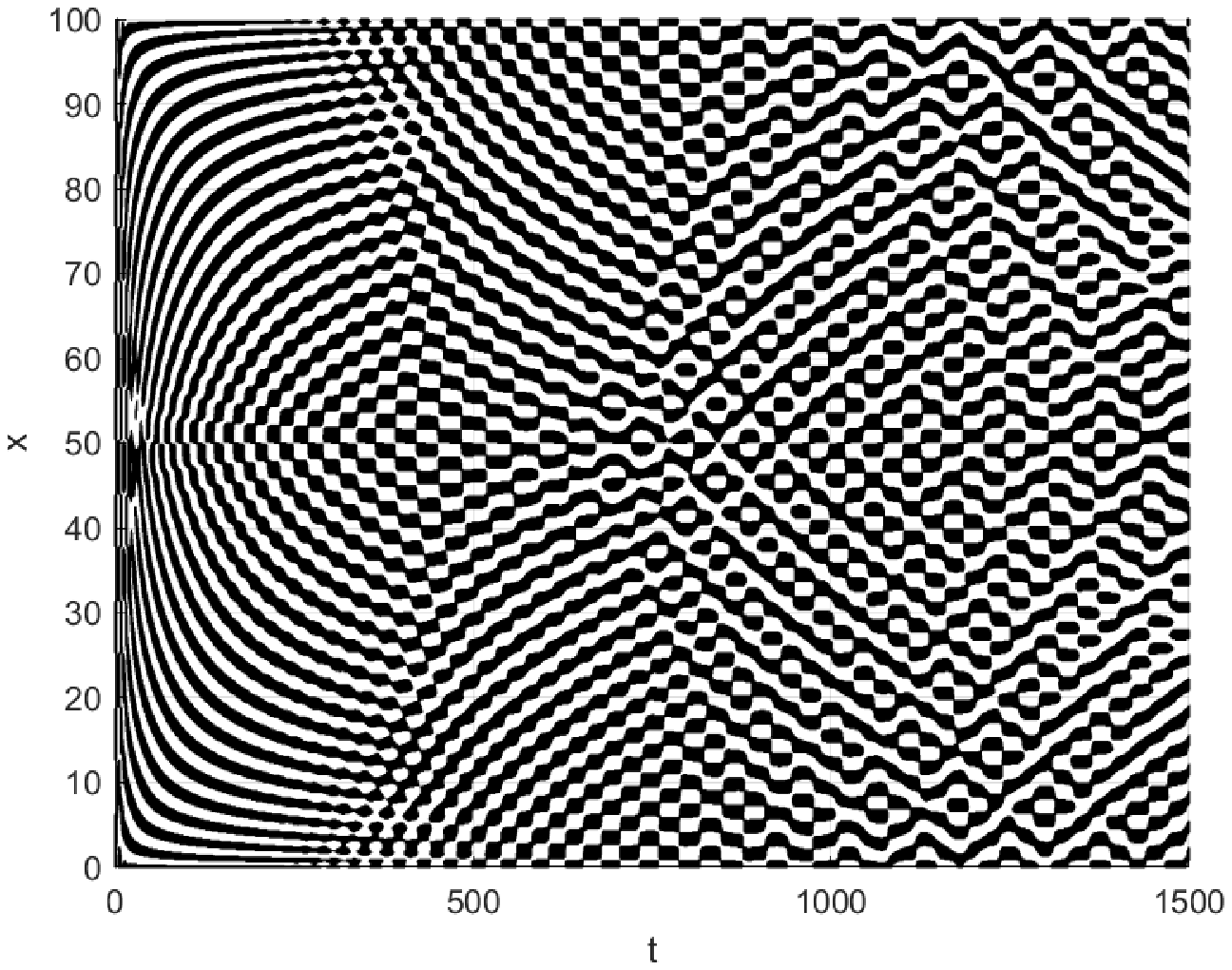,width=5.3cm} \par {(a) }
\end{minipage}
\begin{minipage}{0.47\textwidth}\centering
\epsfig{figure=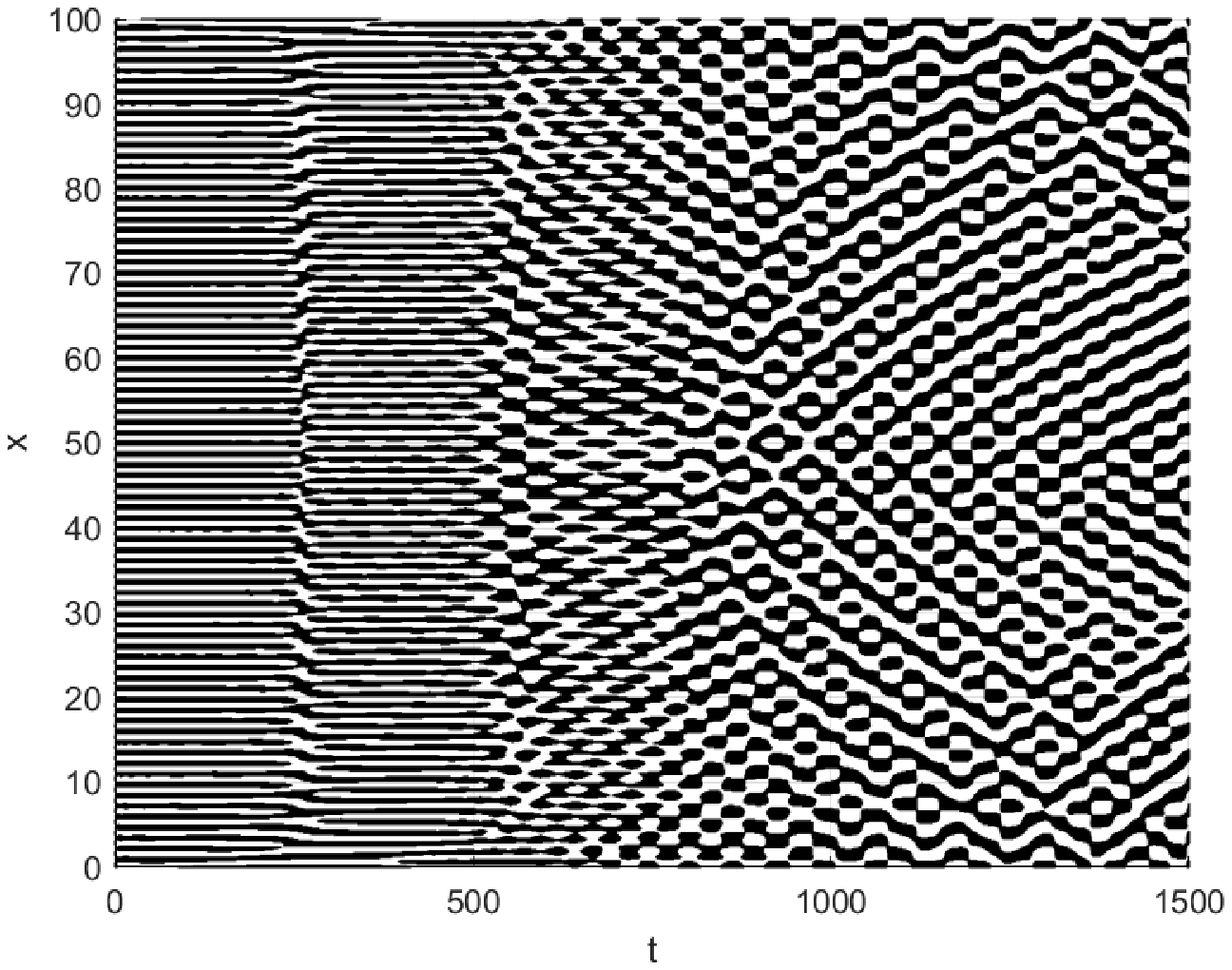,width=5.cm} \par {(a) }
\end{minipage}
\begin{minipage}{0.47\textwidth}\centering
\epsfig{figure=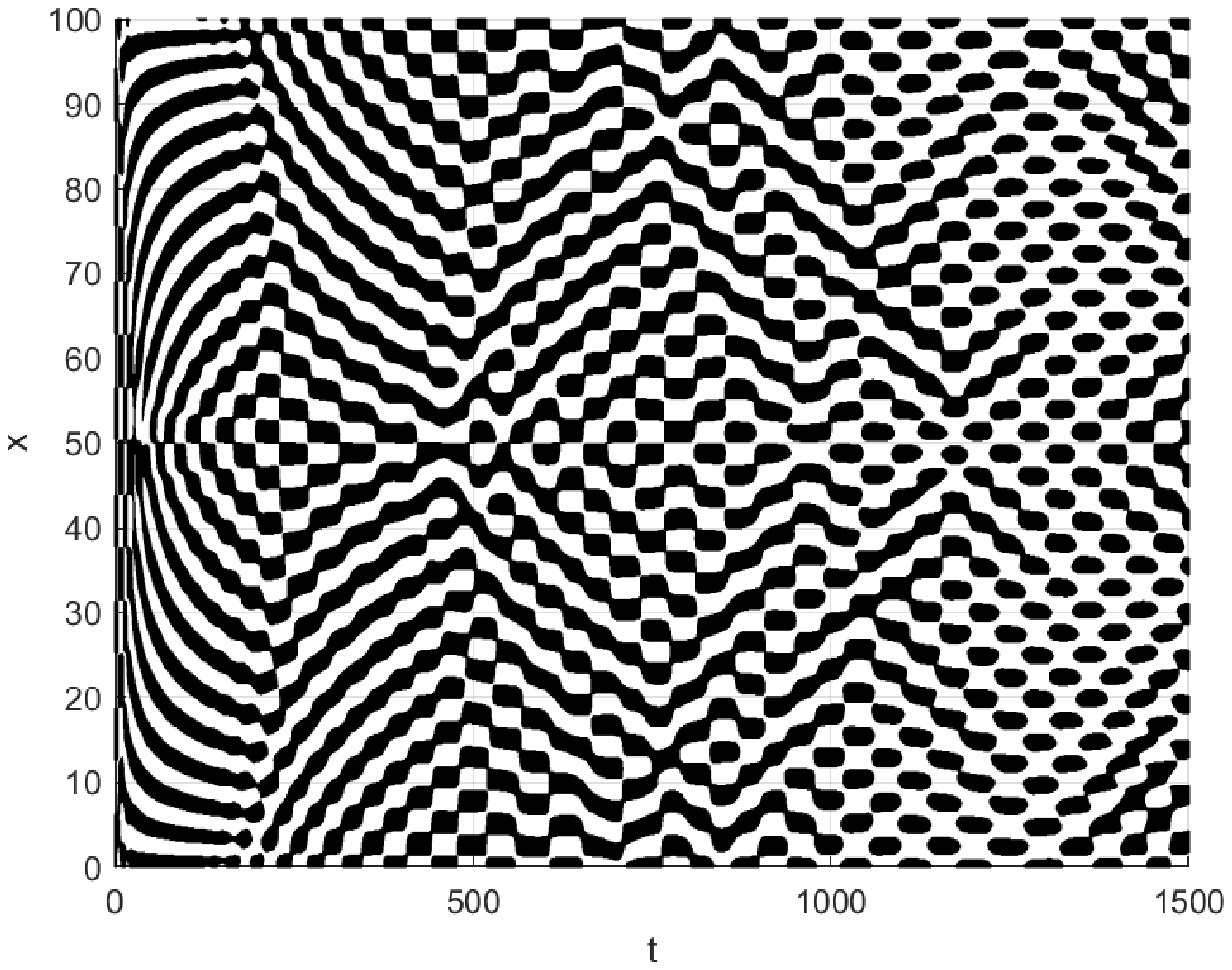,width=5cm} \par {(b) }
\end{minipage}
\begin{minipage}{0.47\textwidth}\centering
\epsfig{figure=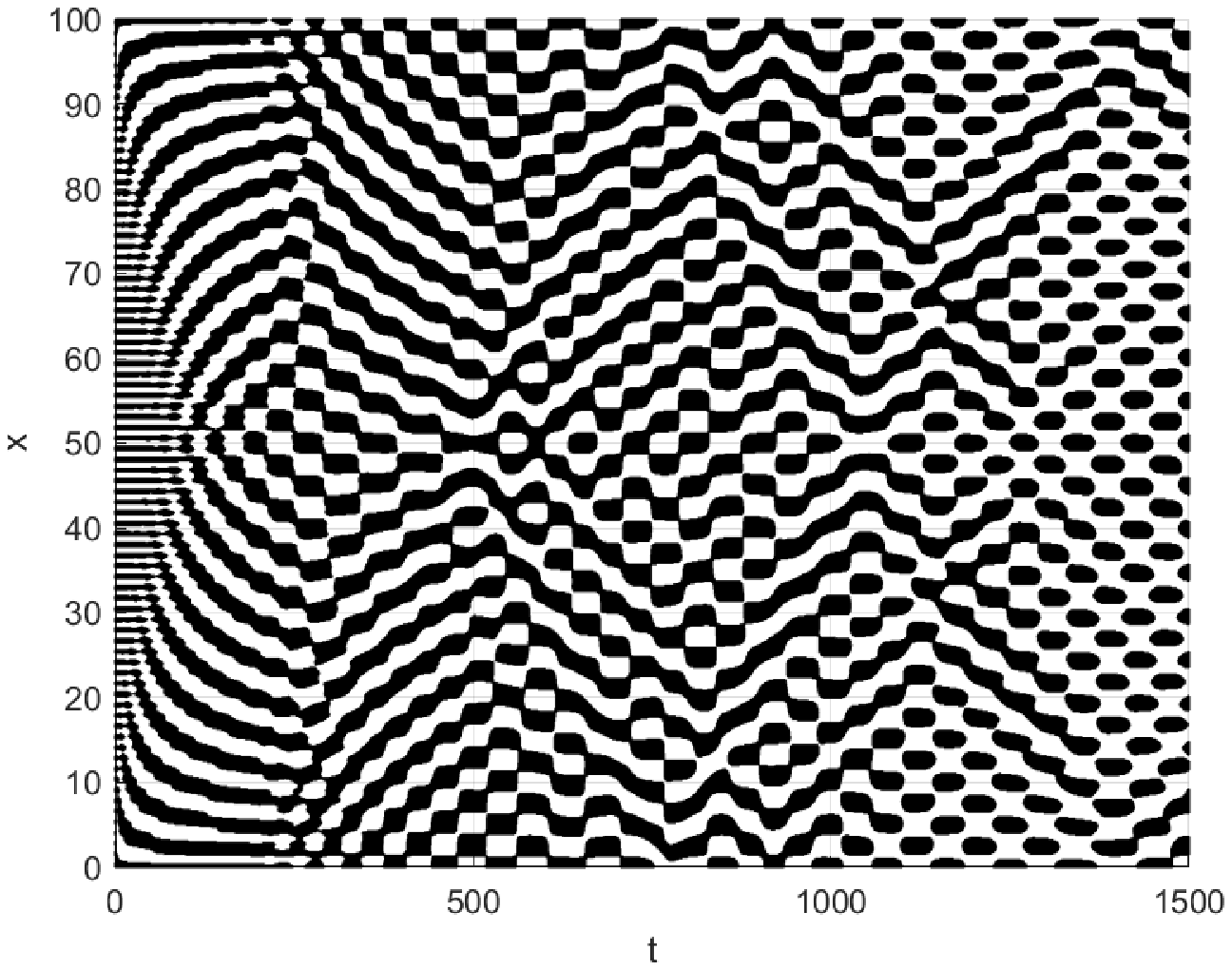,width=5cm} \par {(b) }
\end{minipage}
\begin{minipage}{0.47\textwidth}\centering
\epsfig{figure=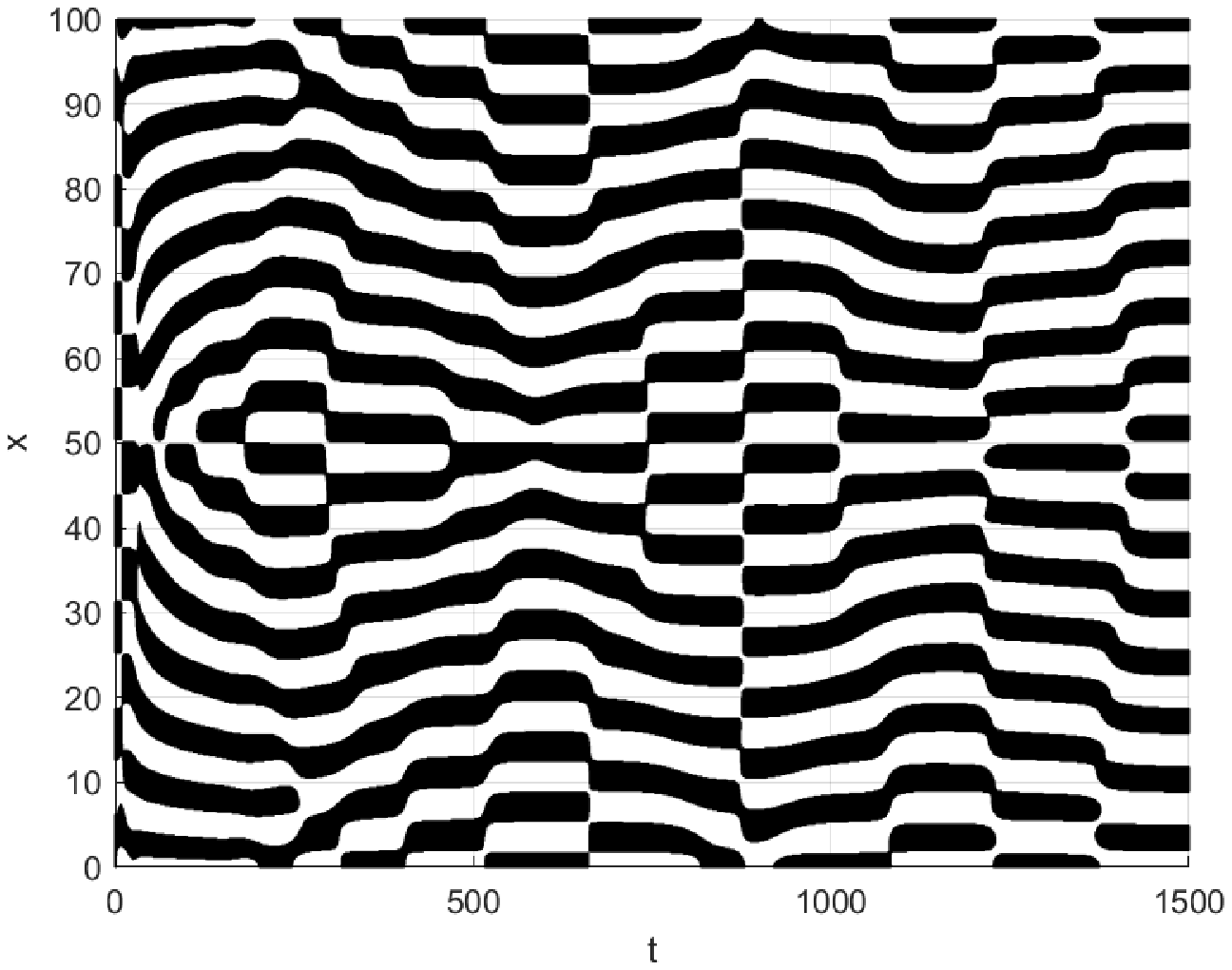,width=5cm} \par {(c) }
\end{minipage}
\begin{minipage}{0.47\textwidth}\centering
\epsfig{figure=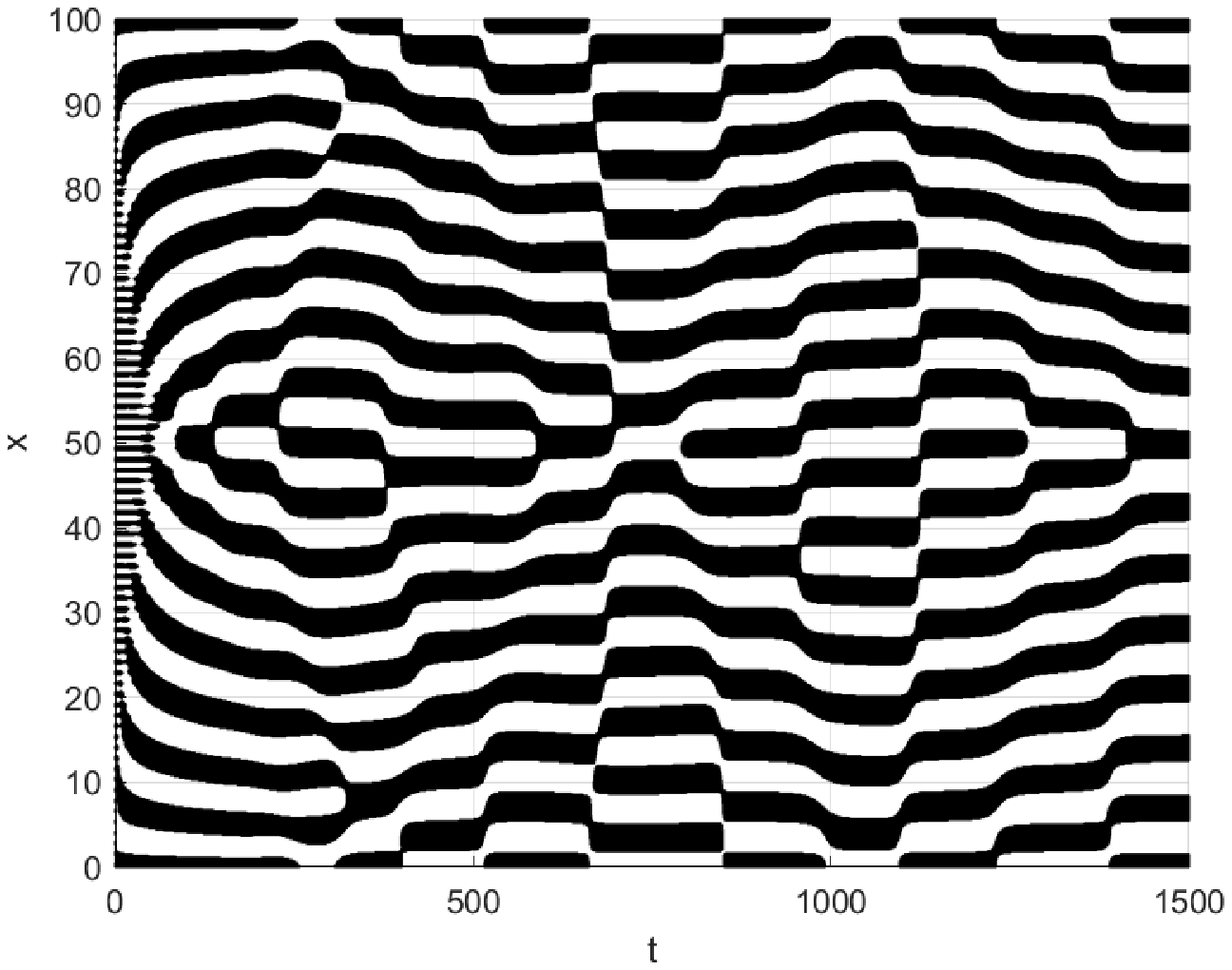,width=5cm} \par {(c) }
\end{minipage}
\end{center}
\caption{Surface density plots of fractional Brusselator model for $u(x,t)$
with long-wavelength (left column) and
short-wavelength (right column) sinusoidally perturbed initial conditions (ii),
Example \ref{s5-eg-2}:
(a) $\alpha_1=\alpha_2=0.2$, $d=9$;
(b) $\alpha_1=\alpha_2=0.5$, $d=17$; (c) $\alpha_1=\alpha_2=0.8$, $d=23$.\label{turing_pattern3}}
\end{figure}


Next, we choose different fractional orders $\alpha_1=0.5$ (anomalous subdiffusion in
the activator $u(x,t)$) and $\alpha_2=1$ (standard diffusion in the inhibitor $v(x,t)$).
In such a  case, turning-instability-induced pattern formation might occur for any
$d>0$ (see \cite{Henry05,Henry02}). We perform a number of numerical simulations of
the fractional activator-inhibitor model with anomalous diffusion in the activator
and standard diffusion in the inhibitor over a range of parameters. Sample
results are shown in Figures \ref{turing_pattern5}--\ref{turing_pattern6}.
Figure \ref{turing_pattern5} shows the surface profiles of the activator and inhibitor for
the fractional Gierer--Meinhardt model with randomly perturbed initial conditions (i).
Figure \ref{turing_pattern6} shows  the surface profiles of the activator and inhibitor of
the fractional Brusselator model with short-wavelength sinusoidally perturbed initial conditions (iii).
Obviously, for both models,  the activator and inhibitor display different fluctuations
about the homogenous steady-state solution. We obtain similar results as those obtained in
\cite{Henry05}.

\begin{figure}[!h]
\begin{center}
\begin{minipage}{0.47\textwidth}\centering
\epsfig{figure=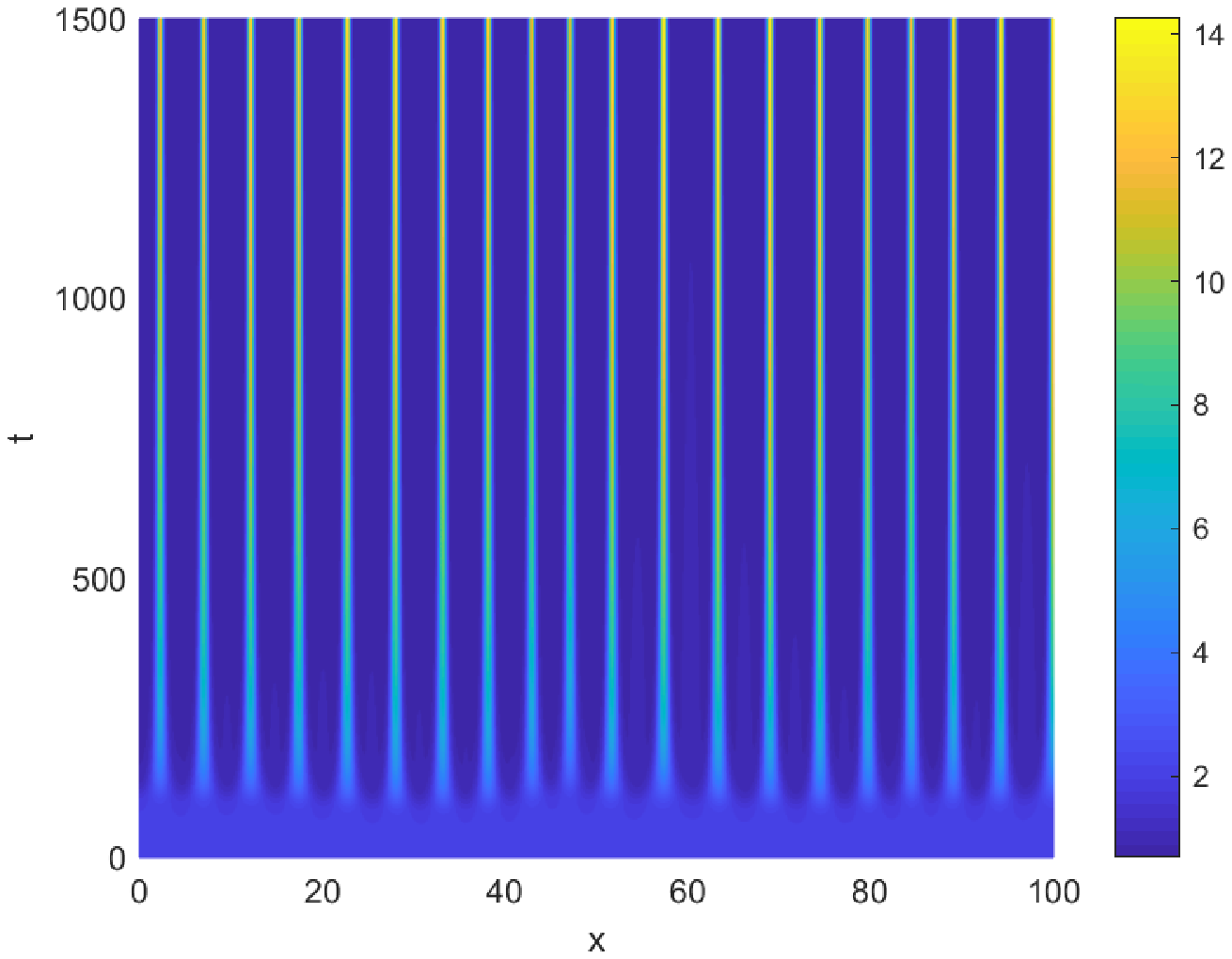,width=5cm} \par {(a) Surface  profile of  $u(x,t)$.}
\end{minipage}
\begin{minipage}{0.47\textwidth}\centering
\epsfig{figure=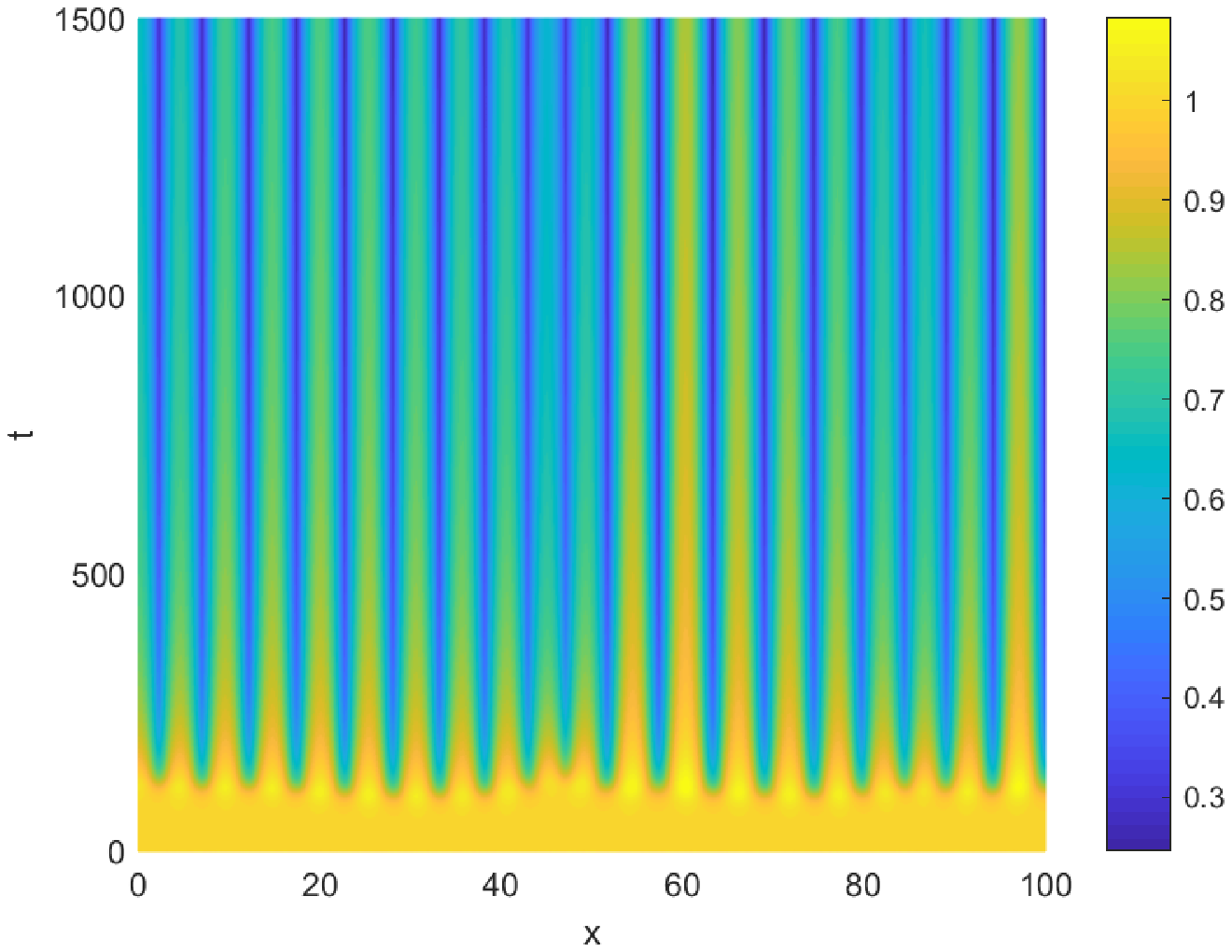,width=5cm} \par {(b) Surface  profile of  $v(x,t)$.}
\end{minipage}
\end{center}
\caption{Fractional Gierer--Meinhardt model with randomly perturbed initial conditions (i), Example \ref{s5-eg-2}: $\alpha_1=0.5$, $\alpha_2=1.0$, $d=8$.\label{turing_pattern5}}
\end{figure}

\begin{figure}[!h]
\begin{center}
\begin{minipage}{0.47\textwidth}\centering
\epsfig{figure=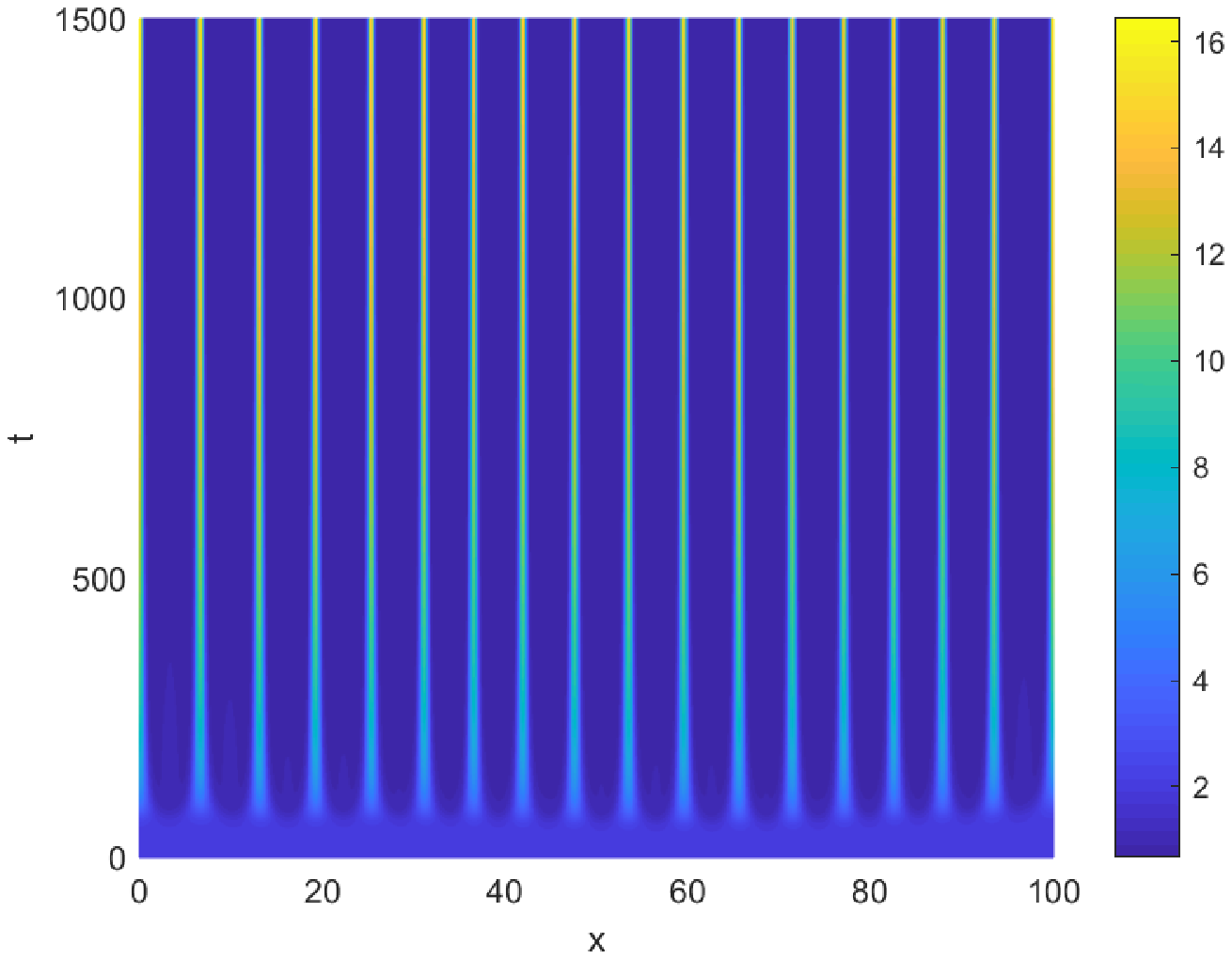,width=5cm} \par {(a) Surface  profile of  $u(x,t)$.}
\end{minipage}
\begin{minipage}{0.47\textwidth}\centering
\epsfig{figure=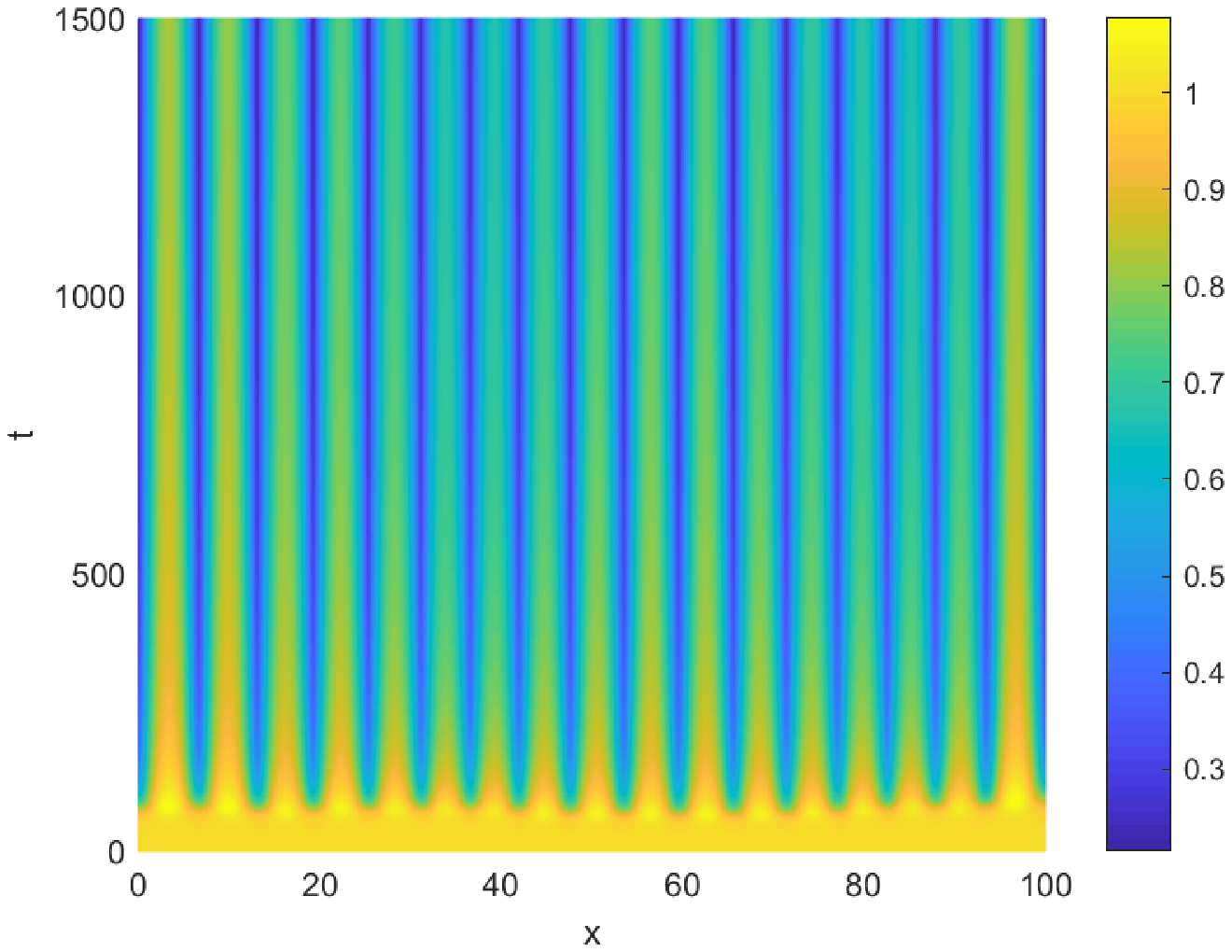,width=5cm} \par {(b) Surface  profile of  $v(x,t)$.}
\end{minipage}
\end{center}
\caption{Fractional Brusselator model with short-wavelength sinusoidally perturbed initial conditions (iii),
Example \ref{s5-eg-2}: $\alpha_1=0.5$, $\alpha_2=1.0$, $d=10$.\label{turing_pattern6}}
\end{figure}

Finally in this section, we show the efficiency and accuracy of the fast method.
Figure \ref{cputime8} (a) displays the difference $\|{}_Fu_{h}(t)-{}_Du_{h}(t)\|_{\infty}$
of the fast solution and the direct solution of the fractional Brusselator model
with long-wavelength sinusoidally perturbed initial conditions (ii),
where ${}_Fu_{h}$ is the fast solution obtained from   \eqref{sec5:eq-21-1}--\eqref{sec5:eq-22-1},
${}_Du_{h}$ is the direct method solution that is obtained from \eqref{sec5:eq-21-1}--\eqref{sec5:eq-22-1}
with ${}_FD_{\tau}^{1-\alpha,0,0,0,n}$ replaced by the direct calculation
method $D_{\tau}^{1-\alpha,0,0,0,n}$. We choose $256$ quadrature points in the trapezoidal rule
used in the fast method,  and an accuracy of   $10^{-9}$ is achieved
(more accurate results can be obtained if we increase the number of quadrature points $Q$,
but these results are not shown here).
The most obvious observation is that the computational time of the fast  method
increases linearly, while the computational cost of the direct method
increases quadratically; see Figure \ref{cputime8} (b).
For the case shown in    Figure \ref{cputime8} (b), the computational times
of the fast method and direct method are about 4000 seconds (about one hour)
and 87000 seconds (about one day and two hours), respectively.

\begin{figure}[!h]
\begin{center}
\begin{minipage}{0.47\textwidth}\centering
\epsfig{figure=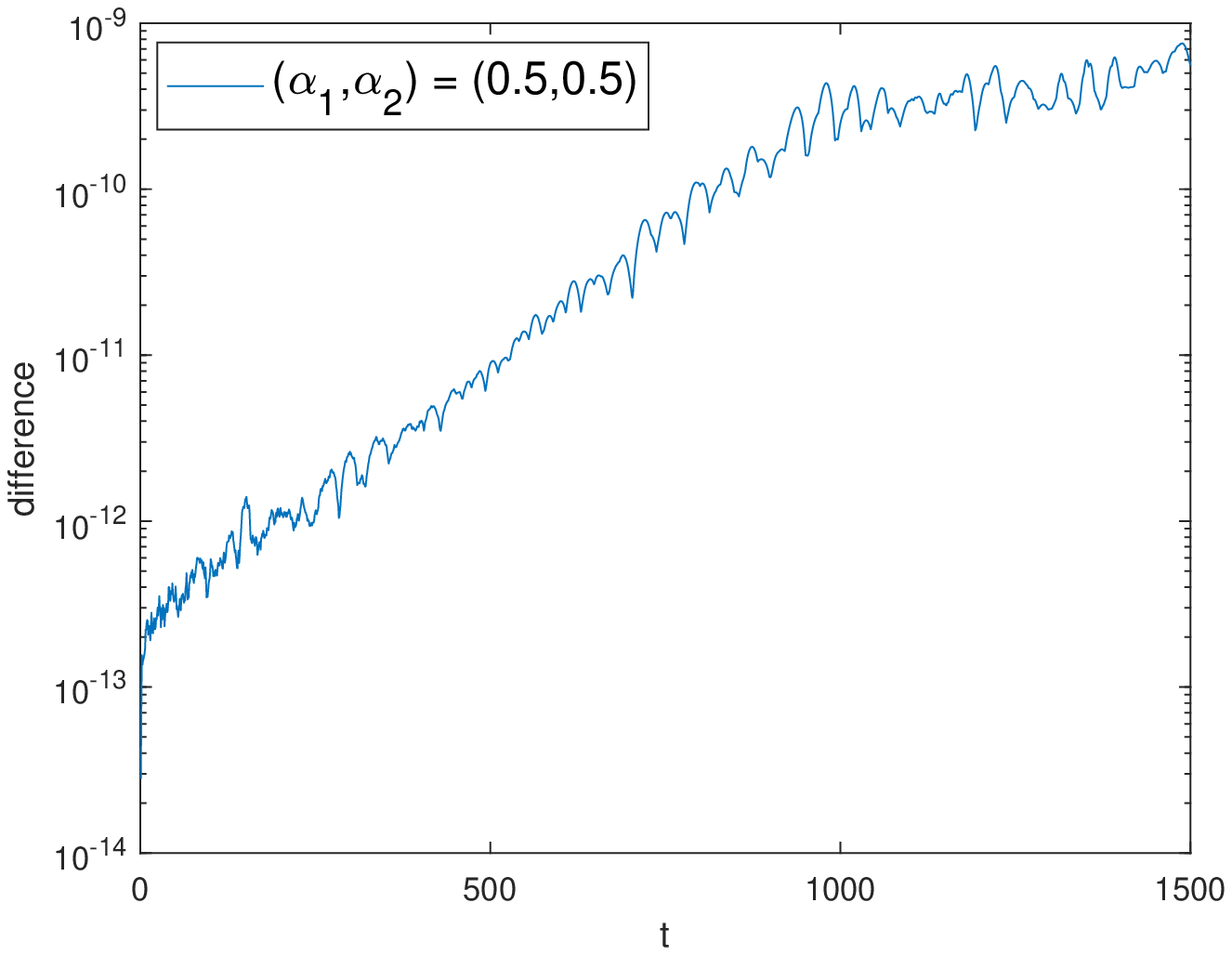,width=5cm}
\par {(a)  $\|{}_Fu_{h}(t)-{}_Du_{h}(t)\|_{\infty}$.}
\end{minipage}
\begin{minipage}{0.47\textwidth}\centering
\epsfig{figure=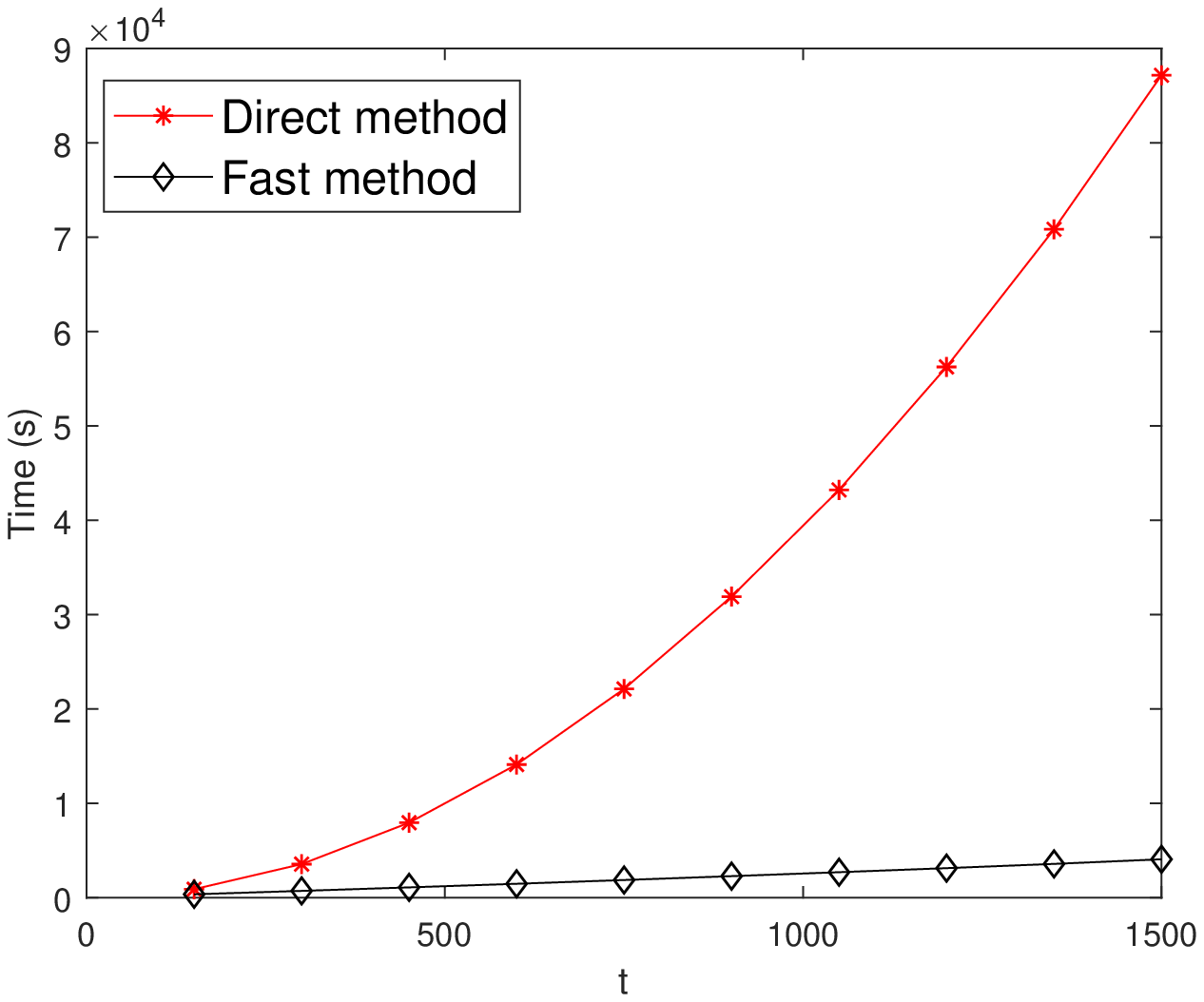,width=5cm} \par {(b) Computational time.}
\end{minipage}
\end{center}
\caption{(a) The maximum difference $\|{}_Fu_{h}(t)-{}_Du_{h}(t)\|_{\infty}$ between the direct solution
and fast solution of the fractional Brusselator model
with long-wavelength sinusoidally perturbed initial conditions (ii);
(b) the computational time of the fast method and the direct method;
Example \ref{s5-eg-2}, $\alpha_1=\alpha_2=0.5$, $d=17$.\label{cputime8}}
\end{figure}

\section{Conclusion and discussion}\label{concl}
In this work, we first prove the equivalence between the tempered fractional derivative
operator and the Hadamard finite-part integral.
The interpretation of the tempered fractional derivative in terms of the finite-part integral
makes a direct and obvious extension of Lubich's FLMMs to both the tempered fractional integral and derivative operators, which greatly simplifies the method in  \cite{ChenDeng15}.

We then propose two fast methods, Fast Method I and Fast Method II, to
approximate the discrete convolution $\sum_{j=0}^n\omega^{(\alpha,\sigma)}_{n-j}u_j$
in the considered FLMM. Both   methods are effective and efficient.
Fast Method I can be seen as a direct extension of the
fast method in \cite{ZengTBK2018} ($\sigma=0$) to the tempered fractional operator $(\sigma>0)$.
In Fast Method I, the convolution weight $\omega^{(\alpha,\sigma)}_{n}$ is represented by a contour
integral, which is approximated by a local contour quadrature for different $n$.
The use of the local approximation for approximating $\omega^{(\alpha,\sigma)}_{n}$  makes
the implementation of Fast Method I a little complicated.
Furthermore,   complex arithmetic operations are performed in Fast Method I,
which leads to slightly larger roundoff errors, see Figure \ref{eg51-relativeerror} (a).

In order to overcome the drawbacks of Fast Method I, we propose Fast Method II,
which has the following advantages.
\begin{itemize}[leftmargin=*]
  \item In Fast Method II, the convolution weight $\omega^{(\alpha,\sigma)}_{n}$ is expressed
  by an integral on the real line  instead of the contour integral in the complex plane in Fast Method I. A uniform approximation is derived to approximate this integral on the real line, which makes the implementation of Fast Method II much easier and simpler than that of Fast Method I.
  \item Only real arithmetic operations are performed in Fast Method II.
  \item In Fast Method I, an ODE of the form $y'(t)=\lambda y(t) + u(t)$ is solved by the
  backward Euler method (see also \eqref{ode2}).
  However, the coefficient $\lambda$ may have positive real part if the Talbot or hyperbolic
  contour quadrature (see, e.g., \cite{SchLopLub06,ZengTBK2018}) is applied, which may affect the stability of Fast Method I.
  We always perform a stable recurrence relation \eqref{real-int-7} in Fast Method II, which
  avoids a possible negative effect caused by the positive real part of $\lambda$ in Fast Method I.
\end{itemize}

In summary, Fast Method II outperforms Fast Method I in terms of both accuracy and efficiency, and
yields easier implementation,
which is also verified by numerical simulations in this work. Fast Method I still works well, but
the most obvious disadvantage is its complicated implementation.
The
code for numerical simulations in this paper can be found at
https://github.com/fanhaizeng/fast-method-for-fractional-operators-generating-functions.

%

\def\cprime{$'$} \def\cprime{$'$} \def\cprime{$'$} \def\cprime{$'$}

\end{document}